\documentclass[reqno,12pt]{amsart}

\usepackage[colorlinks=true, linkcolor=blue, citecolor=blue]{hyperref}

\usepackage{amssymb}
\usepackage{amsmath, graphicx, rotating}
\usepackage{color}
\usepackage{soul}
\usepackage[dvipsnames]{xcolor}

\usepackage{ifthen}
\usepackage{xkeyval}
\usepackage{todonotes}
\usepackage[T1]{fontenc}
\usepackage{lmodern}
\usepackage[english]{babel}

\usepackage{ upgreek }
\usepackage{stmaryrd}
\SetSymbolFont{stmry}{bold}{U}{stmry}{m}{n}
\usepackage{amsthm}
\usepackage{float}

\usepackage{ bbm }
\usepackage{ stmaryrd }
\usepackage{ mathrsfs }
\usepackage{ frcursive }
\usepackage{ comment }

\usepackage{pgf, tikz}
\usetikzlibrary{shapes}
\usepackage{varioref}
\usepackage{enumitem}

\definecolor{rouge}{rgb}{0.7,0.00,0.00}
\definecolor{vert}{rgb}{0.00,0.5,0.00}
\definecolor{bleu}{rgb}{0.00,0.00,0.8}
\usepackage[margin=1in]{geometry}
\newtheorem{theorem}{Theorem}[section]
\newtheorem*{theorem*}{Theorem}
\newtheorem{lemma}[theorem]{Lemma}

\newtheorem{proposition}[theorem]{Proposition}

\labelformat{hypothesis}{\textbf{M\kern-0.1mm#1}}

\newtheorem{condition}{Condition}

\newtheorem{conditionA}{A\kern-0.1mm}
\labelformat{conditionA}{\textbf{A\kern-0.1mm#1}}

\newtheorem{conditionB}{B\kern-0.1mm}
\labelformat{conditionB}{\textbf{B\kern-0.1mm#1}}

\theoremstyle{definition}
\newtheorem{example}[theorem]{Example}
\newtheorem{remark}[theorem]{Remark}

\def \eref#1{\hbox{(\ref{#1})}}

\numberwithin{equation}{section}

\def\geq{\geqslant}
\def\leq{\leqslant}

\def\RR{\mathbb{R}}
\def\PP{\mathbb{P}}
\def\EE{\mathbb{E}}

\def\vare{{\varepsilon}}
\def \eref#1{\hbox{(\ref{#1})}}

\def\EE{\mathbb{ E}}

\allowdisplaybreaks

\begin{document}
	
	\title[Asymptotic behavior for multi-scale SDEs driven by L\'evy processes]
	{Asymptotic behavior for multi-scale SDEs with monotonicity coefficients driven by L\'evy processes }
	
	\author{Yinghui Shi}
	\curraddr[Shi, Y.]
	{School of Mathematics and Statistics, Jiangsu Normal University, Xuzhou, 221116, China}
	\email{shiyinghui@jsnu.edu.cn}
	
	\author{Xiaobin Sun}
	\curraddr[Sun, X.]{ School of Mathematics and Statistics/RIMS, Jiangsu Normal University, Xuzhou, 221116, China}
	\email{xbsun@jsnu.edu.cn}
	
	\author{Liqiong Wang}
	\curraddr[Wang, L.]{School of Science, Key Laboratory of Mathematics and Information Networks, Ministry of Education, Beijing University of Posts and Telecommunications, Beijing,
100876, China}
	\email{wlq@jsnu.edu.cn}
	
	\author{Yingchao Xie}
	\curraddr[Xie, Y.]{ School of Mathematics and Statistics/RIMS, Jiangsu Normal University, Xuzhou, 221116, China}
	\email{ycxie@jsnu.edu.cn}

	\begin{abstract}
		In this paper, we study the asymptotic behavior for multi-scale stochastic differential equations driven by L\'evy processes.
		The optimal strong convergence order 1/2 is obtained by studying the regularity estimates for the solution of Poisson equation with polynomial growth coefficients, and the optimal weak convergence order 1 is got by using the technique of Kolmogorov equation. The main contribution is that the obtained results can be applied to a class of multi-scale stochastic differential equations with monotonicity coefficients, as well as the driven processes can be the general L\'evy processes, which seems new in the existing literature.
	\end{abstract}

	\date{\today}

	\maketitle
	
	\tableofcontents
	\section{Introduction}
	\subsection{Background}
	In this paper, we consider the following multi-scale stochastic system driven by L\'evy processes:
	\begin{equation}
		\left\{\begin{array}{l} \label{Equation}
			\displaystyle
			dX^{\varepsilon}_t=b(X^{\varepsilon}_{t},Y^{\varepsilon}_{t})dt+\sigma( X^{\varepsilon}_t,Y^{\varepsilon}_t)dW^{1}_t\!+\!\int_{\mathcal{Z}_1} h_{1} (X_{t-}^{\varepsilon} ,z)\tilde{N}^{1} (dz,dt),\\
			dY^{\varepsilon}_{t}\!=\!\frac{1}{\varepsilon}f(X^{\varepsilon}_{t},Y^{\varepsilon}_{t})dt\!+\!\frac{1}{\sqrt{\varepsilon}}g(X^{\varepsilon}_{t},Y^{\varepsilon}_{t})d W^{2}_t\!+\!\int_{\mathcal{Z}_2} h_{2} (X_{t-}^{\varepsilon} ,Y_{t-}^{\varepsilon}, z)\tilde{N} ^{2,\varepsilon} (dz,dt)\end{array}\right.
	\end{equation}
	with initial values $X^{\varepsilon}_0=x\in \RR^{n}$ and $Y^{\varepsilon}_0=y\in\RR^{m}$, where $\vare$ is a small and positive parameter describing the ratio of the time scale between the slow component $X^{\vare}_t$ and fast component $Y^{\vare}_t$.  $W^{1}$ and $W^{2}$ are independent $d_1$ and $d_2$ dimensional standard Wiener processes, $\tilde{N}^{1}$ and $\tilde{N}^{2,\vare}$ are compound Poisson random measures with L\'evy measures $\nu_1$ and $\frac{1}{\varepsilon}\nu_2$ respectively.
	
	\vspace{0.1cm}
	The averaging principle describes the asymptotic behavior of the slow component as $\vare$ goes to $0$, i.e., $X^{\vare}$ weakly converges to the solution of a limiting equation which usually satisfies the following form:
	\begin{equation}
		d\bar{X_{t} }=\bar{b}(\bar{X}_{t})dt+\bar \sigma (\bar{X}_{t})d W_{t}+\int_{\mathcal{Z}_1 }h_{1}(\bar{X }_{t-},z)\tilde{N}^{1}(dz,dt)\label{AVE}
	\end{equation}
	with initial value  $\bar{X}_0=x$, and $\bar{b}$ and $\bar{\sigma}$ are the corresponding averaged coefficients, thus \eref{AVE} is called the averaged equation. Since the pioneer works for the averaging principle for multi-scale stochastic differential equations (SDEs) by Khasminskii \cite{K1}, and multi-scale stochastic partial differential equations (SPDEs) by Cerrai and Freidlin \cite{CF2009}, there have been various results on the averaging principle for multi-scale stochastic systems, see e.g. \cite{C2011,GP1,HL2020,L2010,LRSX1,V0,WY2022}.
	
	\vspace{0.1cm}
	In addition to the convergence of $X^{\vare}$, one is also interested in the rate of convergence, since the rate can be used to construct the efficient numerical schemes, see e.g. \cite{B2020}. Moreover, the optimal rate of convergence is also known to be very important for diffusion approximation or homogenization problems and normal derivation or central limit type theorems, see e.g. \cite{HLLS2021,PV1,PV2,PS,RX2021,WR}. In the past several decades, there have been many results in studying the convergence rate in the averaging principle. A common strategy used to study is to apply the technique of classical Khasminskii's time discretization, see e.g. \cite{B2012, F2022,L2010},
	and the method of asymptotic expansion of solutions of Kolmogorov equations in the parameter $\vare$ is  used to study the  weak convergence rate, see e.g. \cite{B2012,DSXZ, FWLL,KY2004}.
	Recently, the technique of Poisson equation is widely used to study the optimal strong and weak convergence rates, see e.g. \cite{B2020,RSX2021}.
	
	\vspace{0.1cm}
	The mentioned references all considered stochastic systems driven by Gaussian noise, whose solutions have continuous paths. However, solutions with discontinuous paths appears naturally in many applications. The averaging principle for stochastic system driven by jump noise has been researched by many scholars, see e.g. \cite{BYY,GLSX,G2007,PXW,SWXY2022, SG2022,XML1,XDX}. However, only few studies have shown the optimal convergence rate in the strong and weak sense. For example, Liu \cite{L2012} obtained the optimal convergence order $1/2$ in the strong sense for a class of  jump-diffusion processes. Sun et.al. \cite{SXX} obtained the optimal strong convergence order $1-1/\alpha$ and weak convergence order $1$ for a class of SDEs driven by $\alpha$-stable processes, where $\alpha\in(1,2)$. Sun and Xie \cite{SX2023} obtained the optimal strong convergence order $1-1/\alpha$ and weak convergence order $1-r$ (for any $r\in(0,1)$) for a class of SPDEs driven by $\alpha$-stable processes, where $\alpha\in(1,2)$.
	
	\vspace{0.1cm}
	The main purpose of this paper focuses on studying the optimal strong and weak convergence rates for stochastic system \eref{Equation}. More precisely, for any initial value $(x,y)\in \RR^n\times\RR^m$, $T>0$, $p>1$ and small enough $\vare>0$, under some proper conditions on the coefficients and $\sigma(x,y)\equiv\sigma(x)$, we have
	\begin{eqnarray}
		\mathbb{E} \left(\sup_{0\leq t\leq T}|X_{t}^{\vare}-\bar{X}_{t}|^{p}\right)\leq C\vare^{p/2},\label{Intro(1)}
	\end{eqnarray}
	where $C$ is a constant depending on $T, |x|, |y|, p$ and $\bar{X}$ is the solution of the corresponding averaged equation  (see \eref{AR1} below).
	Furthermore, for some proper function $\phi$, we obtain
	\begin{eqnarray}
		\sup_{0\leq t\leq T}|\mathbb{E}\phi(X^{\vare}_t)-\EE\phi(\bar{X}_t)|\leq C\vare,\label{Intro(2)}
	\end{eqnarray}
	where $\bar{X}$ is the solution of another  averaged equation  (see \eref{AR2} below). This means that the strong and weak convergence orders are $1/2$ and $1$ respectively, which are the optimal convergence orders.
	
	\vspace{0.1cm}
	It is worth  pointing out an independent interest of the paper is that the coefficient $b(x,y)$ satisfies the monotonicity condition and polynomial growth with respect to $x$ and $y$ respectively, such as the form $b(x,y)=-|x|^2 x +x+|y|^2 y$ (see Example 2.7 below for detailed discussion). As far as we know, there seems to be few results on this topic, even in the case of the Wiener noises. Compare with \cite{LRSX1}, where Liu. et. al. have proved the strong averaging principle holds without the convergence order for a class of SDEs with locally Lipschitz coefficients, we here achieve the optimal strong convergence order 1/2. Compare with \cite{CDGOS2022}, where Crisan et. al. study the optimal weak convergence order $1$ in the sense of \emph{uniform in time}  for multi-scale SDEs with locally-Lipschitz coefficients, we here obtain the optimal strong convergence order $1/2$ and the driving noises can be the general L\'{e}vy processes. To the best of our knowledge, it seems the first result about the optimal strong convergence rate for SDEs with monotonicity coefficients.
	
	\subsection{Main techniques for the proofs} For reading convenience, we are in a position to show the main techniques of the proofs.
	
	\vspace{0.1cm}
	
	\emph{Techniques for strong convergence:} Firstly, using the monotonicity condition, one can get
	\begin{eqnarray*}
		\mathbb{E}\left(\sup_{0\leq t\leq T}|X^{\varepsilon}_t-\bar{X}_t|^p\right)
		\leq\!\!C_{p,T}\mathbb{E}\left[\sup_{0\leq t\leq T}\left|\int^t_0\langle X_{s}^{\varepsilon}-\bar{X}_{s}, b(X^{\varepsilon}_{s},Y^{\varepsilon}_{s})-\bar{b}(X^{\varepsilon}_{s})\rangle ds\right|^{p/2}\right].
	\end{eqnarray*}
	In order to prove \eref{Intro(1)}, the key step is how to deal with the term $b(X^{\vare}_s,Y^{\vare}_s)-\bar{b}(X^{\vare}_s)$. To do this, we will use the technique of Poisson equation. Roughly speaking, $b(X^{\vare}_s,Y^{\vare}_s)-\bar{b}(X^{\vare}_s)$ can be replaced by $-\mathscr{L}_2(X^{\vare}_s)\Phi(X^{\vare}_s,\cdot)(Y^{\vare}_s)$, i.e., considering the following Poisson equation:
	$$
	-\mathscr{L}_{2}(x)\Phi(x,\cdot)(y)= b(x,y)-\bar{b}(x),
	$$
	where $\mathscr{L}_{2}(x)$ is the generator of the frozen equation  (\ref{Frozen equation}) for fixed $x\in\RR^n$. Then applying It\^{o}'s formula on $\langle\Phi(X^{\vare}_t,Y^{\vare}_t), X_{t}^{\varepsilon}-\bar{X}_{t}\rangle $, one would obtain the new term
	$$
	\int^t_0\langle X_{s}^{\varepsilon}-\bar{X}_{s}, \mathscr{L}_2(X^{\vare}_s)\Phi(X^{\vare}_s,\cdot)(Y^{\vare}_s)\rangle ds
	$$
	has an expression in terms of the solution $\Phi$ to the Poisson equation (see \eref{S5.8} below), thus the remaining works are devoted to studying the regularity estimates of the solution $\Phi$. However, for the purpose of covering some cases of the polynomial growth coefficients,  this paper is devoted to studying the regularity estimates for the solution of the corresponding Poisson equation with polynomial growth coefficients.
	
	\vspace{0.1cm}
	\emph{Techniques for weak convergence: } We shall use the technique based on a combination of Kolmogorov equation and Poisson equation to prove \eref{Intro(2)}. Roughly speaking, we first introduce the Kolmogorov equation:
	\begin{equation*}\left\{\begin{array}{l}
			\displaystyle
			\partial_t u(t,x)=\bar{\mathscr{L}}_1 u(t,x),\quad t\geq 0, \\
			u(0, x)=\phi(x),
		\end{array}\right.
	\end{equation*}
	where $\bar{\mathscr{L}}_1$ is the generator of the transition semigroup of the averaged equation \eref{AR2}. For fixed $t>0$, denote
	$$\tilde{u}^t(s,x):=u(t-s,x),\quad s\in [0,t].$$
	Then we observe that
	$$\tilde{u}^t(t, X^{\vare}_t)=\phi(X^{\vare}_t),\quad \tilde{u}^t(0, x)=\EE\phi(\bar{X}^{x}_t).$$
	Therefore, using It\^{o}'s formula on $\tilde{u}^t(t, X^{\vare}_t)$ and taking expectation on both sides (see \eref{F5.11} below), one would  get
	\begin{eqnarray}
		\EE\phi(X^{\vare}_{t})-\EE\phi(\bar{X}_{t})=\EE\int^t_0 \left[F^t(s,x,y)-\bar{F}^t(s,x)\right] ds, \label{Intro(3)}
	\end{eqnarray}
	where
	\begin{eqnarray*}
		F^t(s,x,y):=\langle b(x,y), \partial_x \tilde{u}^t(s, x)\rangle+\frac{1}{2}\text{Tr}\big[\sigma\sigma^{*}(x,y)\partial_x^{2}\tilde{u}^t(s,x)\big]
	\end{eqnarray*}
	and $\bar{F}^t(s,x)$ is the corresponding averaged coefficient. In order to prove \eref{Intro(3)}, we consider another Poisson equation:
	$$
	-\mathscr{L}_{2}(x)\tilde{\Phi}^t(s,x,\cdot)(y)=F^t(s,x,y)-\bar{F}^t(s,x).
	$$
	The remaining proof returns to the previous procedure for applying the technique of Poisson equation.
	
	\subsection{Organization} The rest of the paper is organized as follows. In Section 2, we give some notations and under some suitable assumptions, we formulate our main results and give several concrete examples to illustrate the applicability of our main results. Section 3 is devoted to giving some a priori estimates of the solution $(X^{\vare}_t, Y^{\vare}_t)$ and studying the frozen and averaged equations. We study the regularity estimates of the solution of Poisson equation in Section 4. The details proofs of the strong convergence and weak convergence are given in Sections 5.1 and 5.2 respectively. In last section is the appendix, where we give detailed proofs about the differentiability of solution to the frozen equation, and the well-posedness of the Kolmogorov equation for the averaged equation.
	
	\vspace{0.1cm}
	
	Throughout this paper, $C$, $C_{T}$ and $C_{p,T}$ stand for constants whose value may change from line to line, and $C_{T}$ and $C_{p,T}$ is used to emphasize  that the constant depend on $T$ and $p,T$ respectively.

	\section{Main results}\label{sec.prelim}
	
	This section is divided into three subsections, we first introduce some notations and assumptions in subsection 2.1. The main results are presented in the subsection 2.2. Finally, we show some examples to illustrate our main results.
	
	\subsection{Notations and assumptions}
	Denote by $|\cdot|$ and  $\langle\cdot, \cdot\rangle$ the Euclidean vector norm and the usual Euclidean inner product, respectively. Let $\|\cdot\|$ be the matrix norm or the operator norm if there is no confusion possible.
	
	For a vector-valued or matrix-valued function $\varphi(x)$ defined on $\RR^n$ or $\varphi(x,y)$ defined on $\RR^n\times\RR^m$, for any $i,j\in \mathbb{N}$, we use $D^i\varphi(x)$ to denote its $i$-th order derivatives of $\varphi$, and use $\partial^j_y \partial^i_x\varphi(x,y)$ to denote its $i$-th and $j$-th  partial derivative of $\varphi(x,y)$ with respect to (w.r.t.)  $x$ and $y$, respectively. We  call $\varphi(x,y)$ is a function of polynomial growth if there exist $C>0$ and $k\geq 1$ such that $|\varphi(x,y)|\leq C(1+|x|^k+|y|^k), \forall x\in \RR^n, y\in \RR^m $.
	
	For any $k,k_1,k_2\in \mathbb{N}_{+}$ ,
	\begin{itemize}
		\item{ $C^k(\RR^n):=\big\{\varphi: \RR^n\to \RR: \mbox{ for } 0\le i\le k,  D^i\varphi(x) \mbox{ are continuous}\big\}$ ;}
		
		\item{ $C^k_b(\RR^n):=\big\{\varphi\in C^k(\RR^n): \mbox{ for } 1\le i\le k,D^i\varphi(x) \mbox{ are bounded}\big \}$;}
		
		\item{ $C^k_p(\RR^n):=\big\{\varphi\in C^k(\RR^n): \mbox{ for } 0\le i\le k, D^i\varphi(x) \mbox{ are polynomial growth}\big \}$;}
		
		\item{$C^{k_1,k_2}(\RR^n\times\RR^m):=\big\{\varphi: \RR^n\times\RR^m\to \RR: \mbox{ for } 0\le i\le k_1,0\le j\le k_2, \\ 0\le i+j\le k_1\vee k_2,\quad\partial^j_y \partial^i_x\varphi(x,y) \mbox{ are joint continuous}\big\}$;}
		
		\item{$C^{k_1,k_2}_p(\RR^n\times\RR^m):=\big\{\varphi\in C^{k_1,k_2}(\RR^n\times\RR^m): \mbox{ for } 0\le i\le k_1,0\le j\le k_2, \\ 0\le i+j\le k_1\vee k_2,
			\quad \partial^j_y \partial^i_x\varphi(x,y) \mbox{ are polynomial growth}\big \}$.}

	\end{itemize}
	Similarly, for a function $\varphi:\RR^n\times \RR^m\rightarrow \RR^n$, we say $\varphi\in C^{k_1,k_2}(\RR^n\times \RR^m, \RR^n)$, if all the components of $\varphi_i\in C^{k_1,k_2}(\RR^n\times \RR^m)$, $i=1,2,\ldots,n$. Other notations can be interpreted similarly.
	
	\vspace{0.1cm}
	Let $W^{1}$ and $W^{2}$ be $d_1$ and $d_2$-dimensional standard Wiener process on a complete probability space $(\Omega,\mathcal{F},\PP)$ with filtration $\{\mathcal{F}_t\}_{t\geq 0}$, respectively. Let $D_{1}$ and $D_{2}$ be two countable subsets of $\RR_{+}$. $p^1_t$ and $p^{2,\varepsilon}_t$ are two stationary $\mathscr{F}_{t}$-adapted Poisson point processes on measurable spaces $(\mathcal{Z}_1,\mathscr{B}(\mathcal{Z}_1))$ and $(\mathcal{Z}_2,\mathscr{B}(\mathcal{Z}_2))$ with characteristic measures $\nu_1$ and $\frac{1}{\varepsilon}\nu_2$ respectively, where $\int_{\mathcal{Z}_i}1\wedge |z|^2 \nu_i(dz)<\infty$, $i=1,2$. Define for any $t>0$ and $A_i\in\mathscr{B}(\mathcal{Z}_i)$, $i=1,2$,
	$$
	N^1([0,t], A_1):=\sum_{s\in D_{1},s\leq t}1_{A_1}(p^1_s)\quad \text{and}\quad N^{2,\vare}([0,t], A_2):=\sum_{s\in D_{2},s\leq t}1_{A_2}(p^{2,\varepsilon}_s),
	$$
	which are two Poisson random measures with corresponding compensated  martingale measures
	$$\tilde{N}^1(ds, du) :=N^1(ds, du)-\nu_1(du)ds$$
	and $$\tilde{N}^{2,\vare}(ds, du):= N^{2,\vare}(ds, du)-\frac{1}{\vare}\nu_2(du)ds.$$
	Note that we always assume that $W^1$, $W^{2}$, $\tilde{N}^1$ and $\tilde{N}^{2,\vare}$ are mutually independent.
	
	\vspace{0.1cm}
	Let the maps $b=b(x,y)$, $\sigma=\sigma(x,y)$, $h_1=h_1(x,z)$, $f=f(x,y)$, $g=g(x,y)$ and $h_2=h_2(x,y,z)$ be given:	
	\begin{eqnarray*}
		&& b:\RR^{n}\times \RR^{m} \longrightarrow \RR^{n}, \quad \sigma:\RR^{n}\times \RR^{m}\longrightarrow \RR^{n\times d_1},\quad h_{1}:\RR^{n}\times \mathcal{Z}_1\longrightarrow \RR^{n},\nonumber\\
		&& f:\RR^{n}\times \RR^{m} \longrightarrow \RR^{m},\quad g: \RR^{n}\times \RR^{m}\longrightarrow \RR^{m\times d_2},\quad h_{2}:\RR^{n}\times \RR^{m}\times \mathcal{Z}_2\longrightarrow \RR^{m}.
	\end{eqnarray*}
	For the coefficients of the slow equation, we suppose that for any $p\geq 2$, $x,x_1,x_2\in \RR^n$ and $y,y_1,y_2\in \RR^m$, there exist positive constants $C$, $C_p$  and $k,q\in [2,\infty)$ such that the following
	conditions hold:
	\begin{conditionA}\label{A1} (Monotonicity)
		\begin{eqnarray}
			\left\langle b(x_{1},y)-b(x_{2},y),x_{1}-x_{2}\right\rangle \leq    C|x_{1}-x_{2}|^{2}.  \label{ConA11}
		\end{eqnarray}
		Moreover,
		\begin{eqnarray}
			&&|b (x,y_1)-b (x,y_2)|\leq    C\left(1+|y_1|^k+|y_2|^k\right)|y_{1}-y_{2}|, \label{ConA112}\\
			&&\|\sigma (x_{1},y_1)-\sigma (x_{2},y_2)\|\leq  C|x_1-x_2|+C\left(1+|y_1|^k+|y_2|^k\right)|y_{1}-y_{2}|, \label{ConA113}\\
			&&\int_{\mathcal{Z}_1}|h_{1}(x_1,z)-h_{1}(x_2,z)|^{p}\nu_1(dz)\le C_p|x_1-x_2|^p. \label{ConA12}
		\end{eqnarray}
	\end{conditionA}
	
	\begin{conditionA} (Coercivity)\label{A2}
		\begin{eqnarray}
			2\left\langle b(x,y),x\right\rangle+\|\sigma(x,y)\|^2 \leq C\left(1+|x|^{2}+|y|^{q}\right). \label{ConA2}
		\end{eqnarray}
	\end{conditionA}
	
	\begin{conditionA}(Growth)\label{A3} $b\in C^{2,3}(\RR^n\times\RR^m,\RR^n)$ and for any $i=0,1,2, j=0,1,2,3$ with $0\leq i+j\leq 3$,
		\begin{eqnarray}
			\|\partial_{x}^{i}\partial_{y}^{j}b(x,y)\| \le C\left(1+|x|^{k}+|y|^{k}\right).\label{ConA31}
		\end{eqnarray}
		Moreover,
		\begin{eqnarray}
			\int_{\mathcal{Z}_1}|h_{1}(x,z)|^{p}\nu_1(dz)\le C_p\left(1+|x|^{p}\right). \label{ConA32}
		\end{eqnarray}
	\end{conditionA}
	
	For the coefficients of the fast equation, we suppose that for any $p\geq 2$, there exist positive constants $C$, $C_p$, $k$, $\ell>8$, $\beta,\lambda$, $L_{h_2}\in [0,\beta)$ and $\zeta_1,\zeta_2\in [0,1)$ such that the following conditions hold for any $x,x_1,x_2\in \RR^n$, $y,y_1,y_2\in\RR^m$:
	\begin{conditionB} (Strong monotonicity) \label{B1}
		\begin{eqnarray}
			&&2\left \langle f(x_1,y_{1})-f(x_2,y_{2}),y_{1}-y_{2}\right\rangle +(\ell-1)\| g(x_1,y_{1})-g(x_2,y_{2})\|^{2} \nonumber  \\
			&&+2^{\ell-3}(\ell-1)\!\int_{\mathcal{Z}_2}\!\!|h_{2}(x_1,y_{1},z)-h_{2}(x_2,y_{2},z)|^{2}\nu_2(dz)\le\!\! -\beta|y_1-y_2|^{2}\!+\!C|x_1-x_2|^2,
			\label{ConB1}
		\end{eqnarray}
		\begin{eqnarray}
			2^{\ell-3}(\ell-1)\int_{\mathcal{Z}_2}|h_{2}(x_1,y_{1},z)-h_{2}(x_2,y_{2},z)|^{\ell}\nu_2(dz)\le L_{h_2} |y_1-y_2|^{\ell}+C|x_1-x_2|^{\ell}.  \label{ConB11}
		\end{eqnarray}
	\end{conditionB}
	
	\begin{conditionB} (Coercivity)\label{B2}
		\begin{eqnarray}
			\left\langle f(x,y),y\right\rangle \leq -\lambda|y|^{2}-\lambda|y|^{q}+C, \label{ConB2}
		\end{eqnarray}
		where $q$ is the one in \eref{ConA2}.
	\end{conditionB}
	
	\begin{conditionB}(Growth)\label{B3} $f\in C^{3,3}(\RR^n\times\RR^m,\RR^m)$, $g\in C^{3,3}(\RR^n\times\RR^m,\RR^m\times \RR^{d_2})$, $h_2(\cdot,\cdot,z)\in C^{3,3}(\RR^n\times\RR^m,\RR^m)$ and  for any $0\leq i,j\leq 3$ with $1\leq i+j\leq 3$,
		\begin{eqnarray}
			&&\|\partial^{j}_y\partial^i_{x}f(x,y)\|+\|\partial^{j}_y\partial^i_{x}g(x,y)\|+\int_{\mathcal{Z}_2}\|\partial^j_{y}\partial^i_{x}h_{2}(x,y,z)\|^{p}\nu_{2}(dz)\le C_p(1+|y|^k).\label{ConB31}
		\end{eqnarray}
		Moreover,
		\begin{eqnarray}
			\quad \|g(x,y)\| \le C(1+|y|^{\zeta_1}), \quad \int_{\mathcal{Z}_2}|h_{2}(x,y,z)|^{p}\nu_{2}(dz)\le C_p(1+| y|^{p \zeta_2}).\label{ConB32}
		\end{eqnarray}
	\end{conditionB}
	
	\begin{remark}
		Here we give some comments on the conditions above.
		\begin{itemize}
			\item{ (\ref{ConA11}) and \eref{ConA112} are monotonicity and  local Lipschitz continuous for the coefficient $b$ \emph{w.r.t.} $x$ and $y$, respectively. For example, $b(x,y)=-|x|^2 x +x+|y|^3 y$. }
			\item{The constant $q\geq 2$ in \eref{ConA2} and \eref{ConB2} is used to ensure the existence and uniqueness of the solution of the stochastic system \eref{Equation} (see details in Lemma \ref{PMY}). }
			\item{ (\ref{ConB1})} ensures that the frozen equation admits a unique invariant measure $\mu^x$, which together with \eref{ConB11} ensure the partial derivative of the solution $Y^{x,y}_t$ of frozen equation with respect to $x$ lives in $L^{\ell}$ norm, i.e., $\sup_{t\geq 0}\EE\|\partial_xY^{x,y}_t\|^{\ell}<\infty$, for some $\ell>8$.
			
			\item{(\ref{ConA31}) and \eref{ConB31} are used to study the regularity estimates  of the solution of Poisson equation. }
			\item{The constants $\zeta_1,\zeta_2\in [0,1)$ in $\eref{ConB32}$ are used to prove the fast component $Y^{\vare}$ has finite moments of any order.}
		\end{itemize}
	\end{remark}
	
	\subsection{Main results}
	Let $\mu^{x}$ be the unique invariant measure for the transition semigroup of the following frozen equation:
	\begin{equation}
		\left\{\begin{array}{l} \label{Frozen equation}
			\displaystyle
			dY_{t}=f(x,Y_{t})dt+g(x,Y_{t})d\tilde{W}_{t}^{2}+\int_{\mathcal{Z}_2 }h_{2}(x,Y_{t-},z)\tilde{N}^{2}(dt,dz),\\
			Y_{0}=y\in\RR^m,\end{array}\right.
	\end{equation}
	where $\tilde{W}^{2}$ is a $d_2$-dimensional standard Wiener process and $\tilde{N}^{2}$ is compensated martingale measure with L\'evy measure $\nu_2$, $\tilde{W}^{2}$  and $\tilde{N}^{2}$ are independent.
	
	The following theorem is our first main result.
	\begin{theorem}(\textbf {Strong convergence})\label{main result 1}
		Suppose that assumptions \ref{A1}-\ref{A3}  and  \ref{B1}-\ref{B3} hold, and $\sigma(x,y)\equiv\sigma(x)$. Then for any $x\in\RR^{n}$, $y\in\RR^{m}$ and $T > 0$, there exists $C>0$ depends on $p,T$, $|x|,|y|$ such that for small enough $\varepsilon>0$, we have
		\begin{eqnarray*}
			\mathbb{E}\left(\sup_{0\leq t \leq T}\left |X_{t}^{\varepsilon}-\bar{X_{t}}\right |^{p}\right)\le C\varepsilon ^{\frac{p}{2}}.
		\end{eqnarray*}
		where $\bar{X_{t} }$ is the solution of the corresponding averaged equation:
		\begin{equation*}
			d\bar{X_{t} }=\bar{b}(\bar{X}_{t})dt+\sigma (\bar{X}_{t})dW_{t}^{1}+\int_{\mathcal{Z}_1 }h_{1}(\bar{X }_{t-},z)\tilde{N}^{1}(dz,dt) \\
			\displaystyle ,\bar{X}_0=x,
		\end{equation*}
		where $\bar{b}(x)=\int_{\RR^{m}}b(x,y)\mu^{x}(dy)$.
	\end{theorem}
	\begin{remark}
		The above result implies that the convergence order is $1/2$, which is the optimal order in the strong sense (see \cite[Example 1]{L2010}). Note that the L\'evy process considered here can not cover the $\alpha$-stable process, which does not have finite second moment, thus this is no contradiction with the obtained optimal strong convergence order $1-1/\alpha$ in \cite{SXX}. It is also worthy to point that the diffusion coefficient $\sigma$ does not depend on the fast component $Y^{\vare}$, otherwise the strong convergence may fail (see a counter-example in \cite[section 4.1]{L2010}).
	\end{remark}
	
	In order to prove the weak convergence order, we need the following assumption:
	
	For any for any $p\geq 2$, there exist positive constants $C,C_p$ and $k$ such that the following conditions hold for any $x\in\RR^n,y\in\RR^m, z_1\in\mathcal{Z}_1,z_2\in\mathcal{Z}_2$:

	\begin{conditionA} \label{A4} There exist $\ell>16$, $\beta>0$ and $L_{h_2}\in [0,\beta)$ such that
		\begin{eqnarray}
			&&2\left \langle f(x_1,y_{1})-f(x_2,y_{2}),y_{1}-y_{2}\right\rangle +(\ell-1)\| g(x_1,y_{1})-g(x_2,y_{2})\|^{2} \nonumber  \\
			&&\!\!\!\!\!+2^{\ell-3}(\ell-1)\!\int_{\mathcal{Z}_2}\!\!|h_{2}(x_1,y_{1},z)-h_{2}(x_2,y_{2},z)|^{2}\nu_2(dz)\le\!\! -\beta|y_1-y_2|^{2}\!+\!C|x_1-x_2|^2,
			\label{ConA422}
		\end{eqnarray}
		\begin{eqnarray}
			2^{\ell-3}(\ell-1)\int_{\mathcal{Z}_2}|h_{2}(x_1,y_{1},z)-h_{2}(x_2,y_{2},z)|^{\ell}\nu_2(dz)\le L_{h_2} |y_1-y_2|^{\ell}+C|x_1-x_2|^{\ell}.  \label{ConA423}
		\end{eqnarray}
		Moreover, $b\in C^{4,4}_p(\RR^n\times\RR^m,\RR^n)$, $\sigma\in C^{4,4}(\RR^n\times\RR^m,\RR^n\times\RR^{d_1})$, $h_1(\cdot,z_1)\in C^4(\RR^n,\RR^n)$, and for any $0\leq i,j\leq 4$ with $1\leq i+j\leq 4$,
		\begin{eqnarray}
			\|\partial^j_{y}\partial^i_{x}\sigma(x,y)\|\leq C(1+|y|^k),\label{ConA412}
		\end{eqnarray}
		\begin{eqnarray}
		\int_{\mathcal{Z}_1}\|\partial^i_xh_{1}(x,z)\|^{p}\nu_1(dz)\le C_p\label{ConA413}
		\end{eqnarray}
		and
		\begin{eqnarray}
			\inf_{x\in\RR^n,y\in \RR^m,z\in \RR^n\backslash\{0\}}\frac{\langle(\sigma(x,y)\sigma^{\ast}(x,y))\cdot z, z\rangle}{|z|^2}>0. \label{ConA42}
		\end{eqnarray}
		Furthermore $f\in C^{4,4}(\RR^n\times\RR^m,\RR^m)$, $g\in C^{4,4}(\RR^n\times\RR^m,\RR^m\times \RR^{d_2})$, $h_2(\cdot,\cdot,z_2)\in C^{4,4}(\RR^n\times\RR^m,\RR^m)$ and for any $0\leq i,j\leq 4$ with $1\leq i+j\leq 4$,
		\begin{eqnarray}
			\|\partial^{j}_y\partial^i_{x}f(x,y)\|+\|\partial^{j}_y\partial^i_{x}g(x,y)\|+\int_{\mathcal{Z}_2}\|\partial^j_{y}\partial^i_{x}h_{2}(x,y,z)\|^{p}\nu_{2}(dz)\le C_p(1+|y|^k).\label{ConA43}
		\end{eqnarray}
	\end{conditionA}

	\begin{remark}
		Roughly speaking, assumption \ref{A4} is used to study the regularity estimate of the solution $u(t,x)$ of the Kolmogorov equation (see equation \eref{KE} in section 5.2). More precisely,
		\eref{ConA412} is used to study the regularity estimate of the averaged coefficient  $\overline{\sigma\sigma^{\ast}}$; \eref{ConA42} is used to study the regularity estimate of the averaged coefficient $\bar{\sigma}=\left(\overline{\sigma\sigma^{\ast}}\right)^{1/2}$;
		\eref{ConA422} and \eref{ConA423} are stronger than \eref{ConB1} and \eref{ConB11}, however these two conditions are used to study the partial derivative of the solution $Y^{x,y}_t$ with respect to $x$ lives in $L^{\ell}$ norm, i.e., $\sup_{t\geq 0}\EE\|\partial_xY^{x,y}_t\|^{\ell}<\infty$, for some $\ell>16$.
	\end{remark}

	The following theorem is our second main result.
	\begin{theorem}(\textbf {Weak convergence})\label{main result 2}
		Suppose that assumptions \ref{A1}-\ref{A4}  and  \ref{B1}-\ref{B3} hold. Then for any $\phi\in C^{4}_p(\RR^n)$, $x\in\RR^{n}$, $y\in\RR^{m}$ and $T > 0$, there exists $C>0$ depends on $T$, $|x|,|y|$ such that for any $\varepsilon>0$,
		\begin{eqnarray*}
			\sup_{t\in[0,T]}\left | \mathbb{E}\phi (X_{t}^{\varepsilon})-\mathbb{E}\phi (\bar{X_{t} })\right |\le C\varepsilon.
		\end{eqnarray*}
		where $\bar{X_{t} }$ is the solution of the corresponding averaged equation:
		\begin{equation*}
			d\bar{X_{t} }=\bar{b}(\bar{X}_{t})dt+\bar \sigma (\bar{X}_{t})d W_{t}+\int_{\mathcal{Z}_1 }^{}h_{1}(\bar{X }_{t-},z)\tilde{N}^{1}(dz,dt) \\
			\displaystyle ,\bar{X}_0=x,
		\end{equation*}
		where $\bar{b}(x)=\int_{\RR^{m}}b(x,y)\mu^{x}(dy)$, $\bar \sigma (x)=\left[\int_{\RR^m}\sigma(x,y)\sigma^{\ast}(x,y)\mu^x(dy)\right]^{1/2}$ (i.e., $\bar \sigma (x)\bar \sigma (x)=\int_{\RR^m}\sigma(x,y)\sigma^{\ast}(x,y)\mu^x(dy)$) and $W$ is a $n$-dimensional standard Wiener process independent of $\tilde{N}^1$.
	\end{theorem}
	\begin{remark}
		The above result implies that the weak convergence order is $1$. It is worthy to point that the diffusion coefficient $\sigma(x,y)$ can depend on the fast component, however the coefficient $h_1(x,z)$ here is independent of the fast component due to the reason of technique. In fact, since the generator of the jump part is a nonlocal operator, it usually non-trivial to describe the averaged coefficients $\bar{h}_1(x,z)$ if $h_1(x,y,z)$ depends of the fast component. Of course, $h_1(x,y,z)$ may achieve its averaged coefficient $\bar{h}_1(x,z)$ in some special jump processes.
	\end{remark}
	
	\subsection{Examples}
	In this subsection, we give two concrete examples to illustrate the applicability of our main results. For simplicity, we only consider the 1-dimensional case, but one can
	easily extend to the multi-dimensional case.
	
	\begin{example} \label{EX1}Let us consider the following slow-fast SDEs,
		\begin{equation}\left\{\begin{array}{l}
				\displaystyle
				d X^{\vare}_t = \left[-(X^{\vare}_t)^3+X^{\vare}_t+(Y^{\vare}_t)^3\right]dt+ \sigma(X^{\vare}_t,Y^{\vare}_t)d W^{1}_t+\int_{\RR}z\tilde{N}^1(dz,dt),\quad X^{\vare}_0=x\in \RR,\nonumber\\
				\displaystyle
				d Y^{\vare}_t =\frac{1}{\vare}\left[\sin(X^{\vare}_t)-Y^{\vare}_t-(Y^{\vare}_t)^5\right]dt+\frac{1}{\sqrt{\vare}}d W^{2}_t+\int_{\RR}z\tilde{N}^{2,\vare}(dz,dt),\quad Y^{\vare}_0=y\in \RR,\nonumber
			\end{array}\right.
		\end{equation}
		where $\{W^{1}_t\}_{t\geq 0}$ and $\{W^{2}_t\}_{t\geq 0}$ are independent $1$-dimensional Brownian motions, $\tilde{N}^{1}$ and $\tilde{N}^{2,\vare}$ are two compound Poisson random measures with L\'evy measures $\nu_1$ and $\frac{1}{\varepsilon}\nu_2$ respectively, here $\int_{|z|\leq 1}|z|^2\nu_i(dz)\leq C$ and $\nu_i([1,\infty))=0$, for $i=1,2$.
		
		Let
		$$
		b(x,y)=-x^3+x+y^3, \quad h_1(x,z)=z,
		$$
		and
		$$
		f(x,y)=\sin(x)-y-y^5, \quad g(x,y)=1,\quad \quad h_2(x,y,z)=z.
		$$
		
		On one hand, if $ \sigma(x,y)=x$, it is easy to verify that \ref{A1}-\ref{A3} and \ref{B1}-\ref{B3} hold with
		$k=4$, $q=6$, $\beta=2$, $\lambda=1/2$, $\forall \ell>8$ and $L_{h_2}=\zeta_1=\zeta_2=0$.
		Thus,  by Theorem \ref{main result 1} for any $p>0$ we have
		\begin{eqnarray*}
			\mathbb{E}\left(\sup_{t\in [0, T]}|X_{t}^{\vare}-\bar{X}_{t}|^p\right)\leq C\vare^{p/2},
		\end{eqnarray*}
		where $\bar{X}_{t}$ is the solution of the corresponding averaged equation.
		
		On the other hand, if $\sigma(x,y)=\sin(x)+\sin(y)+3$, it is easy to verify that \ref{A1}-\ref{A4} and \ref{B1}-\ref{B3} hold with
		$k=4$, $q=6$, $\beta=2$, $\lambda=1/2$, $\forall \ell>16$ and $L_{h_2}=\zeta_1=\zeta_2=0$.
		Thus,  by Theorem \ref{main result 2} for any  $\phi\in C^{4}_p(\RR^n)$, we have
		\begin{eqnarray*}
			\sup_{t\in[0,T]}\left | \mathbb{E}\phi (X_{t}^{\varepsilon})-\mathbb{E}\phi (\bar{X_{t} })\right |\le C\varepsilon,
		\end{eqnarray*}
		where $\bar{X}_{t}$ is the solution of the corresponding averaged equation.
	\end{example}
	
	\vspace{0.2cm}

	\begin{example} Let us consider the following slow-fast SDEs,
		\begin{equation}\left\{\begin{array}{l}\label{EX2}
				\displaystyle
				d X^{\vare}_t = \left[X^{\vare}_t-\arctan(X^{\vare}_t) (Y^{\vare}_t)^2+Y^{\vare}_t\right]dt+d W^{1}_t+\int_{\RR}z\tilde{N}^1(dz,dt),\quad X^{\vare}_0=x\in \RR,\nonumber\\
				\displaystyle
				d Y^{\vare}_t =\frac{1}{\vare}\left[\cos(X^{\vare}_t)-(Y^{\vare}_t)^3 \right]dt+\frac{1}{\sqrt{\vare}}d W^{2}_t+\int_{\RR}z\tilde{N}^{2,\vare}(dz,dt),\quad Y^{\vare}_0=y\in \RR,\nonumber
			\end{array}\right.
		\end{equation}
		where $\{W^{1}_t\}_{t\geq 0}$, $\{W^{2}_t\}_{t\geq 0}$, $\tilde{N}^{1}$ and $\tilde{N}^{2,\vare}$ are the same setting in Example \ref{EX1}.
		
		Let
		$$
		b(x,y)=x-\arctan(x) y^2+ y, \quad \sigma(x,y)=1, \quad h_1(x,z)=z,
		$$
		and
		$$
		f(x,y)=\cos(x)-y^3, \quad g(x,y)=1,\quad h_2(x,y,z)=z,
		$$
		Then it is easy to verify that \ref{A1}-\ref{A3} and \ref{B1}-\ref{B3} hold with $k=2$, $q=4$,$\beta=1$, $\lambda=1/2$, $\forall \ell>8$ and $L_{h_2}=\zeta_1=\zeta_2=0$.
		Thus,  by Theorem \ref{main result 1} for any $p>0$ we have
		\begin{eqnarray*}
			\mathbb{E}\left(\sup_{t\in [0, T]}|X_{t}^{\vare}-\bar{X}_{t}|^p\right)\leq C\vare^{p/2},
		\end{eqnarray*}
		where $\bar{X}_{t}$ is the solution of the corresponding averaged equation.
		
		Furthermore, it is easy to verify that \ref{A4} holds for any $\forall \ell>16$. Thus,  by Theorem \ref{main result 2} for any  $\phi\in C^{4}_p(\RR^n)$, we have
		\begin{eqnarray*}
			\sup_{t\in[0,T]}\left | \mathbb{E}\phi (X_{t}^{\varepsilon})-\mathbb{E}\phi (\bar{X_{t} })\right |\le C\varepsilon,
		\end{eqnarray*}
		where $\bar{X}_{t}$ is the solution of the corresponding averaged equation.
	\end{example}
	
	\section{Preliminaries}
	
	This section is a preparation for the proofs of our main results. In subsection 3.1, we give some  a priori estimates of the solution $(X^{\vare}_t, Y^{\vare}_t)$.
	In subsection 3.2, we introduce the frozen equation and give some estimates of the solution, then prove the exponential ergodicity of the corresponding transition semigroup. In the final subsection, we study the averaged equation. Note that the assumptions \ref{A1}-\ref{A3}  and  \ref{B1}-\ref{B3} hold in this section.
	
	\subsection{A priori estimates of $(X^{\vare}_t, Y^{\vare}_t)$}
	\begin{lemma} \label{PMY}
		For any initial value $x\in\RR^{n}, y\in \RR^{m}$ and $\vare>0$, system (\ref{Equation})  admits a unique strong solution $\{(X^{\vare}_t,Y^{\vare}_t), t\geq 0\}$. Moreover, for any $p\ge 1$ and  $T>0$, there exists $C_{p,T}>0$ such that
		\begin{eqnarray}
			\sup_{\vare>0}\mathbb{E}\left(\sup_{0\leq t\leq T}|X_{t}^{\vare}|^p\right)\leq C_{p,T}(1+|x|^p+|y|^{\frac{pq}{2}\vee (k+1)p}),  \label{X}
		\end{eqnarray}
		\vspace{0.1cm}
		and
		\vspace{0.1cm}
		\begin{eqnarray}
			\!\!\!\!\!\!\!\!\!\!&&\sup_{\vare>0}\sup_{t\ge 0}\mathbb{E}| Y_t^{\varepsilon }|^{p}\leq C_{p}(1\!+|y|^{p}).\label{Y}
		\end{eqnarray}
	\end{lemma}
	
	\begin{proof}
		We denote
		$$
		Z^{\vare}_t=\left(
		\begin{array}{c}
			X^{\vare}_t \\
			Y^{\vare}_t \\
		\end{array}
		\right)
		,\quad \tilde{b}^{\vare}(x,y)=\left(
		\begin{array}{c}
			b(x,y) \\
			\frac{1}{\vare}f(x,y) \\
		\end{array}
		\right)
		,\quad \tilde{h}_1(x,y,z)=\left(
		\begin{array}{c}
			h_1(x,z) \\
			0 \\
		\end{array}
		\right)
		$$
		and
		$$
		\tilde{h}_2(x,y,z)=\left(
		\begin{array}{c}
			0\\
			h_2(x,y,z) \\
		\end{array}
		\right)
		,\quad \tilde{\sigma}^{\vare}(x,y)=\left(
		\begin{array}{cc}
			\sigma(x,y) & 0 \\
			0 & \frac{1}{\sqrt{\vare}}g(x,y) \\
		\end{array}
		\right)
		, \quad \tilde{W}_t=\left(
		\begin{array}{c}
			W^1_t \\
			W^2_t \\
		\end{array}
		\right).
		$$
		Then system \eref{Equation} can be rewritten as the following equation
		\begin{equation}
			\left\{\begin{array}{l} \label{Eq2}
				\displaystyle
				dZ^{\vare}_t=\tilde{b}^{\vare}(Z^{\vare}_t)dt+\tilde{\sigma}^{\vare}( Z^{\vare}_t)d \tilde{W}_t+\int_{\mathcal{Z}_1}\tilde{h}_1(Z^{\vare}_{t-},z)\tilde{N}^1(dt,dz)+\int_{\mathcal{Z}_2}\tilde{h}_2(Z^{\vare}_{t-},z)\tilde{N}^{2,\vare}(dt,dz),\\
				Z^{\vare}_0= \left(
				\begin{array}{c}
					x \\
					y\\
				\end{array}
				\right). \end{array}\right.
		\end{equation}
		
		Under the assumptions \ref{A1}-\ref{A3} and \ref{B1}-\ref{B3}, it is easy to check that for any $R>0$, $z_i=(x_i, y_i)\in \RR^{n+m}$ with $|z_i|\leq R$, $i=1,2$, there exists $C_{R,\vare}>0$ such that
		\begin{eqnarray*}
			&&2\langle \tilde{b}^{\vare}(z_1)-\tilde{b}^{\vare}(z_2), z_1-z_2\rangle+\|\tilde{\sigma}^{\vare}(z_1)-\tilde{\sigma}^{\vare}(z_2)\|^2+\int_{\mathcal{Z}_1}|\tilde{h}_1(z_1,z)-\tilde{h}_1(z_2,z)|^2\nu_1(dz)\\
			&&\quad\quad\quad+\int_{\mathcal{Z}_2}|\tilde{h}_2(z_1,z)-\tilde{h}_2(z_2,z)|^2\frac{1}{\vare}\nu_2(dz)\\
			\leq\!\!\!\!\!\!\!\!&&2\langle b(x_1, y_1)-b(x_2, y_2), x_1-x_2\rangle+\frac{2}{\vare}\langle f(x_1, y_1)-f(x_2, y_2), y_1-y_2\rangle\\
			&&+\|\sigma(x_1,y_1)-\sigma(x_2,y_2)\|^2+\frac{1}{\vare}\|g(x_1, y_1)-g(x_2, y_2)\|^2\\
			&&+\int_{\mathcal{Z}_1}|h_1(x_1,z)-h_1(x_2,z)|^2\nu_1(dz)+\frac{1}{\vare}\int_{\mathcal{Z}_2}|h_2(x_1,y_1,z)-h_2(x_2,y_2,z)|^2\nu_2(dz)\\
			\leq\!\!\!\!\!\!\!\!&& C_{\vare}|x_1-x_2|^2+C(1+|y_1|^k+|y_2|^k)|y_1-y_2||x_1-x_2|+C(1+|y_1|^{2k}+|y_2|^{2k})|y_1-y_2|^2\\
			\leq\!\!\!\!\!\!\!\!&&  C_{R,\vare}|z_1-z_2|^2.
		\end{eqnarray*}
		Furthermore, for small enough $\vare>0$,
		\begin{eqnarray*}
			&&2\langle \tilde{b}^{\vare}( z_1), z_1\rangle+\|\tilde{\sigma}^{\vare}(z_1)\|^2+\int_{\mathcal{Z}_1}|\tilde{h}_1(z_1,z)|^2\nu_1(dz)+\int_{\mathcal{Z}_2}|\tilde{h}_2(z_1,z)|^2\frac{1}{\vare}\nu_2(dz)\\
			\leq\!\!\!\!\!\!\!\!&&2\langle b( x_1, y_1), x_1\rangle+\frac{2}{\vare}\langle f(x_1, y_1), y_1\rangle+\|\sigma(x_1,y_1)\|^2+\frac{1}{\vare}\|g(x_1, y_1)\|^2\\
			&&\quad\quad\quad\quad+\int_{\mathcal{Z}_1}|h_1(x_1,z)|^2\nu_1(dz)+\int_{\mathcal{Z}_2}|h_2(x_1,y_1,z)|^2\frac{1}{\vare}\nu_2(dz)\\
			\leq\!\!\!\!\!\!\!\!&&C(1+|x_1|^2)+C|y_1|^{q}-\frac{2\lambda|y_1|^{q}}{\vare}+\frac{C}{\vare}(1+|y_1|^{2})\\
			\leq\!\!\!\!\!\!\!\!&& C_\vare(1+|z_1|^2).
		\end{eqnarray*}
		Hence by \cite[Theorem 2.8]{XZ2019}, there exists a unique solution $\{(X^{\vare}_t,Y^{\vare}_t), t\geq 0\}$ to system (\ref{Equation}).
		
		It is sufficient to prove \eref{X} and \eref{Y} for $p$ is large enough. Using It\^{o}'s formula and taking expectation, we get for any $p\geq 4$,
		\begin{eqnarray*}
			\mathbb{E}\left|Y_t^{\varepsilon }\right|^{p}=\!\!\!\!\!\!\!\!&&|y|^{p}+\frac{p}{\varepsilon}\mathbb{E}\int_{0}^{t}\big|Y_s^{\varepsilon }\big|^{p-2}\left\langle Y_{s}^{\varepsilon },f(X_{s}^{\varepsilon },Y_{s}^{\varepsilon })\right\rangle ds\\
			\!\!\!\!\!\!\!\!&&+\frac{p(p\!-\!2)}{2\varepsilon}\mathbb{E}\int_{0}^{t}\big|Y_s^{\varepsilon }\big|^{p-4}|g^{*}(X_{s}^{\varepsilon },Y_{s}^{\varepsilon })\cdot Y_{s}^{\varepsilon }|^{2}ds\!+\!\frac{p}{2\varepsilon}\mathbb{E}\int_{0}^{t}\big|Y_s^{\varepsilon }\big|^{p-2}\left\|g(X_{s}^{\varepsilon },Y_{s}^{\varepsilon }) \right\|^{2}ds\\
			\!\!\!\!\!\!\!\!&&+\frac{1}{\varepsilon}\mathbb{E}\int_{0}^{t}\int_{\mathcal{Z}_2}   \left[\big|Y_s^{\varepsilon}+h_2(X_{s}^{\varepsilon},Y_s^{\varepsilon},z)\big|^{p}-\big|Y_s^{\varepsilon}\big|^{p}-p\big|Y_s^{\varepsilon}\big|^{p-2}\langle Y_s^{\varepsilon}, h_2(X_{s}^{\varepsilon},Y_s^{\varepsilon},z)\rangle\right]\nu_2(dz)ds.\\
		\end{eqnarray*}
	
Note that  using Taylor's formula on $F\in C^2(\mathbb{R}^m)$, it follows for any $x=(x_1,\ldots,x_m)$ and $h=(h_1,\ldots,h_m)$, there exists $\xi\in (0,1)$ such that
\begin{eqnarray*}
			F(x+h)-F(x)=\!\!\!\!\!\!\!\!\!\!&&\sum^m_{i=1}\partial_{x_i}F(x)h_i+\frac{1}{2}\sum^m_{i,j=1}\partial_{x_i}\partial_{x_j}F(x+\xi h)h_i h_j,
 \end{eqnarray*}
Then taking $F(x)=|x|^p$ for $p\geq 4$, it is easy to see
 \begin{eqnarray*}
 &&D F(x)=(\partial_{x_1}F(x), \ldots,\partial_{x_m}F(x))=p|x|^{p-2}x,\\
 &&D^2 F(x)=\left(\partial_{x_i}\partial_{x_j}F(x)\right)_{1\leq i,j\leq m}=p(p-2)|x|^{p-4}x\otimes x+p|x|^{p-2}E_{m\times m},
 \end{eqnarray*}
 where $E_{m\times m}$ is the unit matrix on $\mathbb{R}^m$, thus we have for any $a=(a_1,\ldots,a_m)$ and $b=(b_1,\ldots,b_m)$, there exists $\xi\in (0,1)$ such that
\begin{eqnarray}
			&&\left||a+b|^{p}-|a|^{p}-p|a|^{p-2}\langle a, b\rangle\right|\nonumber\\
=\!\!\!\!\!\!\!\!\!\!&&\frac{1}{2}\left|p(p-2)|a+\xi b|^{p-4}\sum^n_{i,j=1}(a_i+b_i\xi)(a_j+b_j\xi)b_i b_j+p|a+\xi b|^{p-2}|b|^2\right|\nonumber\\
\leq\!\!\!\!\!\!\!\!\!\!&&\frac{1}{2}p(p-1)|a+\xi b|^{p-2}|b|^2\nonumber\\
			\leq\!\!\!\!\!\!\!\!\!\!&&2^{p-4}p(p-1)\left(|a|^{p-2}+|b|^{p-2}\right)|b|^2\nonumber\\
\leq\!\!\!\!\!\!\!\!\!\!&&2^{p-4}p(p-1)\left(|a|^{p-2}|b|^2+|b|^{p}\right). \label{Taylor}
		\end{eqnarray}
By \eref{Taylor}, \eref{ConB2} and \eref{ConB32}, we get
		\begin{eqnarray*}
			\frac{d}{dt}\mathbb{E}\left|Y_t^{\varepsilon}\right|^{p}\leq\!\!\!\!\!\!\!\!\!\!&& \frac{p}{\varepsilon}\mathbb{E}\left[\big|Y_t^{\varepsilon }\big|^{p-2}\left \langle Y_{t}^{\varepsilon },f(X_{t}^{\varepsilon },Y_{t}^{\varepsilon })\right \rangle \right]+\!\frac{C_p}{\varepsilon}\mathbb{E}\left[|Y_t^{\varepsilon }|^{p-2}\left\|g(X_{t}^{\varepsilon },Y_{t}^{\varepsilon }) \right\|^{2}\right] \\ \nonumber
			\!\!\!\!\!\!\!\!&&+ \frac{C_p}{\varepsilon}\mathbb{E}\int_{\mathcal{Z}_2} |Y_t^{\varepsilon }|^{p-2}   \big|h_2(X_{t}^{\varepsilon},Y_t^{\varepsilon},z)\big|^{2}\nu_2(dz)+\frac{C_p}{\varepsilon}\mathbb{E}\int_{\mathcal{Z}_2}|h_2(X_{t}^{\varepsilon},Y_t^{\varepsilon},z)|^{p}\nu_2(dz) \\ \nonumber
			\leq\!\!\!\!\!\!\!\!\!\!&& \frac{1}{\varepsilon}\mathbb{E}\left[p|Y_{t}^{\varepsilon} |^{p-2}(-\lambda |Y_{t}^{\varepsilon}|^{2}+C)+C_p|Y_{t}^{\varepsilon} |^{p-2}(1+|Y_{t}^{\varepsilon} |^{2\zeta_1}) \right.\\  \nonumber
			\!\!\!\!\!\!\!\!\!\!&&\quad\quad+\left. C_p|Y_{t}^{\varepsilon} |^{p-2}(1+|Y_{t}^{\varepsilon} |^{2\zeta_2})+C_p(1+|Y_{t}^{\varepsilon} |^{p\zeta_2})
			\right]  \\ \nonumber	
			\leq\!\!\!\!\!\!\!\!\!\!&& -\frac{\gamma p}{\varepsilon}\mathbb{E}\big|Y_{t}^{\varepsilon } \big|^{p}+ \frac{C_p}{\varepsilon},
		\end{eqnarray*}
		where $\gamma\in(0,\lambda)$. Thus, by comparison theorem, we have
		\begin{eqnarray*}
			\mathbb{E}\left|Y_t^{\varepsilon }\right|^{p}\leq e^{-\frac{\gamma p t}{\vare}}|y|^p+\frac{C_p}{\vare}\int^t_0 e^{-\frac{\gamma p (t-s)}{\vare}}ds,
		\end{eqnarray*}
		which implies
		\begin{eqnarray}
			\sup_{t\ge 0}\mathbb{E}\left|Y_t^{\varepsilon }\right|^{p}\leq e^{-\frac{\gamma p t}{\vare}}|y|^p+C_p\leq C_{p}(1+|y|^{p}).	  \label{Y1}
		\end{eqnarray}
		
		According to It\^{o}'s formula again, it is easy to see
		\begin{eqnarray}
			|X_t^{\varepsilon}|^{p}=\!\!\!\!\!\!\!\!\!\!&&|x|^{p}+p\int_{0}^{t}|X_s^{\varepsilon }|^{p-2}\left \langle X_{s}^{\varepsilon },b(X_{s}^{\varepsilon },Y_{s}^{\varepsilon })\right \rangle ds \nonumber \\
			&&+\int_{0}^{t}\left[\frac{p(p\!-\!2)}{2}|X_s^{\varepsilon }|^{p-4}|\sigma^{*}(X_{s}^{\varepsilon },Y_{s}^{\varepsilon } )\cdot X_{s}^{\varepsilon } |^{2}+\frac{p}{2}|X_s^{\varepsilon }|^{p-2}\left\|\sigma(X_s^{\varepsilon },Y_{s}^{\varepsilon })\right\|^{2}\right]ds \nonumber \\
			&&+\int_0^t\int_{\mathcal{Z}_1} \left[\big|X_s^{\varepsilon}\!+\!h_1(X_{s}^{\varepsilon},z)\big|^{p}\!-\!|X_s^{\varepsilon}|^{p}\!-\!p|X_s^{\varepsilon }|^{p-2}\langle X_{s}^{\varepsilon}, h_1(X_{s}^{\varepsilon},z) \rangle\right] \nu_1(dz)ds \nonumber \\
			&&+p\int_{0}^{t}|X_s^{\varepsilon }|^{p-2}\left \langle X_{s}^{\varepsilon },\sigma(X_{s}^{\varepsilon }, Y_{s}^{\varepsilon })dW_s^{1}\right \rangle \nonumber \\
			&&+\int_0^t\int_{\mathcal{Z}_1} \left[  |X_{s-}^{\varepsilon}+h_1(X_{s-}^{\varepsilon},z)|^{p}-|X_{s-}^{\varepsilon}|^{p}\right]\tilde{N}^{1}(dz,ds) \nonumber \\
			=:\!\!\!\!\!\!\!\!\!\!&&|x|^{p}+\sum_{i=1}^{5}I_i(t). \label{X1}
		\end{eqnarray}
		
		For the term $I_1(t)$. By (\ref{ConA2}), we get that
		\begin{eqnarray}
			\EE\left(\sup_{0\leq t\leq T}|I_1(t)|\right)\le\!\!\!\!\!\!\!\!\!\!&& C_p\EE\int_{0}^{T}|X_s^{\varepsilon }|^{p-2}(1+|X_s^{\varepsilon }|^{2}+|Y_s^{\varepsilon }|^{q})ds  \nonumber  \\
			&& \le C_{p}\int_{0}^{T}(1+\EE|X_s^{\varepsilon }|^{p}+\EE|Y_s^{\varepsilon }|^{\frac{pq}{2}}) ds, \label{X2}
		\end{eqnarray}
		
		For the term $I_2(t)$. Note that (\ref{ConA113}) implies
\begin{eqnarray}
\|\sigma(x,y)\|\leq C(1+|x|+|y|^{k+1}),\label{sigmaP}
\end{eqnarray}
 thus we have
		\begin{eqnarray}
			\EE\left(\sup_{0\leq t\leq T}|I_2(t)|\right)\le\!\!\!\!\!\!\!\!\!\!&&C_p \EE\int_{0}^{T}|X_s^{\varepsilon }|^{p-2}\left\|\sigma(X_s^{\varepsilon }, Y_s^{\varepsilon })\right\|^{2}ds \nonumber \\
			\le\!\!\!\!\!\!\!\!\!\!&&C_p \EE\int_{0}^{T}|X_s^{\varepsilon }|^{p-2}\left(1+|X_s^{\varepsilon }|^2+|Y_s^{\varepsilon }|^{2(k+1)}\right)ds \nonumber \\
			\le\!\!\!\!\!\!\!\!\!\!&& C_{p}\int_{0}^{T}(1+\EE|X_s^{\varepsilon }|^{p}+\EE|Y_s^{\varepsilon }|^{(k+1)p}) ds\nonumber\\
			\le\!\!\!\!\!\!\!\!\!\!&& C_{p,T}\left(1+|y|^{(k+1)p}\right)+C_p\int^T_0 \EE|X_s^{\varepsilon }|^{p} ds.\label{X3}
		\end{eqnarray}
		
		For the term $I_3(t)$. Using \eref{Taylor} and (\ref{ConA32}),
		\begin{eqnarray}
			\EE\left(\sup_{0\leq t\leq T}|I_3(t)|\right)\le\!\!\!\!\!\!\!\!\!\!&& C_{p}\EE\int_{0}^{T}\int_{\mathcal{Z}_1} \left[|X_s^{\varepsilon }|^{p-2}|h_1(X_{s}^{\varepsilon},z)|^{2}+|h_1(X_{s}^{\varepsilon},z)|^{p} \right]\nu_1(dz)ds \nonumber \\
			&& \le C_{p}\int_{0}^{T}(1+\EE|X_s^{\varepsilon }|^{p}) ds. \label{X4}
		\end{eqnarray}
		
For the term $I_4(t)$. By Burkholder-Davis-Gundy's inequality (see \cite[Theorem 3.49]{PZ2007}), \eref{sigmaP} and Young's inequality, we have
		\begin{eqnarray}
			\mathbb{E}\left(\sup_{0\leq t\leq T}|I_4(t)|\right)\leq\!\!\!\!\!\!\!\!\!\!&& C_p\mathbb{E}\left[\int_{0}^{T}\big|X_s^{\varepsilon }\big|^{2p-2}\left\|\sigma(X_s^{\varepsilon }, Y_s^{\varepsilon })\right\|^{2}ds \right]^{\frac{1}{2}}\nonumber \\
			\leq\!\!\!\!\!\!\!\!\!\!&&C_p\mathbb{E}\left[\left(\sup_{0\leq s\leq T}|X_s^{\varepsilon }|^{p}\right)\left(\int_{0}^{T}|X_s^{\varepsilon }|^{p-2}(1+|X_s^{\varepsilon }|^{2}+|Y_s^{\varepsilon }|^{2(k+1)})ds\right) \right]^{\frac{1}{2}}  \nonumber \\
			\leq\!\!\!\!\!\!\!\!\!\!&&\frac{1}{3}\mathbb{E}\left(\sup_{0\leq t\leq T}|X_t^{\varepsilon }|^{p}\right)+ C_p\int_{0}^{T}(1+\mathbb{E}|X_s^{\varepsilon }|^{p}+\mathbb{E}|Y_s^{\varepsilon }|^{p(k+1)})ds.  \label{X5}
		\end{eqnarray}

		For the term $I_5(t)$. Note that  using Taylor's formula on $F\in C^1(\mathbb{R}^n)$, it follows for any $x=(x_1,\ldots,x_n)$ and $h=(h_1,\ldots,h_n)$, there exists $\xi\in (0,1)$ such that
\begin{eqnarray*}
			F(x+h)-F(x)=\!\!\!\!\!\!\!\!\!\!&&\sum^n_{i=1}\partial_{x_i}F(x+\xi h)h_i,
 \end{eqnarray*}
 thus we have for any $a=(a_1,\ldots,a_n)$ and $b=(b_1,\ldots,b_n)$, there exists $\xi\in (0,1)$ such that
\begin{eqnarray}
			\left||a+b|^{p}-|a|^{p}\right|=\!\!\!\!\!\!\!\!\!\!&&\left|p|a+\xi b|^{p-2}\sum^n_{i=1}(a_i+b_i\xi) b_i\right|\nonumber\\
\leq\!\!\!\!\!\!\!\!\!\!&&p|a+\xi b|^{p-1}|b|\nonumber\\
\leq\!\!\!\!\!\!\!\!\!\!&&p2^{p-2}\left(|a|^{p-1}|b|+|b|^{p}\right). \label{Taylor2}
	\end{eqnarray}	
By Burkholder-Davis-Gundy's inequality (see \cite[Theorem 3.50]{PZ2007}), \eref{Taylor2},  Young's inequality and (\ref{ConA32}), we have
		\begin{eqnarray}
			\mathbb{E}\left(\sup_{0\leq t\leq T}|I_5(t)|\right)\leq\!\!\!\!\!\!\!\!&& C\mathbb{E}\left[\int_{0}^{T}\int_{\mathcal{Z}_1}\left(|X_{s-}^{\varepsilon}+h_{1}\left(X_{s-}^{\varepsilon},z\right)|^{p}-|X_{s-}^{\varepsilon}|^{p}\right)^{2}N^1(dz,ds)\right]^{\frac{1}{2}}\nonumber\\
			\leq\!\!\!\!\!\!\!\!&& C_p\mathbb{E}\left[\int_{0}^{T}\int_{\mathcal{Z}_1}\left[|X_{s-}^{\varepsilon}|^{2p-2}|h_{1}(X_{s-}^{\varepsilon},z)|^{2}+|h_{1}(X_{s-}^{\varepsilon},z)|^{2p}\right]N^1(dz,ds)\right]^{\frac{1}{2}}
			\nonumber\\
			\leq\!\!\!\!\!\!\!\!&& C_p\mathbb{E}\left[\left(\sup_{0\leq t\leq T}|X_{t}^{\varepsilon}|^{p}\right)\int_{0}^{T}\int_{\mathcal{Z}_1}|X_{s-}^{\varepsilon}|^{p-2}|h_{1}(X_{s-}^{\varepsilon},z)|^{2}N^1(dz,ds)\right]^{\frac{1}{2}}\nonumber\\
			&&+C_p\mathbb{E}\left[\int_{0}^{T}\int_{\mathcal{Z}_1}|h_{1}(X_{s-}^{\varepsilon},z)|^{2p}N^1(dz,ds)\right]^{\frac{1}{2}}
			\nonumber\\
			\leq\!\!\!\!\!\!\!\!&&\frac{1}{3}\mathbb{E}\left(\sup_{0\leq t\leq T}|X_{t}^{\varepsilon}|^{p}\right)+C_{p}\mathbb{E}\int_{0}^{T}\int_{\mathcal{Z}_1}|X_{s}^{\varepsilon}|^{p-2}|h_{1}(X_{s}^{\varepsilon},z)|^{2}\nu_1(dz)ds\nonumber\\
			&&+C_p\mathbb{E}\int_{0}^{T}\int_{\mathcal{Z}_1}|h_{1}(X_{s}^{\varepsilon},z)|^{p}\nu_1(dz)ds\nonumber\\
			\leq\!\!\!\!\!\!\!\!&&\frac{1}{3}\mathbb{E}\left(\sup_{0\leq t\leq T}|X_{t}^{\varepsilon}|^{p}\right)+C_{p}\mathbb{E}\int_{0}^{T}|X_{t}^{\varepsilon}|^{p}dt+C_{p,T},\label{X6}
		\end{eqnarray}
where the fourth inequality comes from a fact in the proof of \cite[Lemma 8.22]{PZ2007}, i.e., for any $\theta\in (0,1]$,
   	\begin{eqnarray*}
		\left[\int^T_0\int_{\mathcal{Z}_1}|h_{1}(X_{s}^{\varepsilon},z)|^{2p} N^1(dz,ds)\right]^{\theta}\leq \int^T_0\int_{\mathcal{Z}_1}|h_{1}(X_{s}^{\varepsilon},z)|^{2p\theta}N^1(dz,ds).
	\end{eqnarray*}
		
		Combining (\ref{Y}), (\ref{X2})-(\ref{X6}),we get that
		\vspace{0.1cm}
		\begin{eqnarray*}
			\mathbb{E}\left(\sup_{0\leq t\leq T}| X_t^{\varepsilon }|^{p}\right)\leq\!\!\!\!\!\!\!\!\!\!&& C_{p,T}(1\!+\!|x|^{p})\!+\!C_p\!\int_{0}^{T}\!\!\mathbb{E}|X_t^{\varepsilon}|^{p}dt+C_p\!\!\int_{0}^{T}\!\!\mathbb{E}(|Y_t^{\varepsilon }|^\frac{{pq}}{2}+|Y_t^{\varepsilon }|^{(k+1)p})dt \nonumber \\
			\!\!\!\!\!\!\!\!\!\!&&\leq C_{p,T}(1\!+\!|x|^{p}+|y|^{\frac{pq}{2}\vee(k+1)p})\!+\!C_p\!\int_{0}^{T}\!\!\mathbb{E}|X_s^{\varepsilon}|^{p}ds.  \label{X7}
		\end{eqnarray*}
	
		Thus by Gronwall's inequality, we obtain
		\begin{eqnarray*}
			\mathbb{E}\left(\sup_{0\leq t\leq T}| X_t^{\varepsilon }|^{p}\right)\leq C_{p,T}(1\!+\!|x|^{p}+|y|^{ \frac{pq}{2}\vee(k+1)p}).
		\end{eqnarray*}
		The proof is complete.
	\end{proof}

	\vspace{0.1cm}
	\begin{lemma} \label{PMY2}
		For any $p\geq 4$ and $T>0$, there exists positive constant $C_p$ such that for $\vare\in (0,T]$,
		\begin{eqnarray}	
			\mathbb{E}\left(\sup_{0\leq t\leq T}|Y_{t}^{\varepsilon}|^p\right)\le\frac{C_{p}(1+|y|^p)T}{\varepsilon}.    \label{Y.1}
		\end{eqnarray}
	\end{lemma}
	\begin{proof}
		Denote $\tilde{Y}_{t}^{\vare}:=Y_{ t\varepsilon}^{\vare}$, then it is easy to check that $\tilde{Y}^{\vare}$  satisfies the following equation:
		\begin{eqnarray*}
			\tilde{Y}^{\varepsilon}_{t}=\!\!\!\!\!\!\!\!&&y+\frac{1}{\varepsilon}\int^{t\varepsilon}_{0}f(X_{s}^{\varepsilon},Y^{\varepsilon}_s) ds
			+\frac{1}{\sqrt{\varepsilon}}\int^{t\varepsilon}_{0} g(X_{s}^{\varepsilon},Y^{\varepsilon}_s)dW^2_s+\int^{t\varepsilon}_{0}\int_{\mathcal{Z}_2} h_{2}(X_{s-}^{\varepsilon}, Y^{\varepsilon}_{s-},z)\tilde{N}^{2,\vare}(ds,dz)\nonumber\\
			=\!\!\!\!\!\!\!\!&&y+\int^{t}_{0} f(X_{s\vare}^{\varepsilon},\tilde{Y}^{\varepsilon}_s) ds+\int^{t}_{0} g(X_{s\vare}^{\varepsilon},\tilde{Y}^{\varepsilon}_s)d\tilde{W}^{2}_s
			+\int^{t}_{0}\int_{\mathcal{Z}_2}h_{2}(X_{s\vare-}^{\varepsilon},\tilde{Y}^{\varepsilon}_{s-},z)\tilde{N}^{2}(ds,dz),
		\end{eqnarray*}
		where $\tilde{W}_t^2:=\frac{1}{\sqrt{\vare}}W^2_{t\vare}$ that coincides in law with $W^2_t$, $\tilde{N}^{2}(ds,dz)$ is also a compensated Poisson random measure with L\'evy measure $\nu_2$, which is constructed by the Poisson point process $\tilde p^{2,\vare}_t:=p^{2,\vare}_{t\vare}$ with $t\in \tilde{D}:=\frac{1}{\vare} D_{2}=\{t>0, t\vare\in D_{2} \}$.
		
		By It\^{o} formula, for any $p\geq 4$, we have
		\begin{eqnarray}
			 |\tilde{Y}_t^\varepsilon|^p=\!\!\!\!\!\!&&|y|^p+p\int_0^t|\tilde{Y}_s^\varepsilon|^{p-2}\left\langle\tilde{Y}_s^\varepsilon,f(X_{s\varepsilon}^{\varepsilon},\tilde{Y}_{s}^{\varepsilon})\right\rangle ds+p\int_0^t|\tilde{Y}_s^\varepsilon|^{p-2}\left\langle\tilde{Y}_s^\varepsilon,g(X_{s\varepsilon}^{\varepsilon},\tilde{Y}_{s}^{\varepsilon})d\tilde{W}_s^2\right\rangle \nonumber \\ &&+{\frac{p(p-2)}{2}}\int_0^t|\tilde{Y}_s^\varepsilon|^{p-4}| g^{*}(X_{s\varepsilon}^{\varepsilon},\tilde{Y}_{s}^{\varepsilon})\cdot \tilde{Y}_s^{\varepsilon} |^2ds+{\frac{p}{2}}\int_0^t|\tilde{Y}_s^\varepsilon|^{p-2}\|g(X_{s\varepsilon}^{\varepsilon},\tilde{Y}_{s}^{\varepsilon})\|^2ds  \nonumber \\
			 &&+\int_0^t\!\int_{\mathcal{Z}_2}\left[|\tilde{Y}_{s}^{\varepsilon}+h_2(X_{s\varepsilon}^{\varepsilon},\tilde{Y}_{s}^{\varepsilon},z)|^p\!-\!|\tilde{Y}_{s}^{\varepsilon}|^p\!-\!p|\tilde{Y}_{s}^{\varepsilon}|^{p-2}\left\langle\tilde{Y}_s^\varepsilon,h_2(X_{s\varepsilon}^{\varepsilon},\tilde{Y}_{s}^{\varepsilon},z)\right\rangle\right]\nu_2(dz)ds \nonumber \\
			 &&+\int_0^t\int_{\mathcal{Z}_2}\left[|\tilde{Y}_{s-}^{\varepsilon}+h_2(X_{s\varepsilon-}^{\varepsilon},\tilde{Y}_{s-}^{\varepsilon},z)|^p-|\tilde{Y}_{s-}^{\varepsilon}|^p\right]\tilde{N}^2(ds,dz).  \label{Y.2}
		\end{eqnarray}
		
		By the same argument as in the proof of \eref{Y1}, we can easily obtain that
		\begin{eqnarray}
			\sup_{t\ge0}\mathbb{E}|\tilde{Y}_{t}^{\varepsilon}|^p\le C_{p}(1+|y|^p). \label{Y.3}
		\end{eqnarray}
		
		Using (\ref{Y.2}), (\ref{ConB2}), (\ref{ConB32})  and Burkholder-Davis-Gundy's inequality (see \cite[Theorem 3.50]{PZ2007} ),  following the same argument as in the proof of \eref{X5} and \eref{X6},  we have
		\begin{eqnarray*}
			\EE\left(\sup_{0\leq t\leq T}|\tilde{Y}_{t}^{\varepsilon}|^p\right)\leq\!\!\!\!\!\!&&|y|^p+C_{p}T+C_p\int^T_0 \EE(1+|\tilde{Y}_{t}^{\varepsilon}|^p)dt\\
			&&+C_{p}\EE\left[\sup_{0\leq t\leq T}\left|\int_0^t |\tilde{Y}_{s}^{\varepsilon}|^{p-2}\left\langle\tilde{Y}_s^\varepsilon,g(X_{s\varepsilon}^{\varepsilon},\tilde{Y}_{s}^{\varepsilon})d\tilde{W}_s^2\right\rangle\right| \right]\\
			&&+C_{p}\EE\left[\sup_{0\leq t\leq T}\Big|\int_0^t\int_{\mathcal{Z}_2}\left(\left|\tilde{Y}_{s-}^{\varepsilon}+h_2(X_{s\varepsilon-}^{\varepsilon},\tilde{Y}_{s-}^{\varepsilon},z)\right|^p-\left|\tilde{Y}_{s-}^{\varepsilon}\right|^p\right)\tilde{N}^2(ds,dz)\Big|\right]\\
			\leq\!\!\!\!\!\!&&|y|^p+C_{p}T+C_p\int^T_0 \EE(1+|\tilde{Y}_{t}^{\varepsilon}|^p)dt\\
			&&+C_{p}\EE\left[\int_0^T |\tilde{Y}_{s}^{\varepsilon}|^{2p-2}\|g(X_{s\varepsilon}^{\varepsilon},\tilde{Y}_{s}^{\varepsilon})\|^2 ds\right]^{1/2}\\
			 &&+C_{p}\EE\left[\int_0^T\int_{\mathcal{Z}_2}\left(\left|\tilde{Y}_{s-}^{\varepsilon}+h_2(X_{s\varepsilon-}^{\varepsilon},\tilde{Y}_{s-}^{\varepsilon},z)\right|^p-\left|\tilde{Y}_{s-}^{\varepsilon}\right|^p\right)^2 N^2(ds,dz)\right]^{1/2}\\
			\leq\!\!\!\!\!\!&&\frac{1}{2}\EE\left(\sup_{0\leq t\leq T}|\tilde{Y}_{t}^{\varepsilon}|^p\right)+|y|^p+C_{p}T+C_p\int^T_0 \EE(1+|\tilde{Y}_{t}^{\varepsilon}|^p)dt,
		\end{eqnarray*}
		which together with \eref{Y.3}, we finally have that for any $T\ge1$,
		\begin{eqnarray*}
			\EE\left(\sup_{0\leq t\leq T}|\tilde{Y}_{t}^{\varepsilon}|^p\right)\leq C_p(1+|y|^p)T.
		\end{eqnarray*}
		
		Hence, it follows that for any $T>0$ and $\varepsilon\in (0,T]$,
		\begin{eqnarray*}
			\mathbb{E}\left(\sup_{0\leq t\leq T}|Y_{t}^{\varepsilon}|^p\right)=\mathbb{E}\left(\sup_{0\leq t\leq\frac{T}{\varepsilon}}|\tilde{Y}_{t}^{\varepsilon}|^p\right)\le\frac{C_p(1+|y|^p)T}{\varepsilon}.
		\end{eqnarray*}
		The proof is complete.
	\end{proof}

	\subsection{The frozen equation}
	
	For fixed $x\in \RR^n$, recall the frozen equation:
	\begin{equation}\left\{\begin{array}{l} \label{FE}
			\displaystyle
			dY_{t}=f(x,Y_{t})dt+g(x,Y_{t})d \tilde{W}_{t}^{2}+\int_{\mathcal{Z}_2}h_{2}(x,Y_{t-},z)\tilde{N}^{2}(dz,dt),\\
			Y_{0}=y\in\RR^m.\\
		\end{array}\right.
	\end{equation}
	Under the assumptions \ref{B1}-\ref{B3}, it is easy to check that \eref{FE} admits a unique solution $\{Y^{x,y}_t\}_{t\geq 0}$. Moreover, following the same steps as in the proof of \eref{Y1}, we can easily obtain that for any $p>0$, there exist $\gamma_p,C_p>0$ such that
	\begin{eqnarray}
		\mathbb{E}|Y_{t}^{x,y}|^p\leq e^{-\gamma_p t}|y|^p+C_p,\quad \forall t\geq 0.\label{FroY}
	\end{eqnarray}
	
	\begin{lemma}\label{L3.3}
		There exists $\gamma>0$ such that for any $t \ge 0$, $x_1,x_2\in\RR^{n}$ and $y_1,y_2\in\RR^{m}$, we have
		\begin{eqnarray}
			&&\mathbb{E}|Y_t^{x_1,y_1}-Y_t^{x_2,y_2}|^{\ell}\le e^{-\gamma t }|y_1-y_2|^{\ell}+C|x_1-x_2|^{\ell},  \label{FR1}
		\end{eqnarray}
		where $\ell$ is the existing constant in condition \eref{ConB1}.
	\end{lemma}
	\begin{proof}
		For any $t \ge 0$, $x_i\in\RR^{n}$ and $y_i\in\RR^{m}$, $i=1, 2$, note that
		\begin{eqnarray*}
			d(Y^{x_1,y_1}_t-Y^{x_2,y_2}_t)=\!\!\!\!\!\!&&\left[f(x_1, Y^{x_1,y_1}_t)-f(x_2, Y^{x_2,y_2}_t)\right]dt+\big[g(x_1, Y^{x_1,y_1}_t)-g(x_2, Y^{x_2,y_2}_t)\big]d\tilde{W}_{t}^{2}\\
			&&+\int_{\mathcal{Z}_2} \big[h_{2}(x_1,Y_{t-}^{x_1,y_1},z) -h_{2}(x_2,Y_{t-}^{x_2,y_2},z)\big]\tilde{N}^{2}(dz,dt),
		\end{eqnarray*}
		with $Y^{x_1,y_1}_0\!-\!Y^{x_2,y_2}_0=y_1-y_2$.
		
		Using It\^o's formula and taking expectation on both sides, we get
		\begin{eqnarray*}
			&&\mathbb{E}\big|Y_t^{x_1,y_1}-Y_t^{x_2,y_2}\big|^{\ell} \\
			=\!\!\!\!\!\!\!\!\!\!&&|y_1-y_2|^{\ell}\!+\!\ell\mathbb{E}\!\int_{0}^{t}\big|Y_s^{x_1,y_1}-Y_s^{x_2,y_2}\big|^{\ell-2}\left\langle Y^{x_1,y_1}_s\!-\!Y^{x_2,y_2}_s, f(x_1, Y^{x_1,y_1}_s)\!-\!f(x_2,Y^{x_2,y_2}_s)\right\rangle ds\\
			&& +\frac{\ell(\ell-2)}{2}\mathbb{E}\int_{0}^{t} \big|Y_s^{x_1,y_1}-Y_s^{x_2,y_2}\big|^{\ell-4}\left |(g(x_1, Y^{x_1,y_1}_s)-g(x_2, Y^{x_2,y_2}_s))^{\ast} (Y_s^{x_1,y_1}-Y_s^{x_2,y_2})\right|^2 ds\\
			&& +\frac{\ell}{2}\mathbb{E}\int_{0}^{t} \big|Y_s^{x_1,y_1}-Y_s^{x_2,y_2}\big|^{\ell-2}\left \|g(x_1, Y^{x_1,y_1}_s)-g(x_2, Y^{x_2,y_2}_s) \right\|^2 ds\\
			&& +\mathbb{E}\!\int_{0}^{t}\!\int_{\mathcal{Z}_2}\Big[ |Y_s^{x_1,y_1}-Y_s^{x_2,y_2}+h_{2}(x_1,Y_{s}^{x_1,y_1},z)-h_{2}(x_2,Y_{s}^{x_2,y_2},z)|^{\ell}-|Y_s^{x_1,y_1}-Y_s^{x_2,y_2}|^{\ell}\\
			&&\quad-{\ell}|Y_s^{x_1,y_1}-Y_s^{x_2,y_2}|^{\ell-2}\langle Y_s^{x_1,y_1}-Y_s^{x_2,y_2} , h_{2}(x_1,Y_{s}^{x_1,y_1},z)-h_{2}(x_2,Y_{s}^{x_2,y_2},z)\rangle\Big]\nu_2(dz)ds.
		\end{eqnarray*}
Then by  \eref{Taylor}, (\ref{ConB1}), \eref{ConB11} and Young's inequality
		\begin{eqnarray*}
			&&\frac{d}{dt}\mathbb{E}\big|Y_t^{x_1,y_1}-Y_t^{x_2,y_2}\big|^{\ell}\\
			\leq\!\!\!\!\!\!\!\!\!\!&&{\ell}\mathbb{E}\left[\big|Y_t^{x_1,y_1}-Y_t^{x_2,y_2}\big|^{\ell-2}\left\langle Y^{x_1,y_1}_t\!-\!Y^{x_2,y_2}_t, f(x_1, Y^{x_1,y_1}_t)\!-\!f(x_2,Y^{x_2,y_2}_t)\right\rangle\right]\\
			&& +\frac{{\ell}({\ell}-1)}{2}\mathbb{E}\left[\big|Y_t^{x_1,y_1}-Y_t^{x_2,y_2}\big|^{\ell-2}\left \|g(x_1, Y^{x_1,y_1}_t)-g(x_2, Y^{x_2,y_2}_t) \right\|^2\right]\\
			&& +2^{\ell-4}\ell(\ell-1)\mathbb{E}\Big[|Y_t^{x_1,y_1}-Y_t^{x_2,y_2}|^{\ell-2} \int_{\mathcal{Z}_2}|h_{2}(x_1,Y_{t}^{x_1,y_1},z)-h_{2}(x_2,Y_{t}^{x_2,y_2},z)|^2\nu_2(dz)\Big]\\
			&&+2^{\ell-4}\ell(\ell-1)\mathbb{E} \int_{\mathcal{Z}_2}|h_{2}(x_1,Y_{t}^{x_1,y_1},z)-h_{2}(x_2,Y_{t}^{x_2,y_2},z)|^{\ell}\nu_2(dz)\\
			\leq\!\!\!\!\!\!\!\!\!\!&&\mathbb{E}\left[\big|Y_t^{x_1,y_1}-Y_t^{x_2,y_2}\big|^{\ell-2}(-\frac{\ell\beta}{2}\big|Y_t^{x_1,y_1}-Y_t^{x_2,y_2}\big|^2+C|x_1-x_2|^2)\right]\\
			&&\quad\quad+\frac{\ell L_{h_2}}{2}\EE\big|Y_t^{x_1,y_1}-Y_t^{x_2,y_2}\big|^{\ell}+C|x_1-x_2|^{\ell}\\
			\leq\!\!\!\!\!\!\!\!\!\!&&-\frac{\ell(\beta-L_{h_2})}{4}\mathbb{E}\big|Y_t^{x_1,y_1}-Y_t^{x_2,y_2}\big|^{\ell}+C|x_1-x_2|^{\ell}.
		\end{eqnarray*}
		The comparison theorem yields that
		\begin{eqnarray*}
			\mathbb{E}|Y_t^{x_1,y_1}-Y_t^{x_2,y_2}|^{\ell}\le e^{-\gamma t }|y_1-y_2|^{\ell}+C|x_1-x_2|^{\ell},
		\end{eqnarray*}
		where $\gamma:=\frac{\ell(\beta-L_{h_2})}{4}$. The proof is complete.
	\end{proof}
	
	Denote $P^{x}_t$ be the transition semigroup of $\{Y^{x,y}_t\}_{t\geq 0}$, i.e., for proper measurable function $\varphi$ on $\RR^m$,
	\begin{eqnarray*}
		P^x_t \varphi(y):=\EE\left[\varphi\left(Y_{t}^{x,y}\right)\right], \quad y \in \RR^m,\ \ t>0.
	\end{eqnarray*}
	Then the following exponential ergodicity holds:
	\begin{proposition}\label{Ergodicity}
		Under the assumptions \ref{B1}-\ref{B3}. Then $\{P^x_t\}_{t\geq 0}$ admits a unique invariant measure $\mu^x$, and for any $p>0$ there exists $C_p>0$ such that
	\begin{eqnarray}
		\sup_{x\in\RR^n}\int_{\RR^m}|y|^p\mu^x(dy)\leq C_p.\label{F3.17}
	\end{eqnarray}
Moreover, for any $t \ge 0$, $x\in\RR^{n}$, and $y\in\RR^{m}$, we have
		\begin{eqnarray}
			&&\big|\mathbb{E}b(x,Y_t^{x,y})-\bar{b}(x)\big|\le Ce^{-\frac{\gamma t}{\ell}}(1+|y|^{k+1}).   \label{FR2}
		\end{eqnarray}
	\end{proposition}
	\begin{proof}
(i)  We shall use the classical Krylov-Bogoliubov method to prove the existence of {\color{red}an} invariant measure. More precisely, for any $n\in N$, define the Krylov-Bogoliubov measure $$\mu_n^x:=
\frac{1}{n} \int_0^n \delta_0P^x_tdt,~~n\geq 1,$$ where $\delta_0$ is
the Dirac measure at $0$, then each $\mu_n^x$ is a probability measure such that for any bounded measurable function $\varphi$,
$$\int_{\mathbb{R}^n}\varphi (dy)\mu_n^x(dy)=\frac{1}{n}\int_0^nP^x_t\varphi(0)dt.
$$
For any constant $R>0$, define $K_{R}:=\{z\in \mathbb{R}^n:|z|\leq R\}$, which is a compact set in $\RR^n$. Meanwhile, by Chebyshev's inequality and \eref{FroY},
 we obtain
\begin{eqnarray*}
\mu^x_n(K^{c}_{R})\leq
\frac{1}{n R^2}\int_0^n\mathbb{E}|Y_{s}^{x,0}|^2ds\leq\frac{C(1+|x|^2)}{R^2},\quad n\geq1,
\end{eqnarray*}
which implies that $\{\mu^x_n\}_{n\geq 1}$ is tight. Then by Prokhorov's theorem there exists a probability measure $\mu^x$ and a subsequence $(\mu^x_{n_k})_{k\in N}$
 such that $\mu^x_{n_k}\rightarrow\mu^x$ weakly as $k\rightarrow\infty$. Furthermore, it is easy to check that $\mu^x$ is an invariant measure of the transition semigroup $\{P^x_t\}_{t\geq 0}$ by a standard argument.

(ii)  By \eref{FroY} and the definition of invariant measure, we have for any $x\in\RR^n$ and $t>0$,
\begin{eqnarray*}
	\int_{\RR^m}|y|^p\mu^x(dy)=\!\!\!\!\!\!\!\!&&\int_{\RR^m}\EE|Y^{x,y}_t|^p\mu^x(dy)\nonumber\\
\leq \!\!\!\!\!\!\!\!&&\int_{\RR^m}(e^{-\gamma_p t}|y|^p+C_p)\mu^x(dy),
	\end{eqnarray*}	
which implies
\begin{eqnarray*}
	\int_{\RR^m}|y|^p\mu^x(dy)\leq \frac{C_p}{1-e^{-\gamma_p t}}.
	\end{eqnarray*}	
Taking $t\rightarrow \infty$, we get
\begin{eqnarray}
	\int_{\RR^m}|y|^p\mu^x(dy)\leq C_p.\label{F3.172}
	\end{eqnarray}
(iii) For any Lipschitz continuous function $\varphi$, by the definition of invariant measure $\mu^x$, Lemmas \ref{L3.3} and (\ref{F3.172}) , we have for any $t\geq 0$,
\begin{eqnarray*}
\Big|P^x_t\varphi(y)-\int_{\RR^m}\varphi(y')\mu^x(d{y'})\Big|=\!\!\!\!\!\!\!\!&&\Big|\int_{\RR^m} \mathbb{E}\left[\varphi(Y_{t}^{x,y})-\varphi(Y_{t}^{x,{y'}})\right]\mu^x(dy')\Big|
\nonumber\\
\leq\!\!\!\!\!\!\!\!&&\|\varphi\|_{Lip}\int_{\RR^m} \mathbb{E}|Y_{t}^{x,y}-Y_{t}^{x,{y'}}|\mu^x(dy')
\nonumber\\
\leq\!\!\!\!\!\!\!\!&&\|\varphi\|_{Lip}e^{-\frac{\gamma }{\ell} t}\int_{\RR^m} |y-y'|\mu^x(dy')
\nonumber\\
\leq\!\!\!\!\!\!\!\!&&C\|\varphi\|_{Lip}(1+|y|)e^{-\frac{\gamma }{\ell} t},
\end{eqnarray*}
where $\|\varphi\|_{Lip}:=\sup_{x\neq y\in \RR^n}\frac{|\varphi(x)-\varphi(y)|}{|x-y|}$. Now, let $\nu^x$ be another invariant measure, by the same argument above, we have for any $t\geq 0$,
\begin{eqnarray*}
&&\Big|\int_{\RR^m}\varphi(y')\nu^x(d{y'})-\int_{\RR^m}\varphi(y')\mu^x(d{y'})\Big|\\
\leq\!\!\!\!\!\!\!\!&&\Big|P^x_t\varphi(y)-\int_{\RR^m}\varphi(y')\nu^x(d{y'})\Big|+\Big|P^x_t\varphi(y)-\int_{\RR^m}\varphi(y')\mu^x(d{y'})\Big|\\
\leq\!\!\!\!\!\!\!\!&&C\|\varphi\|_{Lip}(1+|y|)e^{-\frac{\gamma }{\ell} t}.
\end{eqnarray*}
Taking $t\rightarrow \infty$ and by the arbitrariness of the $\varphi$, we have $\mu^x=\nu^x$. Hence $\mu^x$ is the unique invariant measure.

Furthermore, by the definition of invariant measure $\mu^x$, \eref{ConA112}, Lemmas \ref{L3.3} and (\ref{F3.172}), we have
		\begin{eqnarray*}
			\big|\mathbb{E}b(x,Y_t^{x,y})-\bar{b}(x)\big|\!\!\!\!\!\!\!\!&&=\left| \mathbb{E}b(x,Y_t^{x,y})-\int_{\RR^{m}}b(x,z)\mu^{x}(dz)\right|\\
			\!\!\!\!\!\!\!\!&&\leq \int_{\RR^{m}}\mathbb{E}\left| b(x,Y^{x,y}_t)-b(x,Y^{x,z}_t)\right|\mu^{x}(dz)\\
			\!\!\!\!\!\!\!\!&&\leq C\int_{\RR^{m}} \big[\mathbb{E}\big(1+|Y^{x,y}_t|^{2k}+|Y^{x,z}_t|^{2k}\big)\big]^{\frac{1}{2}} \big(\mathbb{E}\left| Y^{x,y}_t-Y^{x,z}_t\right|^{2}\big)^{\frac{1}{2}}\mu^{x}(dz)\\
			\!\!\!\!\!\!\!\!&&\leq C e^{-\frac{\gamma t}{\ell}}\int_{\RR^{m}}(1+|y|^{k}+|z|^{k})|y-z|\mu^{x}(dz)\\
			\!\!\!\!\!\!\!\!&&\leq C e^{-\frac{\gamma t}{\ell}}(1+|y|^{k+1}).
		\end{eqnarray*}
The proof is complete.
	\end{proof}
	
	\subsection{The averaged equation}
	
	In the case of $\sigma(x,y)=\sigma(x)$, recall the averaged equation as follows:
	\begin{equation}\left\{\begin{array}{l}
			\displaystyle d\bar{X}_{t}=\bar{b}( \bar{X}_{t})dt+\sigma(\bar X_t)dW_t^1+\int_{\mathcal{Z}_1}h_1(\bar{X}_{t-},z)\tilde{N}^1(dz,dt),\\
			\bar{X}_{0}=x\in \RR^{n},\end{array}\right. \label{AR1}
	\end{equation}
	where $\bar{b}(x)=\int_{\RR^{m}}b(x,z)\mu^{x}(dz)$ and $\mu^{x}$ is the unique invariant measure for the transition semigroup of  equation (\ref{FE}).
	
	\begin{lemma} \label{PMA}
		For any $x\in\RR^{n}$, equation (\ref{AR1}) has a unique solution $\bar{X}_{t}$. Moreover, for any $T>0$ we have
		\vspace{0.1cm}
		\begin{eqnarray}
			\mathbb{E}\left(\sup_{0\leq t\leq T}|\bar{X}_{t}|^{p}\right)\leq C_{p,T}(1+|x|^{p}).\label{AR4}
		\end{eqnarray}
	\end{lemma}
	\begin{proof}
		By  \ref{A1}, \eref{FroY}, Lemma \ref{L3.3}  and Proposition  \ref{Ergodicity} , for any $t>0$, $x_1,x_2\in\RR^{n}$, we have
		\vspace{0.1cm}
		\begin{eqnarray*}
			&&2\langle \bar{b}(x_1)-\bar{b}(x_2),x_1-x_2\rangle + \|\sigma (x_{1})-\sigma (x_{2})\|^{2}+\int_{\mathcal{Z}_1}\big|h_{1}(x_{1},z)-h_{1}(x_{2},z)\big|^{2}\nu_1(dz)\\
			\le \!\!\!\!\!\!\!&&2\langle \bar{b}(x_1)-\mathbb{E} b(x_1,Y_{t}^{x_1,0}),x_1-x_2 \rangle+2\langle\mathbb{E} b(x_2,Y_{t}^{x_2,0})-\bar{b}(x_2),x_1-x_2\rangle\\
			&&+2\langle\mathbb{E} b(x_1,Y_{t}^{x_1,0})-\mathbb{E} b(x_2,Y_{t}^{x_1,0}),x_1-x_2\rangle+ \|\sigma (x_{1})-\sigma (x_{2})\|^{2}\\
			&&+\int_{\mathcal{Z}_1}\big|h_{1}(x_{1},z)-h_{1}(x_{2},z)\big|^{2}\nu_1(dz)+2\langle\mathbb{E} b(x_2,Y_{t}^{x_1,0})-\mathbb{E} b(x_2,Y_{t}^{x_2,0}),x_1-x_2\rangle\\
			\le \!\!\!\!\!\!\!&&Ce^{-\frac{\gamma t}{\ell}}\left|x_1-x_2\right|+C\left|x_1-x_2\right|^{2}+C\left|x_1-x_2\right|\mathbb{E}\big[(1+|Y_{t}^{x_1,0}|^{k}+|Y_{t}^{x_2,0}|^{k})|Y_{t}^{x_1,0}-Y_{t}^{x_2,0}|\big]\\
			\le \!\!\!\!\!\!\!&& Ce^{-\frac{\gamma t}{\ell}}|x_1\!-\!x_2|\!+\!C|x_1\!-\!x_2|^{2}\!+\!C|x_1\!-\!x_2|\big[\mathbb{E}(1+|Y_{t}^{x_1,0}|^{2k}+|Y_{t}^{x_2,0}|^{2k})\big]^{\frac{1}{2}}\big[\mathbb{E}|Y_{t}^{x_1,0}-Y_{t}^{x_2,0}|^{2}\big]^{\frac{1}{2}} \\
			\le \!\!\!\!\!\!\!&& Ce^{-\frac{\gamma t}{\ell}}\left|x_1-x_2\right|+C\left|x_1-x_2\right|^{2},
		\end{eqnarray*}
		Letting $t\rightarrow \infty$, we get that
		\begin{eqnarray}
			&&2\langle \bar{b}(x_{1})-\bar{b}(x_{2}),x_{1}-x_{2}\rangle+ \|\sigma (x_{1})-\sigma (x_{2})\|^{2}\nonumber\\
			&&\quad\quad\quad+\int_{\mathcal{Z}_1}\big|h_{1}(x_{1},z)-h_{1}(x_{2},z)\big|^{2}\nu_1(dz) \leq    C|x_{1}-x_{2}|^{2}.  \label{Avc1}
		\end{eqnarray}
		
		By \eref{ConA2},  \eref{ConA32}, \eref{FroY} and \eref{FR2}, we have
		\begin{eqnarray*}
			&&2\langle \bar{b}(x), x\rangle+ \|\sigma (x)\|^{2}+\int_{\mathcal{Z}_1}\big|h_{1}(x,z)\big|^{2}\nu_1(dz)\\
			\le \!\!\!\!\!\!\!&& 2\langle \bar{b}(x)-\mathbb{E} b(x,Y_{t}^{x,0}), x\rangle+ 2\langle \mathbb{E} b(x,Y_{t}^{x,0}), x\rangle+\|\sigma (x)\|^{2}+C(1+|x|^2)\\
			\le \!\!\!\!\!\!\!&&Ce^{-\frac{\gamma t}{\ell}}|x|+C(1+|x|^2)+C\EE|Y_{t}^{x,0}|^q\\
\le \!\!\!\!\!\!\!&&Ce^{-\frac{\gamma t}{\ell}}|x|+C(1+|x|^2).
		\end{eqnarray*}
		Letting $t\rightarrow \infty$, we get that
		\begin{eqnarray}
			2\langle \bar{b}(x), x\rangle+ \|\sigma (x)\|^{2}+\int_{\mathcal{Z}_1}\big|h_{1}(x,z)\big|^{2}\nu_1(dz) \leq    C(1+|x|^2). \label{Avc2}
		\end{eqnarray}
		
		Hence, \eref{Avc1} and \eref{Avc2} imply that equation (\ref{AR1}) has a unique solution $\bar{X}_{t}$ (see  \cite[Theorem 2.8]{XZ2019}).  Moreover, estimate (\ref{AR4}) can be easily obtained by a similar argument as in the proof of (\ref{X}). The proof is complete.
	\end{proof}

	\section{Poisson equation with polynomial growth coefficients}
	Consider the following Poisson equation:
	\begin{eqnarray}\label{PE}
		-\mathscr{L}_{2}(x)\Phi(x,\cdot)(y)= b(x,y)-\bar{b}(x),
	\end{eqnarray}
	which is equivalent to for any $k=1,\ldots,n$,
	\begin{eqnarray}\label{PEQ}
		-\mathscr{L}_{2}(x)\Phi_k(x,\cdot)(y)= b_k(x,y)-\bar{b}_k(x),
	\end{eqnarray}
	where $\Phi(x,y):=(\Phi_1(x,y),\ldots, \Phi_n(x,y))$ and $\mathscr{L}_{2}(x)$ is the infinitesimal generator of the frozen process $\{Y^{x,y}_{t}\}$, i.e., $\phi\in C^2(\RR^m)$,
	\begin{eqnarray*}
		\mathscr{L}_2(x)\phi(y):=\!\!\!\!\!\!\!&& \langle f(x,y), \partial_y  \phi(y)\rangle+\frac{1}{2}\text{Tr}[gg^{*}(x,y)\partial_y^{2}\phi(y)]\\
		&&+\int_{\mathcal{Z}_2}\left[\phi(y+h_{2}(x,y,z))-\phi(y)-\langle \partial_y  \phi(y),h_{2}(x,y,z)\rangle \right]\nu_2(dz).
	\end{eqnarray*}
Now, we can obtain the existence and regularity estimates of the solution of the Poisson equation \eref{PE}.

	\begin{proposition}\label{P4.1}
	Under the assumptions in Theorem \ref{main result 1}. Define
	\begin{eqnarray}
		\Phi(x,y):=\int^{\infty}_{0}\left[\EE b(x,Y^{x,y}_t)-\bar{b}(x)\right]dt. \label{SPE}
	\end{eqnarray}
Then $\Phi(x,y)$ is a solution of the Poisson equation (\ref{PE}).	Moreover, there exists $C>0$ such that for any $x\in\RR^n, y\in \RR^m$,
		\begin{eqnarray}
			\sup_{x\in\RR^{n}}|\Phi(x,y)|\leq C(1+|y|^{k+1   }),\label{E1}
		\end{eqnarray}
		\begin{eqnarray}
			\sup_{x\in\RR^{n}}\|\partial_y \Phi(x,y)\|\leq C(1+|x|^k+|y|^{k}),\label{E2}
		\end{eqnarray}
		\begin{eqnarray}
			\|\partial_x \Phi(x,y)\|\leq C(1+|x|^{2k+1}+|y|^{2k+1}),\label{E3}
		\end{eqnarray}
		\begin{eqnarray}
			\|\partial_{x}^2 \Phi(x,y)\|\leq C(1+|x|^{3k+1}+|y|^{3k+1}).\label{E4}
		\end{eqnarray}
	\end{proposition}
	\begin{proof}
		We will divide the proof into three steps.
		
		\textbf{Step 1:} In this step, we intend to prove (\ref{E1})-(\ref{E2}). By Proposition \ref{Ergodicity}, we have
		\begin{eqnarray*}
			\left| \Phi (x,y) \right|\le\!\!\!\!\!\!\!&& \int_{0}^{\infty}\left| \mathbb{E}b(x,Y_{t}^{x,y})-\bar{b}(x)\right |dt\\
			\le\!\!\!\!\!\!\!&&C(1+\left| y \right|^{k+1})\int_{0}^{\infty} e^{-\frac{\gamma t}{\ell}}dt\\
			\le\!\!\!\!\!\!\!&&C(1+\left | y \right |^{k+1}).
		\end{eqnarray*}
Thus the notation $\Phi$ is well-defined. Then we have for any $s>0$,
		\begin{eqnarray*}
			\frac{P^x_s\Phi (x,\cdot)(y)-\Phi (x,y)}{s}=\!\!\!\!\!\!\!&& \frac{\EE\Phi (x,Y^{x,y}_s)-\Phi (x,y)}{s}\\
=\!\!\!\!\!\!\!&&\frac{1}{s}\left\{\int_{0}^{\infty}\left[\mathbb{E}b(x,Y_{t+s}^{x,y})-\bar{b}(x)\right]dt-\Phi (x,y)\right\}\\
=\!\!\!\!\!\!\!&&\frac{1}{s}\left\{\int_{s}^{\infty}\left[\mathbb{E}b(x,Y_{t}^{x,y})-\bar{b}(x)\right]dt-\int_{0}^{\infty}\left[\mathbb{E}b(x,Y_{t}^{x,y})-\bar{b}(x)\right]dt\right\}\\
			=\!\!\!\!\!\!\!&&-\frac{1}{s}\int_{0}^{s}\left[\mathbb{E}b(x,Y_{t}^{x,y})-\bar{b}(x)\right]dt,
		\end{eqnarray*}
letting $s\rightarrow 0$, this implies $-\mathscr{L}_{2}(x)\Phi(x,\cdot)(y)= b(x,y)-\bar{b}(x)$, thus $\Phi(x,y)$ is a solution of Poisson equation (\ref{PE}).

		Note that the chain rule implies that for any unit $l\in \RR^m$,
		$$\partial_y \Phi(x,y)\cdot l=\int^{\infty}_0 \EE[\partial_y b(x,Y^{x,y}_t)\cdot (\partial_y Y^{x,y}_t\cdot l)]dt,$$
		Refer to the Proposition \ref{DFY} in the appendix, here $\partial_y Y^{x,y}_t\cdot l$ is the directional derivative of $Y^{x,y}_t$ with respect to $y$ in the direction $l$, which satisfies
		\begin{eqnarray}
			d[\partial _yY_{t}^{x,y}\cdot l]=\!\!\!\!\!\!\!\!\!&&\partial_yf(x,Y_{t}^{x,y})\cdot(\partial_yY_{t}^{x,y}\cdot l)dt+\partial_yg(x,Y_{t}^{x,y})\cdot (\partial_yY_{t}^{x,y}\cdot l)d\tilde{W}_{t}^{2} \nonumber\\
			&&+\int_{\mathcal{Z}_2}\partial_yh_2(x,Y_{t-}^{x,y},z) \cdot (\partial_yY_{t-}^{x,y}\cdot l)\tilde{N}^{2}(dz,dt).\label{partial y}
		\end{eqnarray}
		Moreover, the following estimate holds:
		\begin{eqnarray}
			\mathbb{E}|\partial_y Y^{x,y}_t\cdot l|^{\ell} \leq Ce^{-\gamma t}, \label{ParyY}
		\end{eqnarray}
where $C,\gamma>0$ and $\ell$ is the constant in assumption \ref{B1}.
		
Using \eref{ConA31}  and H\"older's inequality, we get
		\begin{eqnarray*}
			|\partial_y \Phi(x,y)\cdot l|\!\!\!\!\!\!\!&& \leq \int^{\infty}_{0}\EE \| \partial_y b(x,Y^{x,y}_t)\cdot (\partial_y Y^{x,y}_t\cdot l)\| dt \\
			&& \leq C\int^{\infty}_{0} \big[\EE (1+|x|^{2k}+|Y^{x,y}_t|^{2k}) \big]^{\frac{1}{2}}  \big[\EE |\partial_y Y^{x,y}_t\cdot l|^{2} \big]^{\frac{1}{2}} dt  \\
			&& \leq  C(1+|x|^k+|y|^{k})\int^{\infty}_{0} e^{-\gamma t}dt \\
			&& \leq C(1+|x|^k+|y|^{k}).
		\end{eqnarray*}
		
		Now, we define
		\begin{eqnarray*}
			\tilde b_{t_0}(x, y, t):=\hat b(x,y, t)-\hat b(x, y, t+t_0),
		\end{eqnarray*}
		where $\hat b(x, y, t):=\EE b(x, Y^{x, y}_t)$. Proposition \ref{Ergodicity} implies
		$$\lim_{t_0\rightarrow +\infty} \tilde b_{t_0}(x, y, t)=\EE b(x,Y^{x,y}_t)-\bar{b}(x).$$
		In order to prove \eref{E3} and \eref{E4}, it is sufficient to prove that for any unit vectors $l_1,l_2\in\RR^n$, $t_0>0$, $ t> 0$, $x\in\RR^{n}$ and $y\in\RR^{m}$, there exists $\gamma>0$ such that
		\begin{eqnarray}
			| \partial _x\tilde{b}_{t_0}(x,y,t)\cdot l_1|\le Ce^{-\gamma t}(1+ | x |^{2k+1}+|y|^{2k+1}), \label{E3.}
		\end{eqnarray}
		\begin{eqnarray}
			| \partial _{x}^2\tilde{b}_{t_0}(x,y,t)\cdot (l_1,l_2)|\le Ce^{-\gamma t}(1+|x|^{3k+1}+|y|^{3k+1}), \label{E4.}
		\end{eqnarray}
		which will be proved in Step 2 and Step 3, respectively.
		\vspace{0.3cm}
		
		\textbf{Step 2:} In this step, we indent to prove \eref{E3.}. By the Markov property,
		\begin{eqnarray*}
			\tilde b_{t_0}(x,y, t)=\!\!\!\!\!\!\!\!&& \hat b(x, y, t)-\EE b(x, Y^{x,y}_{t+t_0})\nonumber\\
			=\!\!\!\!\!\!\!\!&& \hat b(x,y, t)-\EE \big[\EE[b(x,Y^{x,y}_{t+t_0})|\mathscr{F}_{t_0}]\big]\nonumber\\
			=\!\!\!\!\!\!\!\!&& \hat b(x, y, t)-\EE \hat b(x, Y^{x,y}_{t_0},t),
		\end{eqnarray*}
		which implies
		\begin{eqnarray}
			\partial_{x}\tilde b_{t_0}(x,y, t)\cdot l_1=\!\!\!\!\!\!\!\!&&  [\partial_{x} \hat b(x, y, t)\cdot l_1-\EE \partial_{x}\hat {b}(x, Y^{x,y}_{t_0},t)\cdot l_1]\nonumber\\
			&&- \EE \left[\partial_y\hat b(x,Y^{x,y}_{t_0},t)\cdot (\partial_{x} Y^{x,y}_{t_0}\cdot l_1) \right],\nonumber\\
			=:\!\!\!\!\!\!\!\!\!&&I_1+I_2.\label{partial x}\label{F4.12}
		\end{eqnarray}
		Refer to the Proposition \ref{DFY} in the appendix, here $\partial_x Y^{x,y}_t\cdot l_1$ is the directional derivative of $Y^{x,y}_t$ with respect to $x$ in the direction $l_1$, which satisfies
		\begin{eqnarray*}
			d\left[\partial _xY_{t}^{x,y}\cdot l_1\right]=\!\!\!\!\!\!\!\!\!&&\left[\partial_x f(x,Y_{t}^{x,y})\cdot l_1+\partial_yf(x,Y_{t}^{x,y})\cdot(\partial_xY_{t}^{x,y}\cdot l_1)\right]dt\nonumber\\
			&&+ \left[\partial_x g(x,Y_{t}^{x,y})\cdot l_1+\partial_yg(x,Y_{t}^{x,y})\cdot( \partial_xY_{t}^{x,y}\cdot l_1)\right]d \tilde{W}_{t}^{2} \nonumber\\
			&&+\int_{\mathcal{Z}_2}\left[\partial_x h_2(x,Y_{t-}^{x,y},z)\cdot l_1+\partial_y h_2(x,Y_{t-}^{x,y},z) \cdot (\partial_x Y_{t-}^{x,y}\cdot l_1)\right]\tilde{N}^{2}(dt,dz).
		\end{eqnarray*}
		Meanwhile, the following estimate holds:
		\begin{eqnarray}
			\sup_{t\geq 0}\mathbb{E}|\partial_x Y^{x,y}_t\cdot l_1|^{\ell} \leq C, \label{ParxY}
		\end{eqnarray}
		where $\ell$ is the constant in assumption \ref{B1}.

		Note that
		\begin{eqnarray*}
			\partial_x \hat{b}(x,y,t)\cdot l_1=\EE\left[\partial _xb(x,Y_{t}^{x,y})\cdot l_1\right]+\EE[\partial _yb(x,Y_{t}^{x,y})\cdot (\partial_xY_{t}^{x,y}\cdot l_1)]  ,
		\end{eqnarray*}
		which implies that for any unit vectors $l_1\in\RR^n, l_2\in\RR^m$,
		\begin{eqnarray*}
			\partial_y\partial_x \hat{b}(x,y,t)\cdot (l_1,l_2)=\!\!\!\!\!\!\!\!&&\EE\left[\partial_y\partial _x b(x,Y_{t}^{x,y})\cdot(l_1, \partial_yY_{t}^{x,y}\cdot l_2)\right]\\
			&&+\EE\left[\partial ^2_yb(x,Y_{t}^{x,y})\cdot (\partial_xY_{t}^{x,y}\cdot l_1, \partial_y Y_t^{x,y}\cdot l_2)\right]  \\
			&&+\EE\left[\partial_yb(x,Y_{t}^{x,y})\cdot (\partial_y\partial_xY_{t}^{x,y}\cdot (l_1,  l_2))\right].
		\end{eqnarray*}
		 where $\partial_y\partial_x Y^{x,y}_t\cdot (l_1,l_2)$ is the directional derivative of $\partial_x Y^{x,y}_t\cdot l_1$ with respect to $y$ in the direction $l_2$.	Refer to Remark \ref{R6.2} in the appendix, we have
		\begin{eqnarray}
			\EE|\partial_y\partial_{x}Y^{x,y}_t\cdot (l_1,l_2)|^4\leq Ce^{-4\gamma t}(1+|y|^{4k}),\label{ParyxY}
		\end{eqnarray}
		which combines with \eref{ParyY} and \eref{ParxY}. It follows
		\begin{eqnarray}
			|\partial_y\partial_x \hat{b}(x,y,t)\cdot (l_1,l_2)|\leq C(1+|x|^{2k}+|y|^{2k})e^{-\gamma t}.\label{Paryxhatb}
		\end{eqnarray}
		Then we can obtain
		\begin{eqnarray}
			|I_1|\leq\!\!\!\!\!\!\!\!&& Ce^{-\gamma t}\EE\big[(1+|x|^{2k}+|y|^{2k}+|Y^{x,y}_{t_0}|^{2k}) | y-Y^{x,y}_{t_0}|\big] \nonumber \\
			\leq \!\!\!\!\!\!\!\!&&C e^{-\gamma t}(1+|y|^{2k+1}+|x|^{2k+1}). \label{R6}
		\end{eqnarray}
		
		Note that
		\begin{eqnarray}
			\partial_{y}\hat{b} (x, y, t)\cdot l_1=\EE\left[\partial_y b(x,Y^{x,y}_t) \cdot(\partial_y Y^{x,y}_t\cdot l_1)\right],
		\end{eqnarray}
		which together with \eref{ConA31} and \eref{ParyY}, we have
		\begin{eqnarray}
			|\partial_{y}\hat{b} (x, y, t)\cdot l_1|=\!\!\!\!\!\!\!\!&&|\EE[\partial_y b(x,Y^{x,y}_t) \cdot(\partial_y Y^{x,y}_t\cdot l_1) ]|\nonumber\\
			\leq\!\!\!\!\!\!\!\!&&C\left[\EE(1+|x|^{2k}+|Y^{x,y}_t|^{2k})\right]^{1/2}\left[\EE|\partial_y Y^{x,y}_t\cdot l_1 |^2\right]^{1/2}\nonumber\\
			\leq\!\!\!\!\!\!\!\!&&C e^{-\gamma t}(1+|x|^{k}+|y|^{k}). \label{R4}
		\end{eqnarray}
		Then by \eref{ParxY} and \eref{R4}, we get
		\begin{eqnarray}
			|I_2| \leq C e^{-\gamma t}(1+|x|^k+|y|^{k}). \label{R5}
		\end{eqnarray}
		
		Combining \eref{R6} and \eref{R5}, it follows
		\begin{eqnarray*}
			\left \| \partial _x\tilde{b}_{t_0}(x,y,t)\right \|\le Ce^{-\gamma t}(1+ | x |^{2k+1}+|y|^{2k+1}).
		\end{eqnarray*}

		\textbf{Step 3:} In this step, we indent to prove \eref{E4.}. Recall that \eref{F4.12}, then the chain rule yields
		\begin{eqnarray*}
			\partial_{x}^{2}\tilde{b}_{t_0}(x,y,t)\cdot(l_1,l_2)=\!\!\!\!\!\!\!&&[\partial_{x}^{2}\hat{b}(x,y,t)\cdot (l_1,l_2)-\EE\partial_{x}^{2}\hat{b}(x,Y_{t_0}^{x,y},t)\cdot(l_1,l_2)]\\
			&&-\EE[\partial_y\partial_{x}\hat{b}(x,Y_{t_0}^{x,y},t)\cdot (l_1, \partial_x Y_{t_0}^{x,y}\cdot l_2)]\\
			&&-\EE[\partial_x\partial_{y}\hat{b}(x,Y_{t_0}^{x,y},t)\cdot (\partial_x Y_{t_0}^{x,y}\cdot l_1, l_2)]\\
			&&-\EE[\partial_{y}^{2}\hat{b}(x,Y_{t_0}^{x,y},t)\cdot (\partial _xY_{t_0}^{x,y}\cdot l_1, \partial _xY_{t_0}^{x,y}\cdot l_2 )]\\
			&&-\EE[\partial _y\hat{b}(x,Y_{t_0}^{x,y},t)\cdot (\partial _{x}^{2}Y_{t_0}^{x,y}\cdot (l_1,l_2))]\\
			=:\!\!\!\!\!\!\!&&\sum_{i=1}^{5}J_i .
		\end{eqnarray*}
		
		For the term $J_1$. Note that
		\begin{eqnarray*}
			\partial_{x}^{2} \hat{b}(x,y,t)\cdot(l_1,l_2)=\!\!\!\!\!\!\!&&\EE\left[\partial_{x}^{2} b(x,Y_{t}^{x,y})\cdot (l_1,l_2)\right]+\EE\left[\partial_y\partial_{x}b(x,Y_{t}^{x,y})\cdot (l_1, \partial_x Y_{t}^{x,y}\cdot l_2)\right]\\
			&&+\EE\left[\partial_{x}\partial_y b(x,Y_{t}^{x,y})\cdot(\partial_x Y_{t}^{x,y}\cdot l_1, l_2)\right] \\
&&+\EE\left[\partial_{y}^{2}b(x,Y_{t}^{x,y})\cdot (\partial_xY_{t}^{x,y}\cdot l_1,\partial_xY_{t}^{x,y}\cdot l_2)\right]\\
			&&+\EE\left[\partial_yb(x,Y_{t}^{x,y})\cdot(\partial_{x}^{2}Y_{t}^{x,y}\cdot (l_1,l_2))\right],
		\end{eqnarray*}
		which implies
		\begin{eqnarray}
			\partial_y\partial_{x}^{2} \hat{b}(x,y,t)\cdot (l_1,l_2,l_3)=\!\!\!\!\!\!\!&&\EE\left[\partial_y\partial_{x}^{2} b(x,Y_{t}^{x,y})\cdot (l_1,l_2,\partial_y Y^{x,y}_t\cdot l_3)\right]\nonumber\\
			&&+\EE\left[\partial^2_y\partial_{x}b(x,Y_{t}^{x,y})\cdot (l_1, \partial_x Y_{t}^{x,y}\cdot l_2, \partial_y Y_{t}^{x,y}\cdot l_3)\right]\nonumber\\
			&&+\EE\left[\partial_y\partial_{x}b(x,Y_{t}^{x,y})\cdot (l_1, \partial_y\partial_x Y_{t}^{x,y}\cdot (l_2,l_3))\right]\nonumber\\
			&&+\EE\left[\partial_y\partial_{x}\partial_y b(x,Y_{t}^{x,y})\cdot(\partial_x Y_{t}^{x,y}\cdot l_1, l_2, \partial_y Y_{t}^{x,y}\cdot l_3)\right] \nonumber\\
			&&+\EE\left[\partial_{x}\partial_y b(x,Y_{t}^{x,y})\cdot(\partial_y\partial_x Y_{t}^{x,y}\cdot (l_1,l_3), l_2)\right] \nonumber\\
			&&+\EE\left[\partial_{y}^{3}b(x,Y_{t}^{x,y})\cdot (\partial_xY_{t}^{x,y}\cdot l_1,\partial_xY_{t}^{x,y}\cdot l_2, \partial_yY_{t}^{x,y}\cdot l_3) \right]\nonumber\\
			&&+\EE\left[\partial_{y}^{2}b(x,Y_{t}^{x,y})\cdot (\partial_y\partial_xY_{t}^{x,y}\cdot (l_1,l_3),\partial_xY_{t}^{x,y}\cdot l_2)\right]\nonumber\\
			&&+\EE\left[\partial_{y}^{2}b(x,Y_{t}^{x,y})\cdot (\partial_xY_{t}^{x,y}\cdot l_1,\partial_y\partial_xY_{t}^{x,y}\cdot (l_2,l_3))\right]\nonumber\\
			&&+\EE\left[\partial_{y}^{2}b(x,Y_{t}^{x,y})\cdot(\partial_{x}^{2}Y_{t}^{x,y}\cdot (l_1,l_2),\partial_yY_{t}^{x,y}\cdot l_3 )\right]\nonumber\\
			&&+\EE\left[\partial_yb(x,Y_{t}^{x,y})\cdot(\partial_y\partial_{x}^{2}Y_{t}^{x,y}\cdot (l_1,l_2,l_3))\right],   \label{Paryxx}
		\end{eqnarray}
		where $\partial^2_x Y^{x,y}_t\cdot (l_1,l_2)$ is the directional derivative of $\partial_x Y^{x,y}_t\cdot l_1$ with respect to $x$ in the direction $l_2$. $\partial_y\partial_{x}^{2}Y_{t}^{x,y}\cdot (l_1,l_2,l_3)$ is the directional derivative of $\partial^2_x Y^{x,y}_t\cdot (l_1,l_2)$ with respect to $y$ in the direction $l_3$. Refer to Remark \ref{R6.2} in the appendix, we have
\begin{eqnarray}
&&\sup_{t\geq 0}\EE|\partial^2_{x}Y^{x,y}_t\cdot (l_1,l_2)|^4\leq C(1+|y|^{4k}),\label{ParxxY}\\
&&\EE|\partial_y\partial^2_{x}Y^{x,y}_t\cdot (l_1,l_2,l_3)|^2\leq Ce^{-2\gamma t}(1+|y|^{4k}),\label{ParyxxY}
\end{eqnarray}
where $C,\gamma>0$. Then by \eref{ParxY}, \eref{ParyY}, \eref{ParyxY}, \eref{ParxxY} and \eref{ParyxxY}, we obtain
		\begin{eqnarray*}
			|\partial_y\partial_{x}^{2} \hat{b}(x,y,t)\cdot (l_1,l_2,l_3)|\leq C(1+|x|^{3k}+|y|^{3k})e^{-\gamma t},\label{Paryxxhatb}
		\end{eqnarray*}
		which implies
		\begin{eqnarray}
			J_{1}\le\!\!\!\!\!\!\!&& Ce^{-\gamma t} \EE\big[(1+|x|^{3k}+|y|^{3k}+|Y^{x,y}_{t_0}|^{3k})|y-Y^{x,y}_{t_0}|\big]\nonumber\\
			\le\!\!\!\!\!\!\!&&Ce^{-\gamma t}(1+|x|^{3k+1}+|y|^{3k+1}). \label{J_1}
		\end{eqnarray}
		
		For the terms $J_2$-$J_5$. By a similar argument as in the proof of \eref{Paryxxhatb}, it is easy to prove
		\begin{eqnarray*}
			&&|\partial_x\partial_{y}\hat b( x,y, t)\cdot(l_1,l_2)|\leq C e^{-\gamma t}(1+|x|^{2k}+|y|^{2k}),\\
			&&|\partial_{y}^{2}\hat{b}(x,y,t)\cdot (l_1,l_2)|\leq C e^{-\gamma t}(1+|x|^{2k}+|y|^{2k}),
		\end{eqnarray*}
		which together with \eref{ParxY}, \eref{R4}, \eref{Paryxhatb} and \eref{ParxxY}, we have
		\begin{eqnarray}\label{RQ3}
			\sum^5_{i=2}|J_i|\leq C e^{-\gamma t}(1+|x|^{3k}+|y|^{3k}).\label{J_2-J_5}
		\end{eqnarray}
		Finally, by \eref{J_1}  and \eref{J_2-J_5}, we get \eref{E4.}. The proof is complete.
	\end{proof}
	\begin{remark}
		The regularity estimate of the solution to the Poisson equation with respect to its parameters have been studied in many references, see e.g.  \cite{PV1,PV2,RX2021}, however the non-degenerate and bounded conditions for the singular coefficients are assumed in these mentioned references. Here we use the method by a straightforward computation, see e.g. \cite{RSX2021,SXX}, the coefficients $f$ and $g$ may have polynomial growth, which has its own interest. Meanwhile, the solution of Poisson equation \eref{PE} is not unique if without any other assumptions, however if the solution $\Phi(x,y)$ also satisfies the central condition, i.e., $\int_{\mathbb{R}^m}\Phi(x,y)\mu^x(dy)=0,\forall x\in\mathbb{R}^n$, then the solution is unique.	
	\end{remark}

	\section{Proofs of main results}
	
	In this section, we are going to give the detailed proofs of Theorems \ref{main result 1} and \ref{main result 2} in subsections 5.1 and 5.2, respectively.
	\subsection{Proof of Theorem 2.2}
	\begin{proof}
		Note that
		\begin{eqnarray*}
			X^\varepsilon_t-\bar{X}_t=\!\!\!\!\!\!&&\int_0^{t}\left[b(X_s^\varepsilon,Y_s^\varepsilon)-\bar{b}(\bar{X_s})\right]ds+\int_0^t\left[\sigma(X_s^\varepsilon)-\sigma(\bar{X_s})\right]dW_s^1\\
			&&+\int_0^t\int_{\mathcal{Z}_1}\left[h_1(X_{s-}^{\varepsilon},z)-h_1(\bar{X}_{s-},z)\right]\tilde{N}^1(ds,dz).
		\end{eqnarray*}
		By It\^{o}'s formula, we have for any $p\geq 4$,
		\begin{eqnarray*}
			|X_t^\varepsilon-\bar{X_t}|^{p}=\!\!\!\!\!\!&&p\int_0^t|X_s^\varepsilon-\bar{X}_s|^{p-2}\left\langle X_s^\varepsilon-\bar{X}_s,b(X_s^{\varepsilon},Y_s^{\varepsilon})-\bar{b}({X_s^{\varepsilon }})\right\rangle ds\\
			&&+p\int_0^t|X_s^\varepsilon-\bar{X}_s|^{p-2}\left\langle X_s^\varepsilon-\bar{X}_s,\bar{b}({X_s^{\varepsilon }})-\bar{b}(\bar{X}_s)\right\rangle ds\\
			&&+{\frac{p(p-2)}{2}}\int_0^t|X_s^\varepsilon-\bar{X}_s|^{p-4}\big|\!\left(\sigma(X_s^\varepsilon)-\sigma(\bar{X}_s)\right)^{*}\!\cdot\!\left(X_s^\varepsilon-\bar{X}_s \right) \!\big|^{2}ds\\
			&&+{\frac{p}{2}}\int_0^t|X_s^\varepsilon-\bar{X}_s|^{p-2} \|\sigma(X_s^\varepsilon)-\sigma(\bar{X}_s)\|^2ds\\
			&&+\int_0^t\int_{\mathcal{Z}_1}\big[|X_{s}^{\varepsilon}-\bar{X}_{s}+h_1(X_{s}^{\varepsilon},z)-h_1(\bar{X}_{s},z)|^p-|X_{s}^\varepsilon-\bar{X_{s}}|^{p} \\
			&&\quad\quad\quad-p|X_s^{\varepsilon}-\bar{X}_s|^{p-2}\left\langle X_s^\varepsilon-\bar{X}_s,h_1(X_{s}^\varepsilon,z)-h_1(\bar{X}_{s},z) \right\rangle \big]\nu_1(dz)ds \\
			&&+p\int_0^t|X_s^\varepsilon-\bar{X_s}|^{p-2}\left\langle X_s^\varepsilon-\bar{X_s},\left(\sigma(X_s^\varepsilon)-\sigma(\bar{X}_s)\right)dW_s^1\right\rangle\\
			&&+\int_0^t\!\int_{\mathcal{Z}_1}|X_{s-}^{\varepsilon}-\bar{X}_{s-}+h_1(X_{s-}^{\varepsilon},z)-h_1(\bar{X}_{s-},z)|^p-|X_{s-}^{\varepsilon}-\bar {X}_{s-}|^p\tilde{N}^1(ds,dz)\\
			=:\!\!\!\!\!\!&&\sum^{7}_{i=1}Q_i(t).
		\end{eqnarray*}
		
		By Young's inequality, we obtain that
		\begin{eqnarray}
			\mathbb{E}\left(\sup_{0\leq t\leq T}|Q_1(t)|\right) \le\!\!\!\!\!\!&& C_{p} \mathbb{E}\Bigg[\left(\sup_{0\leq s\leq T}|X_s^\varepsilon-\bar{X}_s|^{p-2}\right)\nonumber\\
			&&\quad\quad\quad \cdot\left(\sup_{0\leq t\leq T}\left|\int_0^t\left\langle X_s^\varepsilon-\bar{X}_s,b(X_s^{\varepsilon},Y_s^{\varepsilon})-\bar{b}({X_s^{\varepsilon }})\right\rangle ds\right|\right)\Bigg]\nonumber \\
			\le\!\!\!\!\!\!&& C_p\mathbb{E}\left(\sup_{0\leq t\leq T}\left|\int_0^t\left\langle X_s^\varepsilon-\bar{X}_s,b(X_s^{\varepsilon},Y_s^{\varepsilon})-\bar{b}({X_s^{\varepsilon }})\right\rangle ds\right|^{\frac{p}{2}}\right) \nonumber\\
			&&+ \frac{1}{4}\mathbb{E}\left(\sup_{0\leq t\leq T}|X_t^\varepsilon-\bar{X}_t|^{p}\right).\label{S5.1}
		\end{eqnarray}
		
		By \eref{ConA113} and \eref{Avc1}, it is easy to prove
		\begin{eqnarray}
			\sum^{4}_{i=2}\EE\left(\sup_{0\leq t\leq T}|Q_i(t)|\right)\le C_{p}\int_0^T\EE|X_s^\varepsilon-\bar{X}_s|^{p}ds.\label{S5.2}
		\end{eqnarray}
		
		By \eref{ConA12} and \eref{Taylor}, we have
		\begin{eqnarray}
			\EE\left(\sup_{0\leq t\leq T}|Q_5(t)|\right)\!\!\!\!\!\!\!\!\!&&\le\! C_{p}\EE\int_0^T\int_{\mathcal{Z}_1}|X_{s}^{\varepsilon}-\bar{X}_{s}|^{p-2}|h_1(X_{s}^{\varepsilon},z)-h_1(\bar{X}_{s},z)|^{2}\nonumber\\
			&&\quad\quad\quad+|h_1(X_{s}^{\varepsilon},z)-h_1(\bar{X}_{s},z)|^p\nu_1(dz) ds \nonumber\\
			&&\le C_{p}\int_0^T\EE|X_s^\varepsilon-\bar{X}_s|^{p}ds.\label{S5.3}
		\end{eqnarray}
		
		Using the Burkholder-Davis-Gundy's inequality (see \cite[Theorem 3.49]{PZ2007})  and Young's inequality, we get
		\begin{eqnarray}
			\mathbb{E}\left(\sup_{0\leq t\leq T}|Q_6(t)|\right)\!\!\!\!\!\!\!\!&&\le C_{p}\mathbb{E}\left[\int_{0}^{T}|X_s^\varepsilon-\bar{X}_s|^{2(p-1)} \|\sigma(X_s^\varepsilon)-\sigma(\bar{X}_s)\|^2ds \right]^{\frac{1}{2}}\nonumber\\
			\!\!\!\!&&\le \frac{1}{4}\mathbb{E}\left[\sup_{0\leq t\leq T}|X_t^\varepsilon-\bar{X}_t|^{p}\right]+\!  C_p\mathbb{E}\int_{0}^{T}|X_s^\varepsilon-\bar{X}_s|^{p}ds\label{S5.4}
		\end{eqnarray}
		and	using the Burkholder-Davis-Gundy's inequality  (see \cite[Theorem 3.50]{PZ2007}), \eref{Taylor2}  and Young's inequality, we obtain
		\begin{eqnarray}
			\mathbb{E}\left(\sup_{0\leq t\leq T}|Q_7(t)|\right)\le\!\!\!\!\!\!\!\!&& C_{p}\mathbb{E}\Big[\!\int_{0}^{T}\!\!\int_{\mathcal{Z}_1}\big(|X_{s-}^{\varepsilon}\!-\!\bar{X}_{s-}|^{2p-2}\cdot|h_1(X_{s-}^{\varepsilon},z)\!-\!h_1(\bar{X}_{s-},z)|^2\nonumber\\
			&&\quad\quad\quad\quad+|h_1(X_{s-}^{\varepsilon},z)\!-\!h_1(\bar{X}_{s-},z)|^{2p} \big){N}^1(dz,ds) \Big]^{\frac{1}{2}}\nonumber\\
			\le\!\!\!\!\!\!\!\!&&\frac{1}{4}\mathbb{E}\left(\sup_{0\leq t\leq T}|X_t^\varepsilon-\bar{X}_t|^{p}\right)+C_p\mathbb{E}\int_{0}^{T}|X_s^\varepsilon-\bar{X}_s|^{p}ds.\label{S5.5}
		\end{eqnarray}
		
		Combining  \eref{S5.1}-\eref{S5.5}, we get
		\begin{eqnarray*}
			\mathbb{E}\left(\sup_{0\leq t\leq T}|X^{\varepsilon}_t-\bar{X}_t|^p\right)\!\!\!\!\!\!&&
			\leq C_{p}\mathbb{E}\left[\sup_{0\leq t\leq T}\left|\int^t_0(X_{s}^{\varepsilon}-\bar{X}_{s})\cdot(b(X^{\varepsilon}_{s},Y^{\varepsilon}_{s})-\bar{b}(X^{\varepsilon}_{s}))ds\right|^{\frac{p}{2}}\right]\\
			\!\!\!\!\!\!&&+C_p\int_{0}^{T}\mathbb{E}|X_s^\varepsilon-\bar{X}_s|^{p}ds.
		\end{eqnarray*}
		
		Then by the Gronwall's inequality, it follows
		\begin{eqnarray}
			\mathbb{E}\left(\sup_{0\leq t\leq T}|X^{\varepsilon}_t-\bar{X}_t|^p\right)
			\leq\!\!C_{p,T}\mathbb{E}\left[\sup_{0\leq t\leq T}\big|\int^t_0(X_{s}^{\varepsilon}-\bar{X}_{s})\cdot(b(X^{\varepsilon}_{s},Y^{\varepsilon}_{s})-\bar{b}(X^{\varepsilon}_{s}))ds\big|^{p/2}\right]. \label{S5.6}
		\end{eqnarray}
		
		By the Proposition \ref{P4.1}, then following Poisson equation
		\begin{eqnarray}
			-\mathscr{L}_2(x)\Phi(x,\cdot)(y)=b(x,y)-\bar{b}(x)\label{S5.7}
		\end{eqnarray}
		admits a solution $\Phi(x,y)$ satisfying \eref{E1}-\eref{E4}.
		
		By It\^{o}'s formula, we obtain
		\begin{eqnarray*}
			&&\langle\Phi(X_t^{\varepsilon},Y_t^{\varepsilon}),X^{\varepsilon}_t-\bar{X}_t\rangle\\
			=\!\!\!\!\!\!&&\int_0^t\big\langle  X^{\varepsilon}_s-\bar{X}_s , \big\{\langle \partial_x\Phi(X_{s}^{\varepsilon},Y_s^{\varepsilon}),b(X_{s}^{\varepsilon},Y_s^{\varepsilon})\rangle+\frac{1}{2}\text{Tr}\big[ \sigma(X_{s}^{\varepsilon})\sigma^{*}(X_{s}^{\varepsilon})\partial_x^{2}\Phi(X_{s}^{\varepsilon},Y_s^{\varepsilon})\big]\\
			&&\quad+\int_{\mathcal{Z}_1}\left[\Phi(X_{s}^{\varepsilon}\!+\!h_1(X_{s}^{\varepsilon},z),Y_{s}^{\varepsilon}) \!-\!\Phi(X_{s}^{\varepsilon},Y_{s}^{\varepsilon}) -\partial_x\Phi(X_{s}^{\varepsilon},Y_s^{\varepsilon})\cdot h_1(X_{s}^{\varepsilon},z) \right] \nu_1(dz)   \big\}\big\rangle ds\\
			&&+\int_0^t \langle\Phi(X_s^{\varepsilon},Y_s^{\varepsilon}),b(X_s^{\varepsilon},Y_s^{\varepsilon})-\bar{b}(\bar{X_s})\rangle + \text{Tr}\left[\sigma(X_s^{\varepsilon})\big(\sigma^{\ast}(X_s^{\varepsilon})-\sigma^{\ast}(\bar {X}_s)\big)\cdot \partial_x\Phi(X_s^{\varepsilon},Y_s^{\varepsilon})\right] ds\\
			&&+\int_0^t\int_{\mathcal{Z}_1}\langle [\Phi(X_{s}^{\varepsilon}+h_1(X_{s}^{\varepsilon},z),Y_{s}^{\varepsilon})-\Phi(X_{s}^{\varepsilon},Y_{s}^{\varepsilon})]\cdot\big(h_1(X_{s}^{\varepsilon},z)-h_1(\bar {X}_{s},z)\big)\rangle \nu_1(dz)ds\\
			&&+\frac{1}{\varepsilon}\int_0^t \langle X^{\varepsilon}_s-\bar{X}_s , \mathscr{L}_{2}(X_{s}^{\varepsilon})\Phi(X^{\varepsilon}_{s},Y^{\varepsilon}_{s}) \rangle  ds\\
			&&+\int_0^t\langle\partial_x\Phi(X_s^{\varepsilon},Y_s^{\varepsilon})\cdot \sigma(X_s^{\varepsilon})dW_s^1, X^{\varepsilon}_{s} -\bar{X}_{s} \rangle+\int_0^t\langle\Phi(X_s^{\varepsilon},Y_s^{\varepsilon}),\big(\sigma(X_s^\varepsilon)-\sigma(\bar{X_s}) \big)dW_s^1 \rangle \\
			&&+\int_0^t\!\int_{\mathcal{Z}_1}\Big[\langle \Phi(X_{s-}^{\varepsilon}\!+\!h_1(X_{s-}^{\varepsilon},z),Y_{s-}^{\varepsilon})-\Phi(X_{s-}^{\varepsilon},Y_{s-}^{\varepsilon}), X^{\varepsilon}_{s-}-\bar{X}_{s-}\rangle\\
			&&\quad\quad\quad+\langle \Phi(X_{s-}^{\varepsilon}+h_1(X_{s-}^{\varepsilon},z),Y_{s-}^{\varepsilon}), h_1(X_{s-}^{\varepsilon},z)-h_1(\bar {X}_{s-},z)\rangle \Big]\tilde{N}^1(ds,dz)\\
			&&+\frac{1}{\sqrt{\varepsilon}}\int_0^t\langle\partial_y\Phi(X_s^{\varepsilon},Y_s^{\varepsilon})\cdot (X^{\varepsilon}_{s} -\bar{X}_{s} ),g(X_s^{\varepsilon},Y_s^{\varepsilon})dW_s^2\rangle\\
			&&+\int_0^t\int_{\mathcal{Z}_2}\langle\Phi(X_{s-}^{\varepsilon},Y_{s-}^{\varepsilon}+h_2(X_{s-}^{\varepsilon},Y_{s-}^{\varepsilon},z))- \Phi(X_{s-}^{\varepsilon},Y_{s-}^{\varepsilon}), X^{\varepsilon}_{s-} -\bar{X}_{s-} \rangle\tilde{N}^{2,\varepsilon}(ds,dz),
		\end{eqnarray*}
		which implies
		\begin{eqnarray}
			&&\int^t_0 \langle-\mathscr{L}_{2}(X_{s}^{\varepsilon})\Phi(X^{\varepsilon}_{s},Y^{\varepsilon}_{s}), X_{s}^{\varepsilon}-\bar{X}_{s}\rangle ds=\varepsilon\Big[-\langle \Phi(X_{t}^{\varepsilon},Y^{\varepsilon}_{t})
			, X_{t}^{\varepsilon}\!-\!\bar{X}_{t}\rangle\Big]\nonumber\\
			&&\quad+\varepsilon\Bigg\{\int_0^t\Big\langle  X^{\varepsilon}_s-\bar{X}_s , \Big\{\langle \partial_x\Phi(X_{s}^{\varepsilon},Y_s^{\varepsilon}),b(X_{s}^{\varepsilon},Y_s^{\varepsilon})\rangle+\frac{1}{2}\text{Tr}\big[ \sigma(X_{s}^{\varepsilon})\sigma^{*}(X_{s}^{\varepsilon})\partial_x^{2}\Phi(X_{s}^{\varepsilon},Y_s^{\varepsilon})\big]\nonumber\\
			&&\quad\quad\quad+\int_{\mathcal{Z}_1}\left[\Phi(X_{s}^{\varepsilon}\!+\!h_1(X_{s}^{\varepsilon},z),Y_{s}^{\varepsilon}) \!-\!\Phi(X_{s}^{\varepsilon},Y_{s}^{\varepsilon}) -\partial_x\Phi(X_{s}^{\varepsilon},Y_s^{\varepsilon})\cdot h_1(X_{s}^{\varepsilon},z) \right] \nu_1(dz)   \Big\}\Big\rangle ds\nonumber\\
			&&\quad+\int_0^t\left\langle\Phi(X_s^{\varepsilon},Y_s^{\varepsilon}),b(X_s^{\varepsilon},Y_s^{\varepsilon})-\bar{b}(\bar{X_s})\right\rangle + \text{Tr}\left[\sigma(X_s^{\varepsilon})\big(\sigma^{\ast}(X_s^{\varepsilon})-\sigma^{\ast}(\bar {X}_s)\big)\cdot \partial_x\Phi(X_s^{\varepsilon},Y_s^{\varepsilon})\right] ds\nonumber\\
			&&\quad+\int_0^t\int_{\mathcal{Z}_1}\left\langle [\Phi(X_{s}^{\varepsilon}+h_1(X_{s}^{\varepsilon},z),Y_{s}^{\varepsilon})-\Phi(X_{s}^{\varepsilon},Y_{s}^{\varepsilon})]\cdot\big(h_1(X_{s}^{\varepsilon},z)-h_1(\bar {X}_{s},z)\big)\right\rangle \nu_1(dz)ds\nonumber\\
			&&\quad+\int_0^t\left\langle\partial_x\Phi(X_s^{\varepsilon},Y_s^{\varepsilon})\cdot \sigma(X_s^{\varepsilon})dW_s^1, X^{\varepsilon}_{s} -\bar{X}_{s} \right\rangle+\int_0^t\left\langle\Phi(X_s^{\varepsilon},Y_s^{\varepsilon}),\big(\sigma(X_s^\varepsilon)-\sigma(\bar{X_s}) \big)dW_s^1 \right\rangle\nonumber\\
			&&\quad+\int_0^t\!\int_{\mathcal{Z}_1}\Big[\langle \Phi(X_{s-}^{\varepsilon}\!+\!h_1(X_{s-}^{\varepsilon},z),Y_{s-}^{\varepsilon})-\Phi(X_{s-}^{\varepsilon},Y_{s-}^{\varepsilon}), X^{\varepsilon}_{s-}-\bar{X}_{s-}\rangle\nonumber\\
			&&\quad\quad\quad+\langle \Phi(X_{s-}^{\varepsilon}+h_1(X_{s-}^{\varepsilon},z),Y_{s-}^{\varepsilon}), h_1(X_{s-}^{\varepsilon},z)-h_1(\bar {X}_{s-},z)\rangle \Big]\tilde{N}^1(ds,dz)\Bigg\}\nonumber\\
			&&\quad+\vare\int_0^t\int_{\mathcal{Z}_2}\langle\Phi(X_{s-}^{\varepsilon},Y_{s-}^{\varepsilon}+h_2(X_{s-}^{\varepsilon},Y_{s-}^{\varepsilon},z))- \Phi(X_{s-}^{\varepsilon},Y_{s-}^{\varepsilon}), X^{\varepsilon}_{s-} -\bar{X}_{s-} \rangle\tilde{N}^{2,\varepsilon}(ds,dz)\nonumber\\
			&&\quad+\sqrt{\vare}\int_0^t\left\langle\partial_y\Phi(X_s^{\varepsilon},Y_s^{\varepsilon})\cdot (X^{\varepsilon}_{s} -\bar{X}_{s} ),g(X_s^{\varepsilon},Y_s^{\varepsilon})dW_s^2\right\rangle\nonumber\\
			=:	\!\!\!\!\!\!\!\!&&\sum^4_{i=1}\tilde{Q}^{\vare}_i(t).\label{S5.8}
		\end{eqnarray}
		
		Combing \eref{S5.6}-\eref{S5.8}, it is obvious that
		\begin{eqnarray}
			\mathbb{E}\left(\sup_{0\leq t\leq T}|X^{\varepsilon}_t-\bar{X}_t|^p\right)\leq\!\!\!\!\!\!\!\!&& C_{p,T}\mathbb{E}\left[\sup_{0\leq t\leq T}\left|\int^t_0\langle\mathscr{L}_{2}(X_{s}^{\varepsilon})\Phi(X^{\varepsilon}_{s},Y^{\varepsilon}_{s}), X_{s}^{\varepsilon}-\bar{X}_{s}\rangle ds\right|^{p/2}\right]\nonumber\\
			\leq \!\!\!\!\!\!\!\!&& C_{p,T}\sum^4_{i=1}\mathbb{E}\left(\sup_{0\leq t\leq T}|\tilde{Q}^{\vare}_i(t)|^{p/2}\right).\label{S5.9}		
		\end{eqnarray}
		
		For the term $\tilde{Q}^{\vare}_1(t)$. By Young's inequality,  (\ref{E1}) and \eref{Y.1}, we get for any $p\geq 2$ and small enough $\varepsilon>0$,
		\begin{eqnarray}
			\mathbb{E}\left(\sup_{0\leq t\leq T}|\tilde{Q}^{\vare}_1(t)|^{\frac{p}{2}}\right)\!\!\!\!\!\!\!\!&&\leq \varepsilon^{\frac{p}{2}}\mathbb{E}\left(\sup_{0\leq t\leq T}|\Phi(X_{t}^{\varepsilon},Y^{\varepsilon}_{t})\cdot(X_{t}^{\varepsilon}-\bar{X}_{t})|^{p/2}\right)\nonumber\\
			\leq\!\!\!\!\!\!\!\!&&\frac{1}{4}\mathbb{E}\left(\sup_{0\leq t\leq T}|X^{\varepsilon}_t-\bar{X}_t|^{p}\right)+C_{p}\varepsilon^{p}\Big\{1+\mathbb{E}\left(\sup_{0\leq t\leq T}|Y^{\varepsilon}_{t}|^{p(k+1)}\right) \Big\}\nonumber\\
			\leq\!\!\!\!\!\!\!\!&&\frac{1}{4}\mathbb{E}\left(\sup_{0\leq t\leq T}|X^{\varepsilon}_t-\bar{X}_t|^p\right)+C_{p,T}\varepsilon^{p/2}\big(1+|y|^{p(k+1)}\big).\label{S5.10}	
		\end{eqnarray}
		
		For the term $\tilde{Q}^{\vare}_2(t)$. Using Burkholder-Davis-Gundy's inequality (see \cite[Theorem 3.50]{PZ2007}),  assumptions \ref{A1}-\ref{A3}, \eref{E1}-\eref{E4}, \eref{X}, \eref{Y}, (\ref{AR4}) and by a straightforward computation,  it is easy to prove
		\begin{eqnarray}
			\mathbb{E}\left(\sup_{0\leq t\leq T}|\tilde{Q}^{\vare}_2(t)|^{p/2}\right)\leq C_{p,T}\varepsilon^{p/2}\big(1+|x|^{3(k+1)p}+|y|^{3q(k+1)p/2}\big).\label{S5.10}	
		\end{eqnarray}

     For the term $\tilde{Q}^{\vare}_3(t)$. Using Kunita's first inequality (see \cite[Theorem 4.4.23]{A2009}), for any $p\geq  4$
		\begin{eqnarray}
			&&\mathbb{E}\left(\sup_{0\leq t\leq T}|\tilde{Q}^{\vare}_3(t)|^{p/2}\right)\nonumber\\
			\leq\!\!\!\!\!\!\!\!&&\varepsilon^{p/4}C_{p}\mathbb{E}\left[\int^T_0\int_{\mathcal{Z}_2}|\Phi(X^{\varepsilon}_{s},Y^{\varepsilon}_{s}+h_2(X_{s}^{\varepsilon},Y_{s}^{\varepsilon},z))
			-\Phi(X^{\varepsilon}_{s},Y^{\varepsilon}_{s})|^2 |X_{s}^{\varepsilon}-\bar{X}_{s}|^{2}\nu_2(dz)ds\right]^{p/4}\nonumber\\	 \!\!\!\!\!\!\!\!&&+\varepsilon^{p/2-1}C_{p}\mathbb{E}\int^T_0\int_{\mathcal{Z}_2}|\Phi(X^{\varepsilon}_{s},Y^{\varepsilon}_{s}+h_2(X_{s}^{\varepsilon},Y_{s}^{\varepsilon},z))
			-\Phi(X^{\varepsilon}_{s},Y^{\varepsilon}_{s})|^{p/2}|X_{s}^{\varepsilon}-\bar{X}_{s}|^{p/2}\nu_2(dz)ds\nonumber\\
			\leq\!\!\!\!\!\!\!\!&&\frac{1}{4} \mathbb{E}\left(\sup_{0\leq t\leq T}|X^{\varepsilon}_t-\bar{X}_t|^p\right)  \nonumber \\
			&&+\varepsilon^{p/2}C_{p}\mathbb{E}\left[\int^{T}_0\!\!\!\int_{\mathcal{Z}_2}(1+|Y^{\varepsilon}_{s}|^{2k}+|h_2(X_{s}^{\varepsilon},Y_{s}^{\varepsilon},z)|^{2k})
			|h_2(X_{s}^{\varepsilon},Y_{s}^{\varepsilon},z)|^{2}\nu_2(dz)ds\right]^{p/2}\nonumber\\
			&&+\varepsilon^{p-2}C_{p}\mathbb{E}\left[\int^{T}_0\!\!\!\int_{\mathcal{Z}_2}(1+|Y^{\varepsilon}_{s}|^{kp/2}+|h_2(X_{s}^{\varepsilon},Y_{s}^{\varepsilon},z)|^{kp/2})
			|h_2(X_{s}^{\varepsilon},Y_{s}^{\varepsilon},z)|^{p/2}\nu_2(dz)ds\right]^{2}\nonumber\\	
			\leq \!\!\!\!\!\!\!\!&&\frac{1}{4}\mathbb{E}\left(\sup_{0\leq t\leq T}|X^{\varepsilon}_t-\bar{X}_t|^p\right)+\varepsilon^{p/2}C_{p,T}\big(1+|y|^{(k+1)p}\big).\label{S5.16}
		\end{eqnarray}
		
		For the term $\tilde{Q}^{\vare}_4(t)$. Using Burkholder-Davis-Gundy's inequality (see \cite[Theorem 3.49]{PZ2007}), Young's inequality and H\"{o}lder's inequality, we obtain
		\begin{eqnarray}
			\mathbb{E}\left(\sup_{0\leq t\leq T}|\tilde{Q}^{\vare}_4(t)|^{p/2}\right)
			\leq\!\!\!\!\!\!\!\!&&\vare^{p/4}C_{p}\EE\left[\int_0^T\|\partial_y\Phi(X_s^{\varepsilon},Y_s^{\varepsilon})\|^2 |X^{\varepsilon}_{s} -\bar{X}_{s} |^2\|g(X_s^{\varepsilon},Y_s^{\varepsilon})\|^2ds\right]^{p/4}\nonumber\\
			\leq\!\!\!\!\!\!\!\!&& \frac{1}{4}\mathbb{E}\left(\sup_{0\leq t\leq T}|X^{\varepsilon}_t-\bar{X}_t|^p\right)+\vare^{p/2}C_{p,T}\EE\int_0^T\|\partial_y\Phi(X_s^{\varepsilon},Y_s^{\varepsilon})\|^p\|g(X_s^{\varepsilon},Y_s^{\varepsilon})\|^pds\nonumber \\
			\leq\!\!\!\!\!\!\!\!&&\frac{1}{4}\mathbb{E}\left(\sup_{0\leq t\leq T}|X^{\varepsilon}_t-\bar{X}_t|^p\right)+\vare^{p/2}C_{p,T}\big(1+|y|^{(k+1)p}\big).\label{S5.17}
		\end{eqnarray}
		
		Hence, by \eref{S5.9}-\eref{S5.17}, we final obtain
		\begin{eqnarray*}
			\mathbb{E}\left(\sup_{0\leq t\leq T}|X^{\varepsilon}_t-\bar{X}_t|^p\right)\leq \varepsilon^{p/2}C_{p,T}\left(1+|x|^{3(k+1)p}+|y|^{3q(k+1)p/2}\right).
		\end{eqnarray*}
		The proof is complete.
	\end{proof}

	\subsection{Proof of Theorem \ref{main result 2}}
	Recall the gerenal averaged equation:
	\begin{equation}\label{AR2}
		d\bar{X_{t} }=\bar{b}(\bar{X}_{t})dt+\bar \sigma (\bar{X}_{t})d W_{t}+\int_{\mathcal{Z}_1 }h_{1}(\bar{X }_{t-},z)\tilde{N}^{1}(dz,dt) \\
		\displaystyle ,\bar{X}_0=x,
	\end{equation}
	where $\bar{b}(x)=\int_{\RR^{m}}b(x,y)\mu^{x}(dy)$, $\bar \sigma (x)=\left[\int_{\RR^m}\sigma(x,y)\sigma^{\ast}(x,y)\mu^x(dy)\right]^{1/2}$ and $W$ is a $n$-dimensional standard Wiener process.

	Under the assumption \ref{A4}, using a similar argument as in the proof of Proposition \ref{DFY} in the appendix, we can  prove the differentiability of $Y^{x,y}_t$ with respect to $x$ up to fourth derivative. By \eref{ConA422}, \eref{ConA423} and \eref{ConA43}, it is easy to prove that there exists $C>0$ such that for any $x\in\RR^n, y\in \RR^m$ and unit vectors $l_i\in\RR^n$, $i=1,2,3,4,$
	\begin{eqnarray}
		&&\sup_{t\geq 0}\EE|\partial_{x}Y^{x,y}_t\cdot l_1|^{\ell}\leq C, \label{Parx^1Y}\\
		&&\sup_{t\geq 0}\EE|\partial^{2}_{x}Y^{x,y}_t\cdot (l_1,l_2) |^{8}\leq C(1+|y|^{8k}), \label{Parx^2Y}\\
		&&\sup_{t\geq 0}\EE|\partial^{3}_{x}Y^{x,y}_t\cdot (l_1,l_2,l_3)|^{4}\leq C(1+|y|^{8k}),\label{Parx^3Y}\\
		&&\sup_{t\geq 0}\EE|\partial^{4}_{x}Y^{x,y}_t\cdot (l_1,l_2,l_3,l_4)|^{2}\leq C(1+|y|^{6k}),\label{Parx^4Y}
	\end{eqnarray}
where $\ell>16$ is the constant in assumption \ref{A4}. Note that
	$$\bar{b}(x)=\int_{\RR^m}b(x,y)\mu^{x}(dy)=\lim_{t\rightarrow \infty}\EE b(x, Y^{x,0}_t).$$
	Using $b\in C^{4,4}_p(\RR^n\times\RR^m,\RR^n)$, \eref{FroY}, \eref{Parx^1Y}-\eref{Parx^4Y} and a straightforward computation, we can obtain that $\bar{b}\in C^{4}_{p}(\RR^n,\RR^n)$.
	
	Similarly, for any $1\leq i,j\leq n$,
	$$(\overline{\sigma\sigma^{\ast}})_{ij}(x)
	=\int_{\RR^m}\sum^{d_1}_{l=1}\sigma_{il}(x,y)\sigma_{jl}(x,y)\mu^{x}(dy)=\lim_{t\rightarrow \infty}\sum^{d_1}_{l=1}\EE\left[ \sigma_{il}(x, Y^{x,0}_t)\sigma_{jl}(x,Y^{x,0}_t)\right].$$
	Using  \eref{ConA412}, \eref{Parx^1Y}-\eref{Parx^4Y} and a straightforward computation, it is easy to prove
	$\overline{\sigma\sigma^{\ast}}
	\in C^{4}_{b}(\RR^n,\RR^{n}\times\RR^n)$.  Note that $\bar{\sigma}(x)=
	(\overline{\sigma\sigma^{\ast}})^{1/2}(x)$ and \eref{ConA42} holds, then  by a similar argument as used in \cite[Lemma A.7]{CLX}, we can get $\bar{\sigma}\in C^4_b(\RR^n,\RR^n\times \RR^n)$.
	
	Thus, by the same argument as in the proof of Lemma \ref{PMA}, it is easy to check that equation \eref{AR2} admits a unique solution $\{\bar{X}^x_t\}_{t\geq 0}$. Moreover, for any $T>0$ and $p>0$, there exists $C_{p,T}>0$ such that
	\begin{eqnarray}
		\mathbb{E}\left(\sup_{0\leq t\leq T}|\bar{X}^x_{t}|^{p}\right)\leq C_{p,T}\left(1+|x|^{p}\right).\label{AR3}
	\end{eqnarray}
	
	\vspace{0.2cm}
	Now, we consider the following Kolmogorov equation:
	\begin{equation}\left\{\begin{array}{l}\label{KE}
			\displaystyle
			\partial_t u(t,x)=\bar{\mathscr{L}}_1 u(t,x),\quad t\in[0, T], \\
			u(0, x)=\phi(x),
		\end{array}\right.
	\end{equation}
	where $\phi\in C^{4}_p(\RR^{n})$ and $\bar{\mathscr{L}}_1$ is the infinitesimal generator of the transition semigroup of the averaged equation \eref{AR2}, which is given by
	\begin{eqnarray*}
		&&\bar{\mathscr{L}}_1\phi(x):=\langle \bar{b}(x), D  \phi(x)\rangle+\frac{1}{2}\text{Tr}[\bar{\sigma}\bar{\sigma}(x)D^{2}\phi(x)]\\
		&&\quad\quad\quad\quad\quad\quad+\int_{\mathcal{Z}_1}\left[\phi(x+h_{1}(x,z))\!-\!\phi(x)\!-\!\langle D  \phi(x),h_{1}(x,z)\rangle \right]\nu_1(dz).
	\end{eqnarray*}
	Refer to Proposition \ref{EUE} in the appendix, \eref{KE} admits a unique solution $u$ which is given by
	$$
	u(t,x)=\EE\phi(\bar{X}^x_t),\quad t\geq 0.
	$$
	
	Next, we present the regularity estimates of the solution $u$ of equation \eref{KE}.
	\begin{lemma} \label{Lemma 5.1}
		For any unit vectors $l_1,l_2,l_3,l_4\in \RR^n$ and $T>0$, there exist $C_T>0$ and $k_1>0$ such that for any $x\in\RR^n$,
		\begin{eqnarray}
			&&\sup_{t\in[0,T]}|\partial_{x} u(t,x)\cdot l_1|\leq C_{T}(1+|x|^{k_1}),\label{UE3}\\
			&&\sup_{t\in[0,T]}|\partial^2_{x} u(t,x)\cdot (l_1,l_2)|\leq C_T (1+|x|^{k_1}),\label{UE4}\\
			&&\sup_{t\in[0,T]}|\partial^3_{x} u(t,x)\cdot (l_1,l_2,l_3)|\leq C_{T}(1+|x|^{k_1}), \label{UE5}\\
			&&\sup_{t\in[0,T]}|\partial^4_{x} u(t,x)\cdot (l_1,l_2,l_3,l_4)|\leq C_{T}(1+|x|^{k_1}), \label{UE6}\\
			&&\sup_{t\in[0,T]}|\partial_t(\partial_x u(t,x))\cdot l_1|\leq C_T (1+|x|^{k_1}),\label{UE1}\\
			&&\sup_{t\in[0,T]}|\partial_t(\partial^2_x u(t,x))\cdot (l_1,l_2)|\leq C_T (1+|x|^{k_1}).\label{UE2}
		\end{eqnarray}
	\end{lemma}
	\begin{proof}
		By the chain rule and $\phi\in C^{4}_p(\RR^{n})$, we have for any unit vectors $l_1,l_2,l_3,l_4\in \RR^n$,
		\begin{eqnarray*}
			&&\partial_x u(t,x)\cdot l_1=\EE[D\phi(\bar{X}^x_t)\cdot (\partial_x \bar{X}^x_t\cdot l_1)];\nonumber\\
			&&\partial^2_{x} u(t,x)\cdot (l_1,l_2)=\EE\left[D^2\phi(\bar{X}^x_t)\cdot (\partial_x \bar{X}^x_t\cdot l_1, \partial_x \bar{X}^x_t\cdot l_2)\right]\\
			&&\quad\quad\quad\quad\quad\quad\quad\quad\quad+\EE\left[D\phi(\bar{X}^x_t)\cdot  (\partial^2_{x} \bar{X}^x_t\cdot (l_1,l_2))\right];\\
			&&\partial^3_{x} u(t,x)\cdot (l_1,l_2,l_3)=\EE\left[D^3\phi(\bar{X}^x_t)\cdot  (\partial_x \bar{X}^x_t\cdot l_1, \partial_x \bar{X}^x_t\cdot l_2, \partial_x \bar{X}^x_t\cdot l_3)\right]\\
			&&\quad\quad\quad\quad\quad\quad\quad\quad\quad\quad+\EE\left[D^2\phi(\bar{X}^x_t)\cdot (\partial^2_x \bar{X}^x_t\cdot ( l_1, l_3), \partial_x \bar{X}^x_t\cdot l_2)\right]\\
			&&\quad\quad\quad\quad\quad\quad\quad\quad\quad\quad+\EE\left[D^2\phi(\bar{X}^x_t)\cdot (\partial_x \bar{X}^x_t\cdot l_1, \partial^2_x \bar{X}^x_t\cdot ( l_2,l_3))\right]\\
			&&\quad\quad\quad\quad\quad\quad\quad\quad\quad\quad+\EE\left[D^2\phi(\bar{X}^x_t)\cdot  (\partial^2_{x} \bar{X}^x_t\cdot (l_1,l_2), \partial_x \bar{X}^x_t\cdot l_3)\right]\\
			&&\quad\quad\quad\quad\quad\quad\quad\quad\quad\quad+\EE\left[D\phi(\bar{X}^x_t)\cdot (\partial^3_x \bar{X}^x_t\cdot (l_1,l_2,l_3))\right];\\
			&&\partial^4_{x} u(t,x)\cdot (l_1,l_2,l_3,l_4)=\EE\left[D^4\phi(\bar{X}^x_t)\cdot  (\partial_x \bar{X}^x_t\cdot l_1, \partial_x \bar{X}^x_t\cdot l_2, \partial_x \bar{X}^x_t\cdot l_3,\partial_x \bar{X}^x_t\cdot l_4)\right]\\
			&&\quad\quad\quad\quad\quad\quad\quad\quad\quad\quad+\EE\left[D^3\phi(\bar{X}^x_t)\cdot  (\partial^2_x \bar{X}^x_t\cdot (l_1,l_4), \partial_x \bar{X}^x_t\cdot l_2, \partial_x \bar{X}^x_t\cdot l_3)\right]\\
			&&\quad\quad\quad\quad\quad\quad\quad\quad\quad\quad+\EE\left[D^3\phi(\bar{X}^x_t)\cdot  (\partial_x \bar{X}^x_t\cdot l_1, \partial^2_x \bar{X}^x_t\cdot (l_2,l_4), \partial_x \bar{X}^x_t\cdot l_3)\right]\\
			&&\quad\quad\quad\quad\quad\quad\quad\quad\quad\quad+\EE\left[D^3\phi(\bar{X}^x_t)\cdot  (\partial_x \bar{X}^x_t\cdot l_1, \partial_x \bar{X}^x_t\cdot l_2, \partial^2_x \bar{X}^x_t\cdot (l_3,l_4))\right]\\
			&&\quad\quad\quad\quad\quad\quad\quad\quad\quad\quad+\EE\left[D^3\phi(\bar{X}^x_t)\cdot (\partial^2_x \bar{X}^x_t\cdot ( l_1, l_3), \partial_x \bar{X}^x_t\cdot l_2, \partial_x \bar{X}^x_t\cdot l_4)\right]\\
			&&\quad\quad\quad\quad\quad\quad\quad\quad\quad\quad+\EE\left[D^2\phi(\bar{X}^x_t)\cdot (\partial^3_x \bar{X}^x_t\cdot ( l_1, l_3,l_4), \partial_x \bar{X}^x_t\cdot l_2)\right]\\
			&&\quad\quad\quad\quad\quad\quad\quad\quad\quad\quad+\EE\left[D^2\phi(\bar{X}^x_t)\cdot (\partial^2_x \bar{X}^x_t\cdot ( l_1, l_3), \partial^2_x \bar{X}^x_t\cdot (l_2,l_4))\right]\\
			&&\quad\quad\quad\quad\quad\quad\quad\quad\quad\quad+\EE\left[D^3\phi(\bar{X}^x_t)\cdot (\partial_x \bar{X}^x_t\cdot l_1, \partial^2_x \bar{X}^x_t\cdot ( l_2,l_3), \partial_x \bar{X}^x_t\cdot l_4)\right]\\
			&&\quad\quad\quad\quad\quad\quad\quad\quad\quad\quad+\EE\left[D^2\phi(\bar{X}^x_t)\cdot (\partial^2_x \bar{X}^x_t\cdot (l_1,l_4), \partial^2_x \bar{X}^x_t\cdot ( l_2,l_3))\right]\\
			&&\quad\quad\quad\quad\quad\quad\quad\quad\quad\quad+\EE\left[D^2\phi(\bar{X}^x_t)\cdot (\partial_x \bar{X}^x_t\cdot l_1, \partial^3_x \bar{X}^x_t\cdot ( l_2,l_3,l_4))\right]\\
			&&\quad\quad\quad\quad\quad\quad\quad\quad\quad\quad+\EE\left[D^3\phi(\bar{X}^x_t)\cdot  (\partial^2_{x} \bar{X}^x_t\cdot (l_1,l_2), \partial_x \bar{X}^x_t\cdot l_3, \partial_x \bar{X}^x_t\cdot l_4)\right]\\
			&&\quad\quad\quad\quad\quad\quad\quad\quad\quad\quad+\EE\left[D^2\phi(\bar{X}^x_t)\cdot  (\partial^3_{x} \bar{X}^x_t\cdot (l_1,l_2,l_4), \partial_x \bar{X}^x_t\cdot l_3)\right]\\
			&&\quad\quad\quad\quad\quad\quad\quad\quad\quad\quad+\EE\left[D^2\phi(\bar{X}^x_t)\cdot  (\partial^2_{x} \bar{X}^x_t\cdot (l_1,l_2), \partial^2_x \bar{X}^x_t\cdot (l_3,l_4))\right]\\
			&&\quad\quad\quad\quad\quad\quad\quad\quad\quad\quad+\EE\left[D^2\phi(\bar{X}^x_t)\cdot (\partial^3_x \bar{X}^x_t\cdot (l_1,l_2,l_3), \partial_x \bar{X}^x_t\cdot l_4)\right]\\
			&&\quad\quad\quad\quad\quad\quad\quad\quad\quad\quad+\EE\left[D\phi(\bar{X}^x_t)\cdot (\partial^4_x \bar{X}^x_t\cdot (l_1,l_2,l_3,l_4))\right],
		\end{eqnarray*}
		where $\partial_x \bar{X}^{x}_t\cdot l$ is directional derivative of $\bar{X}^{x}_t$ with respect to $x$ in the direction $l$, which satisfies
		\begin{equation*}\left\{\begin{array}{l}
				\displaystyle
				d[\partial_x \bar{X}^{x}_t\cdot l]=D\bar{b}(\bar{X}^{x}_t)\cdot (\partial_x \bar{X}^{x}_t\cdot l)dt+D\bar{\sigma}(\bar{X}^{x}_t)\cdot (\partial_x \bar{X}^{x}_t\cdot l)d W_t\nonumber\\
				\quad\quad\quad\quad\quad\quad+\int_{\mathcal{Z}_1}\partial_x h_1(\bar{X}^{x}_t,z)\cdot (\partial_x \bar{X}^{x}_t\cdot l)\tilde{N}^1(dt,dz),\nonumber\\
				\partial_x \bar{X}^{x}_0\cdot l=l,\nonumber
			\end{array}\right.
		\end{equation*}
		and the other notations can be interpreted similarly.
		
		Note that $\bar{b}\in C^{4}_{p}(\RR^n,\RR^n)$ , $\bar{\sigma}\in C^4_b(\RR^n,\RR^n\times\RR^n)$, using a similar argument as in the proof of Proposition \ref{DFY} in the appendix, we can  prove the differentiability of $\bar{X}^{x}_t$ with respective to $x$ up to fourth derivative. Moreover, for any $T>0$, there exist $C_T, \tilde{k}>0$ such that
		\begin{eqnarray}
			&&\sup_{t\in [0,T]}\EE|\partial_x \bar{X}^{x}_t\cdot l_1|^{16}\leq C_T,\label{D barX}\\
			&&\sup_{t\in [0,T]}\EE|\partial^2_x \bar{X}^{x}_t\cdot (l_1,l_2)|^{4}\leq C_T(1+|x|^{\tilde{k}}),\label{Dxx barX}\\
			&&\sup_{t\in [0,T]}\EE|\partial^3_x \bar{X}^{x}_t\cdot (l_1,l_2,l_3)|^{2}\leq C_T(1+|x|^{\tilde{k}}),\label{Dxxx barX}\\
			&&\sup_{t\in [0,T]}\EE|\partial^4_x \bar{X}^{x}_t\cdot (l_1,l_2,l_3,l_4)|^2\leq C_T(1+|x|^{\tilde{k}}).\label{Dxxxx barX}
		\end{eqnarray}
		Using $\phi\in C^{4}_p(\RR^n)$ and applying H\"{o}lder inequality, \eref{UE3}-\eref{UE6} can be easily obtained from \eref{D barX}-\eref{Dxxxx barX}.
		
		Since the proofs of \eref{UE1} and \eref{UE2} follow the same argument, we only give the proof of \eref{UE1}. By It\^{o}'s formula and taking expectation, we have any unit vector $l\in\RR^n$,
		\begin{eqnarray*}
			\EE[D\phi(\bar{X}^{x}_t)\cdot (\partial_x \bar{X}^x_t\cdot l)]=\!\!\!\!\!\!\!\!&&D\phi(x)\cdot l+\int^t_0 \EE\Big[ D^2\phi(\bar{X}^{x}_s)\cdot \left (\partial_x \bar{X}^x_s\cdot l, \bar{b}(\bar{X}^x_s)\right)\Big]ds\\
			&&+ \int^t_0\EE\Big[D\phi(\bar{X}^{x}_s)\cdot\left(D\bar{b}(\bar{X}^{x}_s)\cdot (\partial_x \bar{X}^x_s\cdot l)\right)\Big] ds\\
			&&+\frac{1}{2}\int^t_0 \EE\text{Tr}\Big[D^3\phi(\bar{X}^{x}_s)\cdot (\partial_x \bar{X}^x_s\cdot l, (\bar{\sigma}\bar{\sigma})(\bar{X}^{x}_s)) \Big]ds\\
			&&+\int^t_0 \EE\text{Tr}\Big[D^2\phi(\bar{X}^{x}_s)\cdot\big (\bar{\sigma}(\bar{X}^{x}_s)\cdot[D\bar{\sigma}(\bar{X}^{x}_s)\cdot (\partial_x \bar{X}^x_s\cdot l)]^{\ast}\big) \Big]ds\\
			&&+\int^t_0\EE\int_{\mathcal{Z}_1}\Big[D\phi(\bar{X}^{x}_s+h_1(\bar{X}^{x}_s,z))\cdot [\partial_x \bar{X}^x_s\cdot l+\partial_xh_1(\bar{X}^{x}_s,z)\cdot (\partial_x \bar{X}^x_s\cdot l)]\\
			&&\quad\quad\quad\quad\quad-D\phi(\bar{X}^{x}_s)\cdot (\partial_x \bar{X}^x_s\cdot l)-D^2\phi(\bar{X}^x_s)\cdot (\partial_x \bar{X}^x_s\cdot l, h_1(\bar{X}^{x}_s,z))\\
			&&\quad\quad\quad\quad\quad-D\phi(\bar{X}^{x}_s)\cdot (\partial_x h_1(\bar{X}^{x}_s,z)\cdot  (\partial_x \bar{X}^x_s\cdot l) )\Big]\nu_1(dz)ds.
		\end{eqnarray*}
		Note that
		$$\partial_t(\partial_x u(t,x))\cdot l=\partial_t\EE[D\phi(\bar{X}^{x}_t)\cdot (\partial_x \bar{X}^x_t\cdot l)].$$
		By  $\phi\in C^{4}_p(\RR^n)$ and  \eref{D barX}-\eref{Dxxxx barX}, for any $t\leq T$, there exists $k_1>0$ such that
		\begin{eqnarray*}
			|\partial_t(\partial_x u(t,x))\cdot l|\leq C_T\left (1+|x|^{k_1}\right).
		\end{eqnarray*}
		Thus \eref{UE1} holds. \eref{UE2} can be prove by a similar argument, thus we omit its proof. The proof is complete.
	\end{proof}
	
	\vspace{0.2cm}
	Now we are in a position to prove Theorem \ref{main result 2}
	
	\vspace{0.1cm}
	
	\noindent
	Proof of Theorem \ref{main result 2}.
	For fixed $t>0$, let $\tilde{u}^t(s,x):=u(t-s,x)$, $s\in [0,t]$, by It\^{o}'s formula, we have
	\begin{eqnarray*}
		\tilde{u}^t(t, X^{\vare}_t)=\!\!\!\!\!\!\!\!&&\tilde{u}^t(0,x)+\int^t_0 \partial_s \tilde{u}^t(s, X^{\vare}_s )ds+\int^t_0 \mathscr{L}_{1}(Y^{\vare}_s)\tilde{u}^t(s, \cdot)(X^{\vare}_s) ds+\tilde{M}_t^{1}+\tilde{M}_t^{2},
	\end{eqnarray*}
	where
	\begin{eqnarray*}
		&&\mathscr{L}_{1}(y)\phi(x):=\langle b(x,y), D  \phi(x)\rangle+\frac{1}{2}\text{Tr}[\sigma\sigma^{*}(x,y)D^{2}\phi(x)]\\
		&&\quad\quad\quad\quad+\int_{\mathcal{Z}_1}\left[\phi(x+h_{1}(x,z))\!-\!\phi(x)\!-\!\langle D  \phi(x),h_{1}(x,z)\rangle\right] \nu_1(dz),
	\end{eqnarray*}
	$\tilde{M}_t^{1}$ and$\tilde{M}_t^{2}$  are $\mathscr{F}_{t}$-martingales which are defined as follows:
	\begin{eqnarray*}
		&&\tilde{M}_t^{1}:=\int^t_0 \langle \partial_{x}\tilde{u}^t(s,X_{s}^{\vare}),\sigma(X_{s}^{\vare}, Y_{s}^{\vare})dW_{s}^{1} \rangle	 ,  \\
		&&\tilde{M}_t^{2}:=\int^t_0 \int_{\mathcal{Z}_1}\tilde{u}^t(s,X_{s-}^{\vare}+h_{1}(X_{s-}^{\vare},z))-\tilde{u}^t(s,X_{s-}^{\vare})\tilde{N}^1(dz,ds).
	\end{eqnarray*}

	Note that $\tilde{u}^t(t, X^{\vare}_t)=\phi(X^{\vare}_t)$, $\tilde{u}^t(0, x)=\EE\phi(\bar{X}^{x}_t)$ and $\partial_s \tilde{u}^t(s, X^{\vare}_s )=-\bar{\mathscr{L}}_1 \tilde{u}^t(s,\cdot)(X^{\vare}_s)$, we have
	\begin{eqnarray}
		\left|\EE\phi(X^{\vare}_{t})-\EE\phi(\bar{X}_{t})\right| =\!\!\!\!\!\!\!\!&&|\EE\int^t_0 -\bar{\mathscr{L}}_1 \tilde{u}^t(s, \cdot)(X^{\vare}_s )ds+\EE\int^t_0 \mathscr{L}_{1}(Y^{\vare}_s)\tilde{u}^t(s, \cdot)(X^{\vare}_s)ds|\nonumber\\
		=\!\!\!\!\!\!\!\!&&\big|\EE\!\int^t_0 \!\langle b(X^{\vare}_s,Y^{\vare}_s)\!-\!\bar{b}(X^{\vare}_s),\partial_x \tilde{u}^t(s, X^{\vare}_s )\rangle\nonumber\\
		&&\quad\quad\quad\quad +\frac{1}{2}\text{Tr}\big[(\sigma\sigma^{\ast}(X^{\vare}_s,Y^{\vare}_s)\!-\!\bar{\sigma}\bar{\sigma}(X^{\vare}_s))\partial_x^{2}\tilde{u}^t(s, X^{\vare}_s)\big]  ds\big| .  \label{F5.11}
	\end{eqnarray}
	
	For any $s\in [0,t], x\in\RR^{n},y\in\RR^{m}$, define
	\begin{eqnarray*}
		F^t(s,x,y):=\langle b(x,y), \partial_x \tilde{u}^t(s, x)\rangle+\frac{1}{2}\text{Tr}\big[\sigma\sigma^{*}(x,y)\partial_x^{2}\tilde{u}^t(s,x)\big].
	\end{eqnarray*}
	Thus it is easy to see
	\begin{eqnarray*}
		\bar{F}^t(s,x)\!\!\!\!\!\!\!&&:=\int_{\RR^{m}} F^t(s,x,y)\mu^x(dy)  \\
		&&=\langle \bar{b}(x), \partial_x \tilde{u}^t(s, x)\rangle +\frac{1}{2}\text{Tr}\big[\bar{\sigma}\bar{\sigma}(x)\partial_x^{2}\tilde{u}^t(s, x)\big].
	\end{eqnarray*}
	
	By Lemma \ref{Lemma 5.1} and $b\in C^{2,3}_p(\RR^n\times\RR^m,\RR^n)$, we can obtain that $\partial^j_y\partial^i_xF^t(s,x,y)$ and $\partial_sF^t(s,x,y)$ exist and are polynomial growth, for any $i=0,1,2$, $j=0,1,2,3$ with $0\leq i+j\leq 3$. Then following the argument as in the proof of Proposition \ref{P4.1}, we have
	$$
	\tilde{\Phi}^t(s, x,y):=\int^{\infty}_0\left[\EE F^t(s,x,Y^{x,y}_r)-\bar{F}^t(s,x)\right] dr,
	$$
	is a solution of the following Poisson equation:
	\begin{eqnarray}
		-\mathscr{L}_{2}(x)\tilde{\Phi}^t(s,x,\cdot)(y)=F^t(s,x,y)-\bar{F}^t(s,x),\quad s\in [0,t].\label{WPE}
	\end{eqnarray}
	Moreover, for any $T>0$, $t\in [0,T]$, there exist $C_T, k_2>0$ such that the following estimates hold:
	\begin{eqnarray}
		&&\sup_{s\in [0, t]}|\tilde{\Phi}^t(s,x,y)|\leq C_T(1+|x|^{k_2}+|y|^{k_2}), \label{WR1} \\
		&&\sup_{s\in [0, t]}|\partial_s \tilde{\Phi}^t(s,x,y)|\leq C_T(1+|x|^{k_2}+|y|^{k_2}), \label{WR2} \\
		&&\sup_{s\in [0, t]}|\partial_x \tilde{\Phi}^t(s,x,y)|\leq C_{T}(1+|x|^{k_2}+|y|^{k_2}), \label{WR3}\\
		&&\sup_{s\in[0,t]}\| \partial_x^{2}\tilde{\Phi}^t(s,x, y)\|\leq C_{T}(1+|x|^{k_2}+|y|^{k_2}). \label{WR4}
	\end{eqnarray}
	
	Using It\^o's formula and taking expectation on both sides, we get
	\begin{eqnarray*}
		\EE\tilde{\Phi}^t(t, X_{t}^{\vare},Y^{\vare}_{t})\!\!\!\!\!\!\!\!&&=\tilde \Phi^t(0, x,y)+\EE\int^t_0 \partial_s \tilde{\Phi}^t(s, X_{s}^{\vare},Y^{\vare}_{s})ds\\
		&&	+\EE\int^t_0\mathscr{L}_{1}(Y^{\vare}_{s})\tilde\Phi^t(s, \cdot,Y^{\vare}_{s})(X_{s}^{\vare})ds
		+\frac{1}{\vare}\EE\int^t_0 \mathscr{L}_{2}(X_{s}^{\vare})\tilde{\Phi}^t(s, X_{s}^{\vare},\cdot)(Y^{\vare}_{s})ds,
	\end{eqnarray*}
	which implies
	\begin{eqnarray}
		&&-\EE\int^t_0 \mathscr{L}_{2}(X_{s}^{\vare})\tilde\Phi^t(s, X_{s}^{\vare},\cdot)(Y^{\vare}_{s})ds \nonumber \\
		=\!\!\!\!\!\!\!\!&&\vare\big[\tilde{\Phi}^t(0, x,y)-\EE\tilde{\Phi}^t(t, X_{t}^{\vare},Y^{\vare}_{t})+\EE\int^t_0 \partial_s \tilde{\Phi}^t(s, X_{s}^{\vare},\cdot)(Y^{\vare}_{s})ds\nonumber\\
		&&+\EE\int^t_0\mathscr{L}_{1}(Y^{\vare}_{s})\tilde{\Phi}^t(s, \cdot,Y^{\vare}_{s})(X_{s}^{\vare})ds\big].\label{F3.39}
	\end{eqnarray}
	
	Combining  \eref{F5.11}, \eref{WPE} and \eref{F3.39}, we get
	\begin{eqnarray*}
		\sup_{0\leq t\leq T}\left|\EE\phi(X^{\vare}_{t})-\EE\phi(\bar{X}_{t})\right|
=\!\!\!\!\!\!\!\!&&\sup_{0\leq t\leq T}\left|\EE\int^t_0 \left[F^t(s,X_{s}^{\vare},Y^{\vare}_{s})- \bar{F}^t(s,X_{s}^{\vare})\right]ds\right|\\
=\!\!\!\!\!\!\!\!&&\sup_{0\leq t\leq T}\left|\EE\int^t_0 \mathscr{L}_{2}(X_{s}^{\vare})\tilde{\Phi}^t(s, X_{s}^{\vare},\cdot)(Y^{\vare}_{s})ds\right|\\
		\leq\!\!\!\!\!\!\!\!&&\vare\Big[\sup_{t\in [0,T]}|\tilde{\Phi}^t(0, x,y)|+\sup_{0\leq t\leq T}\left|\EE\tilde{\Phi}^t(t, X_{t}^{\vare},Y^{\vare}_{t})\right|\nonumber\\
		&&+\sup_{t\in [0,T]}\EE\int^t_0\left|\partial_s \tilde{\Phi}^t(s, X_{s}^{\vare},Y^{\vare}_{s})\right|ds.\\
		&& +\sup_{t\in [0,T]}\EE\int^t_0\left|\mathscr{L}_{1}(Y^{\vare}_{s})\tilde{\Phi}^t(s, \cdot,Y^{\vare}_{s})(X_{s}^{\vare})\right|ds\Big].
	\end{eqnarray*}
	Finally, using \eref{WR1}-\eref{WR4}, we obtain for some $k_2>0$,
	\begin{eqnarray*}
		\sup_{0\leq t\leq T}\left|\EE\phi(X^{\vare}_{t})-\EE\phi(\bar{X}_{t})\right|\leq C_{T}\left(1+|x|^{k_2}+|y|^{k_2}\right)\vare.
	\end{eqnarray*}
	The proof is complete.

\section{Appendix}	

\subsection{Differentiability of $Y^{x,y}$} In this subsection, we give the proofs of the differentiability of the solution $Y^{x,y}_t$ of the frozen equation with respect to parameter $x$ and $y$.
\begin{proposition}\label{DFY}
Under the assumptions \ref{B1}-\ref{B3}. Then $Y^{x,y}_t$ is differentiable with respect to $y$ and $x$ in  directions $l\in\RR^m$ and $l_1\in \RR^n$ respectively, which satisfies
		\begin{eqnarray}
			d[\partial _yY_{t}^{x,y}\cdot l]=\!\!\!\!\!\!\!\!\!&&\partial_yf(x,Y_{t}^{x,y})\cdot(\partial_yY_{t}^{x,y}\cdot l)dt+\partial_yg(x,Y_{t}^{x,y})\cdot (\partial_yY_{t}^{x,y}\cdot l)d\tilde{W}_{t}^{2} \nonumber\\
			&&+\int_{\mathcal{Z}_2}\partial_yh_2(x,Y_{t-}^{x,y},z) \cdot (\partial_yY_{t-}^{x,y}\cdot l)\tilde{N}^{2}(dz,dt).\label{Apartial y}
		\end{eqnarray}
and
\begin{eqnarray}
			d\left[\partial _xY_{t}^{x,y}\cdot l_1\right]=\!\!\!\!\!\!\!\!\!&&\left[\partial_x f(x,Y_{t}^{x,y})\cdot l_1+\partial_yf(x,Y_{t}^{x,y})\cdot(\partial_xY_{t}^{x,y}\cdot l_1)\right]dt\nonumber\\
			&&+ \left[\partial_x g(x,Y_{t}^{x,y})\cdot l_1+\partial_yg(x,Y_{t}^{x,y})\cdot( \partial_xY_{t}^{x,y}\cdot l_1)\right]d \tilde{W}_{t}^{2} \nonumber\\
			&&+\int_{\mathcal{Z}_2}\left[\partial_x h_2(x,Y_{t-}^{x,y},z)\cdot l_1+\partial_y h_2(x,Y_{t-}^{x,y},z) \cdot (\partial_x Y_{t-}^{x,y}\cdot l_1)\right]\tilde{N}^{2}(dt,dz).\label{Apartial x}
		\end{eqnarray}
\end{proposition}
Moreover, there exist $C,\gamma>0$ such that
\begin{eqnarray}
\EE|\partial _yY_{t}^{x,y}\cdot l|^{\ell}\leq e^{-\gamma t}|l|^{\ell},\quad \sup_{t\geq 0}\mathbb{E}|\partial_x Y^{x,y}_t\cdot l_1|^{\ell} \leq C|l_1|^{\ell}, \label{EPXY}
\end{eqnarray}
where $\ell$ is the constant in assumption \ref{B1}.
\begin{proof}
We only prove $Y^{x,y}_t$ is differentiable with respect to $y$ in direction $l\in\RR^m$ and its directional derivative $\partial_yY^{x,y}_t\cdot l$ satisfies equation \eref{Apartial y}. Since $Y^{x,y}_t$ is differentiable with respect to $x$ in direction $l_1\in\RR^n$ and its directional derivative $\partial_xY^{x,y}_t\cdot l_1$ satisfies equation \eref{Apartial x} can be proved by a similar argument, thus we omit the details.

In fact, it is sufficient to prove the following result:
\begin{eqnarray}
\lim_{\delta\rightarrow 0}\mathbb{E}\left|\frac{Y^{x,y+\delta l}_t-Y^{x,y}_t}{\delta}-\partial_y Y^{x,y}_t\cdot l\right|^2=0.\label{Partialy Y}
\end{eqnarray}
To do this, denote $Z^{\delta,l}_t:=\frac{Y^{x,y+\delta l}_t-Y^{x,y}_t}{\delta}-\partial _yY_{t}^{x,y}\cdot l$, recall
	\begin{equation*}\left\{\begin{array}{l} \label{AFE}
			\displaystyle
			dY_{t}^{x,y}=f(x,Y_{t}^{x,y})dt+g(x,Y_{t}^{x,y})d \tilde{W}_{t}^{2}+\int_{\mathcal{Z}_2}h_{2}(x,Y_{t-}^{x,y},z)\tilde{N}^{2}(dz,dt),\nonumber\\
			Y_{0}^{x,y}=y,\nonumber\\
		\end{array}\right.
	\end{equation*}
then $Z^{\delta,l}_t$ satisfies the following equation:
	\begin{equation}\left\{\begin{array}{l} \label{AFE}
			\displaystyle
			dZ^{\delta,l}_t=\left[\frac{f(x,Y_{t}^{x,y+\delta l})-f(x,Y_{t}^{x,y})}{\delta}-\partial_yf(x,Y_{t}^{x,y})\cdot(\partial_yY_{t}^{x,y}\cdot l)\right]dt\nonumber\\
\quad\quad\quad+\left[\frac{g(x,Y_{t}^{x,y+\delta l})-g(x,Y_{t}^{x,y})}{\delta}-\partial_yg(x,Y_{t}^{x,y})\cdot(\partial_yY_{t}^{x,y}\cdot l)\right]d\tilde{W}_{t}^{2}\nonumber\\
\quad\quad\quad+\int_{\mathcal{Z}_2}\left[\frac{h_2(x,Y_{t-}^{x,y+\delta l},z)-h_2(x,Y_{t-}^{x,y},z)}{\delta}-\partial_yh_2(x,Y_{t-}^{x,y},z)\cdot(\partial_yY_{t-}^{x,y}\cdot l)\right]\tilde{N}^{2}(dz,dt),\nonumber\\
			Z^{\delta,l}_0=0\nonumber\\
		\end{array}\right.
	\end{equation}
By It\^{o}'s formula and taking expectation, we have
	\begin{eqnarray}
			\frac{d}{dt}\EE|Z^{\delta,l}_t|^2=\!\!\!\!\!\!&&2\EE\left\langle \left[\frac{f(x,Y_{t}^{x,y+\delta l})-f(x,Y_{t}^{x,y})}{\delta}-\partial_yf(x,Y_{t}^{x,y})\cdot(\partial_yY_{t}^{x,y}\cdot l)\right], Z^{\delta,l}_t\right\rangle \nonumber\\
&&+\EE\left\|\frac{g(x,Y_{t}^{x,y+\delta l})-g(x,Y_{t}^{x,y})}{\delta}-\partial_yg(x,Y_{t}^{x,y})\cdot(\partial_yY_{t}^{x,y}\cdot l)\right\|^2\nonumber\\
&&+\EE\int_{\mathcal{Z}_2}\left|\frac{h_2(x,Y_{t}^{x,y+\delta l},z)-h_2(x,Y_{t}^{x,y},z)}{\delta}-\partial_yh_2(x,Y_{t}^{x,y},z)\cdot(\partial_yY_{t}^{x,y}\cdot l)\right|^2\nu_2(dz)\nonumber\\
=:\!\!\!\!\!\!&&\sum^{3}_{i=1}\tilde{I}_i(t). \label{I_1-I_3}
	\end{eqnarray}
For the term $\tilde{I}_1(t)$. By Taylor's formula, there exists $\xi\in(0,1)$ such that
\begin{eqnarray}
\tilde{I}_1(t)=\!\!\!\!\!\!&&2\EE\left\langle \left[\frac{f(x,Y_{t}^{x,y+\delta l})-f(x,Y_{t}^{x,y})}{\delta}-\partial_yf(x,Y_{t}^{x,y})\cdot\frac{Y^{x,y+\delta l}_t-Y^{x,y}_t}{\delta}\right], Z^{\delta,l}_t\right\rangle \nonumber\\
&&+2\EE\left\langle \partial_yf(x,Y_{t}^{x,y})\cdot Z^{\delta,l}_t, Z^{\delta,l}_t\right\rangle\nonumber\\
\leq \!\!\!\!\!\!&&\EE\delta^{-1} \|\partial^2_yf(x,\xi Y_{t}^{x,y+\delta l}\!\!+(1-\xi)Y_{t}^{x,y})\| |Y_{t}^{x,y+\delta l}-Y_{t}^{x,y}|^2| Z^{\delta,l}_t| \nonumber\\
&&+2\EE\left\langle \partial_yf(x,Y_{t}^{x,y})\cdot Z^{\delta,l}_t, Z^{\delta,l}_t\right\rangle.\label{I_1}
	\end{eqnarray}

For the terms $\tilde{I}_2(t)$ and $\tilde{I}_3(t)$. By a similar argument above, there exist $\xi_2,\xi_3\in(0,1)$ such that
\begin{eqnarray}
\tilde{I}_2(t)\leq\!\!\!\!\!\!&&2\EE\left\|\frac{g(x,Y_{t}^{x,y+\delta l})-g(x,Y_{t}^{x,y})}{\delta}-\partial_yg(x,Y_{t}^{x,y})\cdot \frac{Y^{x,y+\delta l}_t-Y^{x,y}_t}{\delta}\right\|^2\nonumber\\
&&+2\EE\left\|\partial_yg(x,Y_{t}^{x,y})\cdot Z^{\delta,l}_t\right\|^2\nonumber\\
\leq\!\!\!\!\!\!&&2^{-1}\EE\delta^{-2} \|\partial^2_yg(x,\xi_2 Y_{t}^{x,y+\delta l}\!\!+(1-\xi_2)Y_{t}^{x,y})\|^2 |Y_{t}^{x,y+\delta l}-Y_{t}^{x,y}|^4\nonumber\\
&&+2\EE\left\|\partial_yg(x,Y_{t}^{x,y})\cdot Z^{\delta,l}_t\right\|^2\label{I_2}
	\end{eqnarray}
and
\begin{eqnarray}
\tilde{I}_3(t)\leq\!\!\!\!\!\!&&2\EE\int_{\mathcal{Z}_2}\left|\frac{h_2(x,Y_{t}^{x,y+\delta l},z)-h_2(x,Y_{t}^{x,y},z)}{\delta}-\partial_yh_2(x,Y_{t}^{x,y},z)\cdot\frac{Y^{x,y+\delta l}_t-Y^{x,y}_t}{\delta}\right|^2\nu_2(dz)\nonumber\\
&&+2\EE\int_{\mathcal{Z}_2}\left|\partial_yh_2(x,Y_{t}^{x,y},z)\cdot Z^{\delta,l}_t\right|^2\nu_2(dz)\nonumber\\
\leq\!\!\!\!\!\!&&2^{-1}\EE\int_{\mathcal{Z}_2}\delta^{-2} \|\partial^2_yh_2(x,\xi_3 Y_{t}^{x,y+\delta l}\!\!+(1-\xi_3)Y_{t}^{x,y},z)\|^2 |Y_{t}^{x,y+\delta l}-Y_{t}^{x,y}|^4\nu_2(dz)\nonumber\\
&&+2\EE\int_{\mathcal{Z}_2}\left|\partial_yh_2(x,Y_{t}^{x,y},z)\cdot Z^{\delta,l}_t\right|^2\nu_2(dz).\label{I_3}
	\end{eqnarray}

Combining (\ref{I_1-I_3})-(\ref{I_3}), then by Young's inequality, assumption \ref{B3}, \eref{FroY} and Lemma \ref{L3.3}, there exists $k>0$ such that
\begin{eqnarray*}
\frac{d}{dt}\EE|Z^{\delta,l}_t|^2\leq \!\!\!\!\!\!&&\frac{\beta}{2}\EE|Z^{\delta,l}_t|^2+C(1+|x|^k+|y|^k)\delta^2\\
&&+2\EE\left\langle \partial_yf(x,Y_{t}^{x,y})\cdot Z^{\delta,l}_t, Z^{\delta,l}_t\right\rangle+2\EE\left\|\partial_yg(x,Y_{t}^{x,y})\cdot Z^{\delta,l}_t\right\|^2\\
&&+2\EE\int_{\mathcal{Z}_2}\left|\partial_yh_2(x,Y_{t}^{x,y},z)\cdot Z^{\delta,l}_t\right|^2\nu_2(dz).
	\end{eqnarray*}
Note that condition \eref{ConB1} implies that for any $x\in\RR^n,y\in\RR^m, l\in\RR^m$
\begin{eqnarray}
&&2\left \langle \partial_yf(x,y)\cdot l,l\right\rangle +(\ell-1)\|\partial_yg(x,y)\cdot l\|^{2}\nonumber\\
&&\quad +2^{\ell-3}(\ell-1)\!\int_{\mathcal{Z}_2}\!\!|\partial_yh_{2}(x,y,z)\cdot l|^{2}\nu_2(dz)\le -\beta|l|^{2},
	\label{ConDP}
\end{eqnarray}
where $\ell>8$, thus this implies that
\begin{eqnarray*}
&&2\left \langle \partial_yf(x,y)\cdot l,l\right\rangle +2\|\partial_yg(x,y)\cdot l\|^{2}\nonumber\\
&&\quad +2\int_{\mathcal{Z}_2}\!\!|\partial_yh_{2}(x,y,z)\cdot l|^{2}\nu_2(dz)\le -\beta|l|^{2}.
\end{eqnarray*}

Then we have
\begin{eqnarray*}
\frac{d}{dt}\EE|Z^{\delta,l}_t|^2\leq \!\!\!\!\!\!&&-\frac{\beta}{2}\EE|Z^{\delta,l}_t|^2+C(1+|x|^k+|y|^k)\delta^2.
	\end{eqnarray*}
By comparison theorem, we have
\begin{eqnarray*}
\EE|Z^{\delta,l}_t|^2\leq \!\!\!\!\!\!&&C_{\beta}(1+|x|^k+|y|^k)\delta^2\rightarrow 0,\quad \text{as} \quad \delta\rightarrow 0.
	\end{eqnarray*}

Using It\^{o}'s formula on $|\partial _yY_{t}^{x,y}\cdot l|^{\ell}$ and taking expectation on both sides, then by \eref{Taylor} again, we have
\begin{eqnarray*}
			\frac{d}{dt}\mathbb{E}|\partial _yY_{t}^{x,y}\cdot l|^{\ell}\leq\!\!\!\!\!\!\!\!\!\!&&{\ell}\mathbb{E}\left[\big|\partial _yY_{t}^{x,y}\cdot l\big|^{\ell-2}\left\langle \partial_yf(x,Y_{t}^{x,y})\cdot(\partial_yY_{t}^{x,y}\cdot l), \partial _yY_{t}^{x,y}\cdot l\right\rangle\right]\\
			&& +\frac{{\ell}({\ell}-1)}{2}\mathbb{E}\left[\big|\partial _yY_{t}^{x,y}\cdot l\big|^{\ell-2}\left \|\partial_yg(x,Y_{t}^{x,y})\cdot (\partial_yY_{t}^{x,y}\cdot l)\right\|^2\right]\\
			&& +2^{\ell-4}\ell(\ell-1)\mathbb{E}\Big[|\partial _yY_{t}^{x,y}\cdot l|^{\ell-2} \int_{\mathcal{Z}_2}|\partial_yh_2(x,Y_{t-}^{x,y},z) \cdot (\partial_yY_{t-}^{x,y}\cdot l)|^2\nu_2(dz)\Big]\\
			&&+2^{\ell-4}\ell(\ell-1)\mathbb{E} \int_{\mathcal{Z}_2}|\partial_yh_2(x,Y_{t-}^{x,y},z) \cdot (\partial_yY_{t-}^{x,y}\cdot l)|^{\ell}\nu_2(dz).
		\end{eqnarray*}
Note that condition \eref{ConB11} implies that for any $x\in\RR^n,y\in\RR^m, l\in\RR^m$
\begin{eqnarray}
2^{\ell-3}(\ell-1)\!\int_{\mathcal{Z}_2}\!\!|\partial_yh_{2}(x,y,z)\cdot l|^{\ell}\nu_2(dz)\le L_{h_2}|l|^{\ell}.
	\label{ConDP2}
\end{eqnarray}
By \eref{ConDP} and \eref{ConDP2}, we have
\begin{eqnarray*}
			\frac{d}{dt}\mathbb{E}|\partial _yY_{t}^{x,y}\cdot l|^{\ell}
			\leq\!\!\!\!\!\!\!\!\!\!&&-\frac{\ell\beta}{2}\mathbb{E}\left[\big|\partial _yY_{t}^{x,y}\cdot l\big|^{\ell}\right]+\frac{\ell L_{h_2}}{2}\EE\big|\partial _yY_{t}^{x,y}\cdot l\big|^{\ell}\\
			=\!\!\!\!\!\!\!\!\!\!&&-\frac{\ell(\beta-L_{h_2})}{2}\mathbb{E}\big|\partial _yY_{t}^{x,y}\cdot l\big|^{\ell}.
		\end{eqnarray*}
Then by the comparison theorem, we get
\begin{eqnarray*}
\EE|\partial _yY_{t}^{x,y}\cdot l|^{\ell}\leq e^{-\frac{\ell(\beta-L_{h_2})t}{2}}|l|^{\ell}.
\end{eqnarray*}
Thus the first estimate in \eref{EPXY} holds. By a similar argument, we can prove the second estimate in \eref{EPXY}. The proof is complete.
\end{proof}
\begin{remark}\label{R6.2}
It is worth noting that by a similar argument to that above, using the additional regularity assumptions \eref{ConB31} in \ref{B3} on the coefficients, we can further prove the differentiability of $\partial _yY_{t}^{x,y}\cdot l$ and $\partial _xY_{t}^{x,y}\cdot l_1$ with respect to parameters. Let $\partial_y\partial_x Y^{x,y}_t\cdot (l_1,l_2)$ be the directional derivative of $\partial_x Y^{x,y}_t\cdot l_1$ with respect to $y$ in the direction $l_2$. Let $\partial^2_x Y^{x,y}_t\cdot (l_1,l_2)$ be the directional derivative of $\partial_x Y^{x,y}_t\cdot l_1$ with respect to $x$ in the direction $l_2$. Let $\partial_y\partial_{x}^{2}Y_{t}^{x,y}\cdot (l_1,l_2,l_3)$ be the directional derivative of $\partial^2_x Y^{x,y}_t\cdot (l_1,l_2)$ with respect to $y$ in the direction $l_3$. We can easily prove for any unit vectors $l_1,l_2,l_3$,
\begin{eqnarray*}
&&\EE|\partial_y\partial_{x}Y^{x,y}_t\cdot (l_1,l_2)|^4\leq Ce^{-4\gamma t}(1+|y|^{4k}),\\
&&\sup_{t\geq 0}\EE|\partial^2_{x}Y^{x,y}_t\cdot (l_1,l_2)|^4\leq C(1+|y|^{4k}),\\
&&\EE|\partial_y\partial^2_{x}Y^{x,y}_t\cdot (l_1,l_2,l_3)|^2\leq Ce^{-2\gamma t}(1+|y|^{4k}),
\end{eqnarray*}
where $C,k,\gamma>0$.
\end{remark}

\subsection{Well-posedness of equation \eref{KE}} In this subsection, we give the detailed proof of the existence and uniqueness of equation \eref{KE}.
\begin{proposition}\label{EUE}
Under the assumptions \ref{A1}-\ref{A3} and \ref{B1}-\ref{B3}. For any $\phi\in C^{2}_p(\RR^{n})$, the following Kolmogorov equation
\begin{equation}\left\{\begin{array}{l}\label{AKE}
			\displaystyle
			\partial_t u(t,x)=\bar{\mathscr{L}}_1 u(t,x),\quad t\geq 0,x\in\RR^n, \\
			u(0, x)=\phi(x),
		\end{array}\right.
	\end{equation}
admits a unique solution $u\in C^{1,2}(\RR_{+}\times \RR^n)$, moreover the solution $u$ is given by
	$$
	u(t,x)=\EE\phi(\bar{X}^x_t),
	$$
where $\bar{\mathscr{L}}_1$ is the infinitesimal generator of the transition semigroup of the averaged equation \eref{AR2}, which is given by
	\begin{eqnarray*}
		&&\bar{\mathscr{L}}_1\phi(x):=\langle \bar{b}(x), D  \phi(x)\rangle+\frac{1}{2}\text{Tr}[\bar{\sigma}\bar{\sigma}(x)D^{2}\phi(x)]\\
		&&\quad\quad\quad\quad\quad\quad+\int_{\mathcal{Z}_1}\left[\phi(x+h_{1}(x,z))\!-\!\phi(x)\!-\!\langle D  \phi(x),h_{1}(x,z)\rangle \right]\nu_1(dz).
	\end{eqnarray*}
\end{proposition}
\begin{proof}
\textbf{Existence:} Using It\^{o}'s formula, it is easy to see $u(t,x)=\EE\phi(\bar{X}^x_t)$ is differentiable with respect to $t$. Moreover, using the chain rule and $\phi\in C^{2}_p(\RR^{n})$, it is easy to see $u(t,x)$ is first and second differentiable with respect to $x$. Hence $u\in C^{1,2}(\RR_{+}\times \RR^n)$. In order to prove $u(t,x)$ solves equation \eref{AKE}, we use the definition of generator $\bar{\mathscr{L}}_1$, more precisely, by the Markov property and homogeneous property, we have for any $s>0$,
\begin{eqnarray*}
\frac{\EE u(t,\bar{X}^x_s)-u(t,x)}{s}=\!\!\!\!\!\!&&\frac{\EE\left[\EE\phi(\bar{X}^{y}_{t})|_{y=\bar{X}^x_s}\right]-\EE\phi(\bar{X}^x_t)}{s}\\
=\!\!\!\!\!\!&&\frac{\EE\left[\EE\phi(\bar{X}^x_{t+s})| \mathscr{F}_{s}\right]-\EE\phi(\bar{X}^x_t)}{s}\\
=\!\!\!\!\!\!&&\frac{\EE\phi(\bar{X}^x_{t+s})-\EE\phi(\bar{X}^x_t)}{s}\\
=\!\!\!\!\!\!&&\frac{u(t+s,x)-u(t,x)}{s}.
\end{eqnarray*}
Then letting $s\rightarrow 0$, we get \eref{AKE}.

\textbf{Uniqueness:} Let $w(t,x)\in C^{1,2}(\RR_{+}\times \RR^n)$ be another solution of  \eref{AKE} with $w(0,x)=\phi(x)$. For any fixed $t>0$, define
$$\tilde{w}(s,x)=w(t-s,x), s\in[ 0,t],$$
then it is easy to check
$$
\partial_s \tilde{w}(s,x)+\bar{\mathscr{L}}_1 \tilde{w}(s,x)=0,\quad \forall s>0, x\in \RR^n.
$$
Then using It\^{o}'s formula on $\tilde{w}(t, \bar{X}^{x}_t)$ and taking expectation, we have
\begin{eqnarray*}
\EE\tilde{w}(t, \bar{X}^{x}_t) =\!\!\!\!\!\!&&\tilde{w}(0, x)+\int^t_0 \left[\partial_s \tilde{w}(s,\bar{X}^{x}_s)+\bar{\mathscr{L}}_1 \tilde{w}(s,\bar{X}^{x}_s)\right]ds=\tilde{w}(0, x).
\end{eqnarray*}
Note that by the definition of $\tilde{w}$, it follows
$$\EE\tilde{w}(t, \bar{X}^{x}_t)=\EE\phi(\bar{X}^{x}_t),\quad\tilde{w}(0, x)=w(t,x).$$
Hence, we obtain $w(t,x)=u(t,x)$.
\end{proof}

\vspace{0.3cm}
\textbf{Acknowledgment}. The authors would like to thank the anonymous referee for their very careful reading of the manuscript and especially for their very valuable suggestions and comments on
improving the manuscript.
	
	\vspace{0.3cm}
	\textbf{Funding} This work is supported by the National Natural Science Foundation of China (Nos, 12271219, 11931004, 12090010, 12090011), the QingLan Project of Jiangsu Province and the Priority Academic Program Development of Jiangsu Higher Education Institutions.
	
	\vspace{0.3cm}
	\textbf{Data Availability}
	Data sharing not applicable to this article as no data sets were generated or analysed during
	the current study.

	\vspace{0.3cm}
	\textbf{Declarations}
	
	\vspace{0.3cm}
	\textbf{Conflict of Interests}
	The authors declare that they have no conflict of interests.

\end{document}